\documentclass{scrartcl}

\usepackage[kerning, tracking, spacing]{microtype}
\microtypecontext{spacing=nonfrench}
\usepackage{mathtools}
\usepackage{amsmath,amsthm, amssymb, commath}

\usepackage{url}
\usepackage{hyperref}

\usepackage{lmodern}
\usepackage[T1]{fontenc}
\usepackage{dutchcal}
\DeclareSymbolFontAlphabet{\mathcalorig}{symbols}

\usepackage{tikz}
\usepackage{tkz-graph}
\usetikzlibrary{matrix}
 \usetikzlibrary{shapes}
\usetikzlibrary{arrows,decorations.markings}
\usepackage{tikz-cd}
\usepackage{enumerate}

\usepackage{float}

\theoremstyle{definition}
\newtheorem{theorem}{Theorem}[section] 
\newtheorem{proposition}[theorem]{Proposition}  
\newtheorem{example}[theorem]{Example}  
\newtheorem{definition}[theorem]{Definition}  
\newtheorem{prgrph}[theorem]{} 
\newtheorem{lemma}[theorem]{Lemma} 
\newtheorem{remark}[theorem]{Remark} 
\newtheorem{notation}[theorem]{Notation} 
\newtheorem{corollary}[theorem]{Corollary} 
\newtheorem{warning}[theorem]{Warning} 
\newtheorem{application}[theorem]{Application} 

\newtheorem{globtheorem}{Theorem} 

\newcommand{\thistheoremname}{}
\newtheorem{genericthm}[theorem]{\thistheoremname}

\pgfarrowsdeclare{bad to}{bad to}
{
  \pgfarrowsleftextend{-2\pgflinewidth}
  \pgfarrowsrightextend{\pgflinewidth}
}
{
  \pgfsetlinewidth{0.8\pgflinewidth}
  \pgfsetdash{}{0pt}
  \pgfsetroundcap
  \pgfsetroundjoin
  \pgfpathmoveto{\pgfpoint{-3\pgflinewidth}{4\pgflinewidth}}
  \pgfpathcurveto
  {\pgfpoint{-2.75\pgflinewidth}{2.5\pgflinewidth}}
  {\pgfpoint{0pt}{0.25\pgflinewidth}}
  {\pgfpoint{0.75\pgflinewidth}{0pt}}
  \pgfpathcurveto
  {\pgfpoint{0pt}{-0.25\pgflinewidth}}
  {\pgfpoint{-2.75\pgflinewidth}{-2.5\pgflinewidth}}
  {\pgfpoint{-3\pgflinewidth}{-4\pgflinewidth}}
  \pgfusepathqstroke
}

\newcommand{\Z}{\mathbb{Z}}
\newcommand{\Q}{\mathbb{Q}}
\newcommand{\C}{\mathbb{C}}
\newcommand{\M}{\overline{\mathcalorig{M}}}
\newcommand{\F}{\mathcalorig{F}}
\newcommand{\X}{\mathcalorig{X}}
\newcommand{\id}{\textup{id}}
\newcommand{\B}{\overline{\mathcal{B}}}
\renewcommand{\H}{\overline{\mathcalorig{H}}}
\newcommand{\Hyp}{\overline{\mathcal{H}}}

\DeclareMathOperator{\codim}{codim}
\DeclareMathOperator{\im}{Im}
\DeclareMathOperator{\aut}{Aut}
\DeclareMathOperator{\spec}{Spec}

\newcommand{\Ho}{\mathcalorig{H}}

\renewcommand{\L}{\mathbb{L}}
\newcommand{\Base}{S}

\renewcommand{\S}{\mathfrak{S}}

\DeclareMathOperator{\Sym}{Sym}

\newcommand{\fundgp}{\Xi}

\newcommand{\hypbiram}{\ensuremath{{n_1}}}
\newcommand{\hypbipair}{\ensuremath{{n_2}}}

\title{Intersections of loci of admissible covers with tautological classes}

 \author{Johannes Schmitt}
 \author{Jason van Zelm}

\author{Johannes Schmitt and Jason van Zelm}
\date{}

\begin{document}

\maketitle

\begin{abstract}
    For a finite group $G$, let $\H_{g,G,\xi}$ be the stack of admissible $G$-covers  $C\to D$ of stable curves with ramification data $\xi$, $g(C)=g$ and $g(D)=g'$. There are source and target morphisms $\phi\colon \H_{g,G,\xi}\to \M_{g,r}$ and $\delta\colon \H_{g,G,\xi}\to \M_{g',b}$, remembering the curves $C$ and $D$ together with the ramification or branch points of the cover respectively. 
    In this paper we study admissible cover cycles, i.e. cycles of the form $\phi_* [\H_{g,G,\xi}]$. Examples include the fundamental classes of the loci of hyperelliptic or bielliptic curves $C$ with marked ramification points.
    
    The two main results of this paper are as follows: Firstly, for the gluing morphism  $\xi_A\colon \M_A\to \M_{g,r}$ associated to to a stable graph $A$ we give a combinatorial formula for the pullback $\xi^*_A \phi_*[\H_{g,G,\xi}]$ in terms of spaces of admissible $G$-covers and $\psi$ classes. This allows us to describe the intersection of the cycles $\phi_*[\H_{g,G,\xi}]$ with tautological classes.
    Secondly, the pull-push $\delta_*\phi^*$ sends tautological classes to tautological classes and we also give a combinatorial description of this map in terms of standard generators of the tautological rings.
    
     We show how to use the pullbacks to algorithmically compute tautological expressions  for cycles of the form $\phi_* [\H_{g,G,\xi}]$. In particular, we compute the classes $[\Hyp_5]$ and $[\Hyp_6]$ of the hyperelliptic loci in $\M_5$ and $\M_6$ and the class $[\B_4]$ of the bielliptic locus in $\M_4$.
    \end{abstract}


\tableofcontents

\section{Introduction}
\subsection{Motivation: Admissible Cover Cycles}

A classical way of constructing closed subvarieties of the moduli space $\M_{g,n}$ of stable curves is by considering families of finite covers of curves; for example we can consider the closure of the locus of smooth curves $(C,p_1,...,p_n)$ such that there exists a finite degree $d$ cover $C\rightarrow D$ to a curve of fixed genus $g'$ with fixed ramification profile at the marked points $p_i$ of $C$. An example would be the loci $\Hyp_{g,\hypbiram,2\hypbipair}$  and $\B_{g,\hypbiram,2\hypbipair}$ of hyperelliptic and bielliptic curves with $\hypbiram$ marked fixed points of the corresponding involution and $2\hypbipair$ marked points pairwise switched.

Given such a locus of  covers it is natural to study its fundamental class in the Chow or cohomology ring\footnote{In this paper we restrict to the case of rational coefficients, both for Chow groups and for cohomology groups.} of $\M_{g,n}$.  The class of any locus of admissible covers of genus 0 curves with fixed degree $d$ and fixed ramification profile  is tautological (see \cite{Faber2005}). However, not all classes coming from spaces of admissible covers are tautological (see for example \cite{Graber2001} and \cite{vanZelm}). In fact such classes are the prime example of explicit nontautological algebraic classes.

In the cases where the admissible cover cycles are tautological, there is a wide literature on expressing such cycles in terms of decorated stratum classes, which are a set of additive generators of the tautological ring\footnote{See the following section for a reminder about the tautological ring.}. 
In \cite[Theorem 2.2]{Eisenbud1987} the class  $[\textup{Wp}]\in A^1(\M_{2,1})$ of the locus of pointed genus 2 curves where the marked point is one of the Weierstrass points of the curve is computed in terms of decorated stratum classes. Since every genus 2 curve is hyperelliptic and the Weierstrass points coincide with ramification points of the hyperelliptic map, the cycle $[\textup{Wp}]$ equals $[\Hyp_{2,1,0}]$, the class  of the locus of hyperelliptic genus $2$ curves with one point fixed by the involution. The computation of the class $[\Hyp_3]\in A^1(\M_3)$ of the locus of genus 3 hyperelliptic curves 
is well known (see for example \cite[Section 3.H]{harris1998moduli}). The computation of $[\Hyp_4]$ in terms of decorated stratum classes is much harder and was first given in \cite[Proposition 5]{Faber2005}. More recently in \cite{Cavalieri2017}  a generating series for the class of the locus $[\Hyp_{2,\hypbiram,0}]$ is given for all $0\leq \hypbiram \leq 6$ in terms of $\psi$ classes and graphs. In \cite{Faber2012} the class $[\B_3]\in A^2(\M_3)$ of the locus of bielliptic curves of genus 3 is calculated.

In this paper, we give a systematic way for computing the intersection of admissible cover cycles with decorated stratum classes and show how to use this information to compute the admissible cover cycles, in the cases where they are tautological.


\subsection{The Tautological Ring}
In the following we give a reminder about the tautological ring on the moduli space of curves. For a more detailed introduction see e.g. \cite{faber2000,calcmodcurves,Arbarello2011}.

The moduli space $\M_{g,n}$ of stable curves $(C,p_1, \ldots, p_n)$ has a stratification according to the topological type of $C$. This type is encoded in a stable graph $A$, whose vertices $V(A)$ correspond to the irreducible components of $C$ and whose edges $E(A)$ correspond to nodes of $C$ connecting these components. Further, the vertices of $A$ carry legs $l_1, \ldots, l_n$ indicating on which component the marked points $p_1, \ldots, p_n$ are located.

Given a stable graph $A$ we can parametrize the closure of the locus of stable curves $(C,p_1, \ldots, p_n)$ with type $A$ by a map
\[\xi_A : \M_A = \prod_{v \in V(A)} \M_{g(v),n(v)} \to \M_{g,n}\]
which glues a union of marked curves by identifying pairs of markings and forms the corresponding stable curve, prescribed by the data of the graph $A$.

On the moduli spaces $\M_{g(v),n(v)}$ there are natural cycle classes given by the cotangent line bundles $\psi_i$, the Arbarello-Cornalba classes $\kappa_a$ and the classes given by taking the closure of the locus of curves of a specified topological type. We call a \emph{decorated stratum class} the pushforward of a product of $\kappa$ and $\psi$ classes under the morphism $\xi_A$. In particular, pushing forward the fundamental class of $\M_A$ gives (a multiple of) the class of the closure of the locus of curves of type $A$. The vector spaces $R^\bullet(\M_{g,n}) \subset A^\bullet(\M_{g,n})$ and $RH^\bullet(\M_{g,n}) \subset H^\bullet(\M_{g,n})$ spanned by decorated stratum classes turn out to be closed under the intersection product and are called the tautological rings. Alternatively (see \cite{Faber2005}) these rings can be defined as the smallest system of $\Q$-subalgebras of $A^\bullet(\M_{g,n})$ and $H^\bullet(\M_{g,n})$ closed under pushforward along the gluing and forgetful morphisms. 

The intersection between two decorated stratum classes has an explicit, combinatorially defined expression in terms of decorated stratum classes. This result was first written down in \cite{Graber2001} and has since been implemented as a computer program, see \cite{Pixtprogram} and \cite{Yang2008}. 

Our primary goal in this paper is to derive a similar formula for the intersection of a locus of admissible covers with a decorated stratum class. In order to state our results, we need to establish some notation in the following section.

\subsection{Stacks of Admissible \texorpdfstring{$G$}{G}-Covers}
A modular interpretation for the loci of covers can be given by stacks of admissible covers. The definition of admissible covers was first written down in \cite{Harris1982} for the case of $d$-sheeted covers of genus 0 curves and later clarified and extended in \cite{Abramovich2001,Jarvis2005,BertinRomagny}. See Definition \ref{def:admGcover} for a reminder.

In this paper, we are going to study loci of admissible $G$-covers for a fixed finite group~$G$, that is, admissible covers $\varphi: C \to D$ of curves together with a $G$-action on $C$ making~$\varphi$ the quotient map. 

The reason why we restrict ourselves to $G$-covers instead of arbitrary covers of fixed degree is that it makes the moduli spaces of such covers behave like $G$-equivariant versions of moduli spaces of curves, allowing a simple description of the corresponding combinatorics. Since every degree $2$ cover is automatically a $\mathbb{Z}/2\mathbb{Z}$-cover, the important cases of hyperelliptic and bielliptic loci are covered 
by our treatment. For a remark about the case of arbitrary covers see Section \ref{Sect:outlook} below.

Given a $G$-cover $\varphi: C \to D$, the ramification points of $\varphi$ are exactly those smooth points of $C$ having a nontrivial $G$-stabilizer. The ramification behaviour of $\varphi$ is encoded in a \emph{monodromy datum} $\xi$, a tuple of elements of $G$ generating the stabilizers at the various marked points in $C$.


Fix a genus $g \geq 0$, a finite group $G$ and a monodromy datum $\xi=(h_1, \ldots, h_b) \in G^b$. Then we can define a stack $\H_{g,G,\xi}$ whose objects  are given by morphisms
\begin{equation} \label{eqn:introGCD}  G \curvearrowright (C,(p_{i,a})_{\substack{i=1, \ldots, b\\ a \in G/\langle h_i \rangle}}) \xrightarrow{\varphi} (D,(q_i)_{i=1, \ldots, b}) \end{equation}
 where $(C,(p_{i,a})_{i,a})$ and $(D,(q_i)_i)$ are (connected) stable curves, with the  genus of $C$ being $g$. The map $\varphi: C \to D$ is an admissible cover such that the action of $G$ on $C$ is a principal $G$-bundle outside the preimages of markings and nodes of $D$. Further we have $\varphi^{-1}(q_i)=\{p_{i,a}: a\in G/\langle h_i \rangle\}$ for all $i$. The monodromy of the $G$-action at the $p_{i,a}$ is described by $h_i$ and the $G$-action permutes the points $p_{i,a}$ according to left-multiplication of $G$ on $a \in G/\langle h_i \rangle$.



The space $\H_{g,G,\xi}$ is a smooth proper Deligne-Mumford stack. It is closely related to the stacks of pointed admissible covers defined in \cite{Jarvis2005}. A crucial difference is that we require the curve $C$ to be connected, while in \cite{Jarvis2005} also disconnected admissible covers are considered.

We ask $C$ to be connected since it allows us to define maps
\begin{equation} \label{eqn:introphidelta}
\begin{tikzcd}
\H_{g,G,\xi} \arrow[r,"\phi"] \arrow[d,"\delta"] &\M_{g,r}\\
\M_{g',b}
\end{tikzcd}
\end{equation}
where $r$ and $b$ are the number of ramification and branch points of the admissible cover and where $\phi$ sends the admissible cover (\ref{eqn:introGCD}) to the stable curve $(C,(p_{i,a})_{i,a})$ and $\delta$ sends it to $(D,(q_i)_i)$. The map $\phi$ is representable, finite, unramified and a local complete intersection, while $\delta$ is flat, proper and quasi-finite, but in general not representable.

We obtain a cycle class $\phi_*[\H_{g,G,\xi}] \in A^\bullet(\M_{g,r})$. These cycles (and their pushforwards under forgetful maps $\pi: \M_{g,r} \to \M_{g,n}$) are the main objects of study in this paper. These include the cycles of loci $\Hyp_{g,\hypbiram,2\hypbipair}$ and $\B_{g,\hypbiram,2\hypbipair}$ of hyperelliptic and bielliptic curves mentioned above.

\subsection{Intersection Results}
Now we are ready to present our results about the intersection of $\phi_*[\H_{g,G,\xi}]$ with a cycle $\xi_{A*} [\M_A]$ coming from a boundary stratum for a stable graph $A$. In the result, we use that, similarly to $\M_{g,n}$, the space $\H_{g,G,\xi}$ has a stratification by topological type. Roughly, the strata are enumerated by the stable graphs $\Gamma$ of the curves $C$ together with a $G$-action on $\Gamma$ induced by the $G$-action on the components, nodes and markings  of $C$ (see Section \ref{Sect:boundaryHurwitz}). Denoting the graph with $G$-action by $(\Gamma,G)$, the corresponding locus in $\H_{g,G,\xi}$ admits a parametrization by an equivariant gluing map
\[\xi_{(\Gamma,G)} : \H_{(\Gamma,G)} = \prod_{w \in V(\Gamma)/G} \H_{g(w),G_w,\xi_w} \to \H_{g,G,\xi}.\]
 The factors $\H_{g(w),G_w,\xi_w}$ correspond to \emph{orbits} of vertices in $\Gamma$ under $G$. Indeed, for a curve in $\H_{g,G,\xi}$ with stable graph $(\Gamma,G)$, an orbit of a vertex $w$ in $\Gamma$ corresponds to an orbit of components of the curve. Since $G$ acts by isomorphisms of the curve, all these components are isomorphic, and their isomorphism type is specified by an element of $\H_{g(w),G_w,\xi_w}$.

One checks that the intersection of (the image of) ${\phi: \H_{g,G,\xi} \to \M_{g,r}}$  with (the image of) $\xi_A$ is given as the union of the spaces $\H_{(\Gamma,G)}$ such that $\Gamma$ is a specialization of the graph $A$. Technically speaking this means that we have a morphism $f: \Gamma \to A$ of stable graphs (also called an $A$-structure on $\Gamma$ or a contraction of $\Gamma$) such that each $G$-orbit of edges of $\Gamma$ contains at least one edge coming from $A$. For details see Definition \ref{Def:genericequivAstructure}. We call such an $f$ a \emph{generic $A$-structure on $(\Gamma,G)$}. Denote by $\mathfrak{H}_{A;G,\xi}$ the set of isomorphism classes of generic $A$-structures $(\Gamma,G,f)$ for a given space $\H_{g,G,\xi}$.

In this situation, the composition of $\xi_{(\Gamma,G)}$ with the map $\phi: \H_{g,G,\xi} \to \M_{g,r}$ factors through the gluing maps $\xi_\Gamma$ and $\xi_A$ via a diagram

\begin{equation}
  \begin{tikzcd}
   \H_{(\Gamma,G)} \arrow[dd, "\phi_f", bend right=35,swap] \arrow[r,"\xi_{(\Gamma,G)}"]\arrow[d,"\phi_{(\Gamma,G)}"] & \H_{g,G,\xi}\arrow[d,"\phi"]\\
   \M_\Gamma  \arrow[d] \arrow[r,"\xi_\Gamma", swap] & \M_{g,r}\\
   \M_A \arrow[ru,"\xi_A",swap] &
  \end{tikzcd}
  \end{equation}


 \begin{globtheorem}[see Theorem \ref{th:main}] \label{Thm:intropullback}
  We have
  \begin{align*}
   \xi_A^*\phi_{*} ([\H_{g,G,\xi}])=  \sum_{ (\Gamma,G,f) \in \mathfrak{H}_{A;G,\xi}} c_{\textup{top}}(E_f)
   \cap  \phi_{f*} 
   [\H_{(\Gamma,G)}].
  \end{align*}
    where $c_{\text{top}}(E_f)$ is the top Chern class of an excess intersection bundle. It is the product of factors $-(\psi_h + \psi_{h'})$ for some half-edges $h,h'$ of $A$, as described in Proposition \ref{prop:excess}.
 \end{globtheorem}
Clearly, by the projection formula this theorem allows us to express the intersection of $\phi_*[\H_{g,G,\xi}]$ with a decorated stratum class for the graph $A$ using tautological classes\footnote{In the description of the map $\phi_{(\Gamma,G)}$, diagonal embeddings $\Delta : \M_{g_i,n_i} \to (\M_{g_i,n_i})^{\times m}$ naturally appear (see Section \ref{Sect:boundaryHurwitz}). Thus to give a complete description of the intersection of $\phi_*[\H_{g,G,\xi}]$ with a decorated stratum class, we need to understand the class of the diagonal. While this class will not be tautological in general, it is tautological for many small $(g,n)$  and in that case we can describe its Kunneth decomposition (see Application \ref{app:tautkunneth}).} and the fundamental classes of the factors in $\H_{(\Gamma,G)}$.
 
There are several variants of the above result. With a bit more care, it is possible to replace $\phi_*[\H_{g,G,\xi}]$ by a pushforward $\pi_* \phi_*[\H_{g,G,\xi}]$ from a moduli space with more markings (in other words, we forget some of the ramification points). This is described in Section \ref{Sect:forgetting}. 

Another interesting consequence of the proof of the theorem above is the following result.
\begin{globtheorem} [see Theorem \ref{th:mainpush} and  Corollary \ref{Cor:pullpushdelta}]
  For the maps $\phi, \delta$ as in the diagram (\ref{eqn:introphidelta}), the composition
  \[
   \delta_*\phi^* \colon A^\bullet(\M_{g,r})\rightarrow A^\bullet(\M_{g',b})
  \]
 sends tautological classes to tautological classes. More precisely, on decorated boundary strata, this map admits an explicit, combinatorial formula. 
\end{globtheorem}
In particular, this theorem allows us to compute intersection numbers of cycles $\phi_*[\H_{g,G,\xi}]$ with decorated boundary strata of the complementary dimension: these strata are mapped to tautological zero cycles via $\delta_* \phi^*$ and we can compute the degrees of those explicitly.

\subsection{Applications}
One of the main applications of the results in the previous section is that they often allow us to compute cycles of admissible covers in terms of decorated stratum classes in the case that these cycles are tautological. Let $\alpha \in H^{2i}(\M_{g,n})$ be an admissible cover cycle, i.e. a cycle of the form $\alpha=\pi_* \phi_* [\H_{g,G,\xi}]$ for some forgetful map $\pi : \M_{g,r} \to \M_{g,n}$. Assume that $\alpha$ is tautological.

The first method to compute $\alpha$ is to calculate its intersection numbers with tautological cycles of the complementary cohomological degree $2(3g-3+n-i)$. If the pairing
\[RH^{2i}(\M_{g,n}) \otimes RH^{2(3g-3+n-i)}(\M_{g,n}) \to \mathbb{Q}\]
is perfect in the relevant degree (e.g. if all cohomology is tautological) then these intersection numbers completely determine $\alpha$.

A second method is to compute the pullback of $\alpha$ via boundary maps $\xi_A : \M_A \to \M_{g,n}$ for various stable graphs $A$. Using a variant of Theorem \ref{Thm:intropullback} (see Theorem \ref{th:mainpush}) we can express the result in terms of cycles of smaller-dimensional spaces $\H_{g_i,G_i,\xi_i}$, which can be computed by recursion. Comparing the result with the pullbacks of a basis of $RH^{2i}(\M_{g,n})$ we obtain a linear condition on $\alpha \in RH^{2i}(\M_{g,n})$. Combining this information for different graphs $A$ is often sufficient to compute $\alpha$.

Since these computations quickly become untractable by hand, we have implemented the operations discussed in this section in a computer program \cite{Schmittprogram}, written in SAGE \cite{sagemath}. This program uses several functions already implemented in \cite{Pixtprogram}. Our program is available on the websites of the authors together with a short manual describing its use and some example computations.

Using the program we are able to greatly extend the list of such classes which can be computed (as well as repeat and verify most of the afformentioned calculations).
In Figures \ref{fig:hyperelliptcycles} and \ref{fig:bielliptcycles} we give lists of the values $(g,\hypbiram,\hypbipair)$ for which we can calculate the cycles $[\Hyp_{g,\hypbiram,2\hypbipair}]$ and $[\B_{g,\hypbiram,2\hypbipair}]$ of hyperelliptic and bielliptic curves  as sums of decorated stratum classes. Note that in the hyperelliptic case, we always have $\hypbiram \leq 2g+2$ and in the bielliptic case $\hypbiram \leq 2g-2$.

\begin{figure}[h]
    \centering
    \begin{tabular}{|c||c|c|c||c|c|c|c|c||c|c|c|c||c|c||c|c||c|c||c||}
    \hline
       g  & \multicolumn{3}{c||}{0} & \multicolumn{5}{c||}{1}& \multicolumn{4}{c||}{2}& \multicolumn{2}{c||}{3}& \multicolumn{2}{c||}{4}& \multicolumn{2}{c||}{5} & \multicolumn{1}{c||}{6}\\
    \hline  
       \hypbiram  & 0 & 1 & 2   &   0 & 1 & 2 & 3 & 4  &   0 & 1 & 2 & 3  &   0 & 1    &    0 & 1    &    0 & 1    & 0 \\
    \hline  
       \hypbipair  & 4 & 3 & 2   &   2 & 2 & 1 & 1 & 0  &   2 & 0 & 0 & 0  &   1 & 0    &    0 & 0    &    0 & 0    & 0 \\
    \hline 
    \end{tabular}
    \caption{Values of $g,\hypbiram,\hypbipair$ for which we computed the cycle $[\Hyp_{g,\hypbiram,2\hypbipair}]$ explicitly}
    \label{fig:hyperelliptcycles}
\end{figure}

\begin{figure}[h]
    \centering
    \begin{tabular}{|c||c||c|c|c||c|c||c||}
    \hline
       g  & \multicolumn{1}{c||}{1} & \multicolumn{3}{c||}{2}& \multicolumn{2}{c||}{3}& \multicolumn{1}{c||}{4}\\
    \hline  
       \hypbiram  & 0 & 0 & 1 & 2   &   0& 1    &  0 \\
    \hline  
       \hypbipair & 2 &1  &  0  & 0  & 0 & 0& 0\\
    \hline 
    \end{tabular}
    \caption{Values of $g,\hypbiram,\hypbipair$ for which which we computed the cycle $[\B_{g,\hypbiram,2\hypbipair}]$ explicitly}
    \label{fig:bielliptcycles}
\end{figure}

As explicit examples, we give expressions for the cycles $[\Hyp_5] \in H^{6}(\M_5)$ and $[\B_4] \in H^{6}(\M_4)$ in Theorem \ref{th:H5} and Theorem \ref{th:B4}. We describe the computation of $[\Hyp_6]\in H^{8}(\M_6)$ in Remark \ref{Rm:H6}. Note that our methods also allow us to go beyond the case of double covers. In Section \ref{Sect:furtherexa} we give some examples of cycles of cyclic triple covers that we can currently compute.




\subsection{Outlook} \label{Sect:outlook}
The results in our paper show that cycles of admissible covers behave very well under intersection with tautological classes. It is therefore natural to define an extension of the tautological ring obtained by adding all such cycles. As shown in \cite{Graber2001} and \cite{vanZelm}, this eventually gives strictly more cycle classes.

However, to work  with this extension we must also understand the intersection of different admissible cover cycles. We investigate a first example of this in Section \ref{sec:inthyp}, where we pull back the locus of bielliptic curves in $\M_4$ to the space of hyperelliptic curves. As shown there, this intersection between the hyperelliptic and bielliptic locus again has a neat description in terms of admissible cover spaces and the Chern class of the excess bundle is tautological. We expect this to be true in general, which would allow us to write down an explicit additive set of generators for the extended tautological ring, together with a combinatorial description of their intersections. 

Another direction is the generalization from admissible $G$-covers to arbitrary covers of a fixed degree $d$. As described in \cite[Section 4.2]{Abramovich2001}, for every admissible cover $E \to D$ of degree $d$ there is a $S_d$-admissible cover $C \to D$ such that $E=C/S_{d-1}$, where $S_d$ is the symmetric group. In other words we have a diagram
\[
\begin{tikzcd}
   C \arrow[d, "/S_{d-1}",swap] \arrow[dr, "/S_{d}"] & \\
   E \arrow[r] & D
  \end{tikzcd}
\]
This suggests, given a space $\H_{g,G,\xi}$ and a subgroup $H \subset G$, to consider the map
\begin{align*}\phi_H : \H_{g,G,\xi} &\to \M_{g'',b'},\\ \big[(G \curvearrowright (C,(p_{i,a})_{i,a}) \to (D,(q_i)_i))\big] &\mapsto (C/H, (\eta(p_{i,a}))_{i,a}),\end{align*}
where $\eta: C \to C/H$ is the quotient map. For the trivial subgroup $H=\{e_G\} \subset G$ we obtain our previous map $\phi$, for $H=G$ we obtain the map $\delta$. By the argument above, for $G=S_d$ the cycles $\phi_{S_{d-1} *} [\H_{g,G,\xi}]$ allow us to describe all degree $d$ admissible covers, and we expect that they behave similarly to the cycles $\phi_* [\H_{g,G,\xi}]$ we saw above. We plan to investigate this in further papers.

Finally, it is possible to look at the intersection theory on the spaces $\H_{g,G,\xi}$ themselves. There is a natural definition of a tautological ring $R^\bullet(\H_{g,G,\xi})\subset A^\bullet(\H_{g,G,\xi})$ and thus of tautological relations in this ring. A first basic way to obtain such relations is to pull back usual tautological relations on $\M_{g',b}$ via the target map $\delta: \H_{g,G,\xi} \to \M_{g',b}$.  Any tautological relation on $\H_{g,G,\xi}$ pushes forward to a relation between cycles from admissible cover spaces under the source map $\phi: \H_{g,G,\xi} \to \M_{g,r}$. Such relations can potentially be used to express one admissible cover cycle in terms of smaller-dimensional ones, enabling us to compute them recursively.

\subsection{Outline of the Paper}
In Section \ref{sec:intbound} we recall the definition of decorated stratum classes and compute the intersection between such classes following \cite{Graber2001}. This section mainly serves as a warm-up and to establish notation. In Section \ref{sec:hurspace} we introduce the stack of admissible covers and state a number of facts for later. We compute the intersection between spaces of admissible covers and decorated stratum classes in Section \ref{sec:theint}. Finally in Section \ref{ch:computing} we show how these theorems can be used to compute classes of admissible covers in terms of decorated stratum classes. We give a number of examples explicitly.






\section*{Acknowledgements}

The authors would like to thank Nicola Pagani and Rahul Pandharipande for the many helpful discussions and comments during this project. 
Special thanks also go to Aaron Pixton for allowing us to use and modify his program \cite{Pixtprogram}. The first author would also like to thank Johannes Lengler from the Algorithm Consulting Service of ETH Z\"urich for advice concerning the implementation of the methods in this paper. 

The first author was
supported by the grant SNF-200020162928. The second author was supported by a GTA fellowship from the University of Liverpool and an Einstein fellowship at Humboldt Universit\"{a}t Berlin.

\section{Intersections in the tautological ring}\label{sec:intbound}

In this section we will compute the intersection between two decorated stratum classes in terms of decorated stratum classes. This was first done in \cite{Graber2001}. The purpose of this section is to introduce notation and to serve as a warmup for later when we will compute the intersection between  decorated stratum classes and stacks of pointed admissible $G$-covers.
 Readers familiar with tautological classes can likely skip to Section \ref{sec:hurspace} and refer back to this section as necessary.

\begin{definition}
  Recall that an \emph{undirected finite graph} (or simply \emph{graph}) is a triple
  \[
   (V,E,s\colon E\rightarrow \Sym^2V)
  \]
  where $V$ is a finite set of vertices, $E$ is a finite set of edges and $s$ sends each edge into the second symmetric power of $V$ thus assigning two vertices to every edge. 
  
  Also recall that a \emph{path} between two vertices $v,v'\in V$ is a sequence of vertices \[v=v_0,v_1,...,v_n=v'\] and a sequence of edges $e_1,...,e_n$ such that $s(e_i)=v_{i-1}v_i$ for all $i$. A graph is said to be \emph{connected} if there is a path between any pair of vertices $v,v'\in V$.
  \end{definition}
  
   \begin{definition}
  We define a \emph{stable graph} $\Gamma$ to be the data
  \[
   \Gamma := (V ,H ,L ,g \colon V \rightarrow \Z_{\geq 0}, \iota \colon H  \rightarrow H , a \colon H \rightarrow V , \zeta  \colon L  \rightarrow V )
  \]
 which satisfies the following conditions:
 \begin{enumerate}[i]
  \item $\iota $ is a fixed point free involution,
  
  \item let $E $ be the set of orbits of $\iota $ and let $s\colon E \rightarrow \Sym^2V $ be the function induced by $a $ on $E $ then $(V ,E ,s )$ is a connected graph,
  
  \item for each vertex $v\in V $ the stability condition $2g(v)-2+|a ^{-1}(v)|+|\zeta ^{-1}(v)|>0$ is satisfied.
 \end{enumerate}
 \end{definition}
 
 \begin{notation}
 We call the elements of $H$ \emph{half edges} and the elements of $L$ \emph{legs}. We call $g$ the \emph{genus function} and the genus of a vertex is defined to be $g(v)$. For a vertex $v\in V$ we set $n(v)= |a^{-1}(v)|+|\zeta^{-1}(v)|$ the number of \emph{half edges plus legs incident to $v$}.  The \emph{genus of a stable graph} $\Gamma$ is defined to be
 \[
  g(\Gamma):=\sum_{v\in V}g(v)+h^1(\Gamma).
 \]
  We will say that a stable graph $\Gamma$ is $n$-pointed if $n=|L|$.
\end{notation}

   \begin{example}
   Given a stable curve $C$ over $\spec \C$, we associate a stable graph to $C$ as follows. Let $V$ be the set of irreducible components of $C$, let $H$ be the set of sections of the normalization of $C$ corresponding to the nodes of $C$, $L$ the set of marked points of $C$, $g$ the function sending the irreducible components to the genus of their normalization, $\iota$ the involution identifying sections $h\in H$ mapping to the same node of $C$, $a$ the map identifying sections with the irreducible component they lie on and $\zeta$ the map sending a marked point to its corresponding irreducible component. The graph $(V,H,L,g,\iota,a,\zeta)$ is stable and is called the \emph{dual graph} of $C$.
  \end{example}
   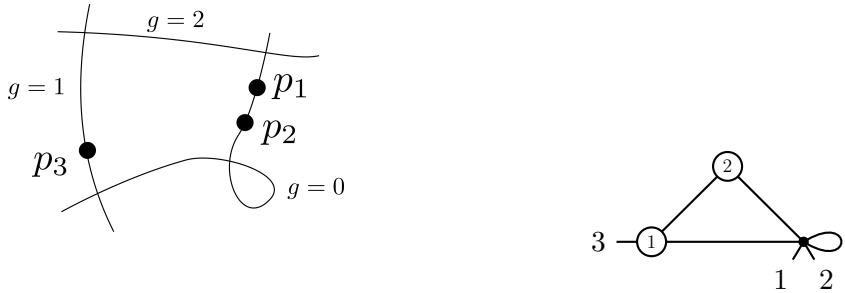
\begin{figure}[H]
  \centering
    \includegraphics[width=.3\textwidth]{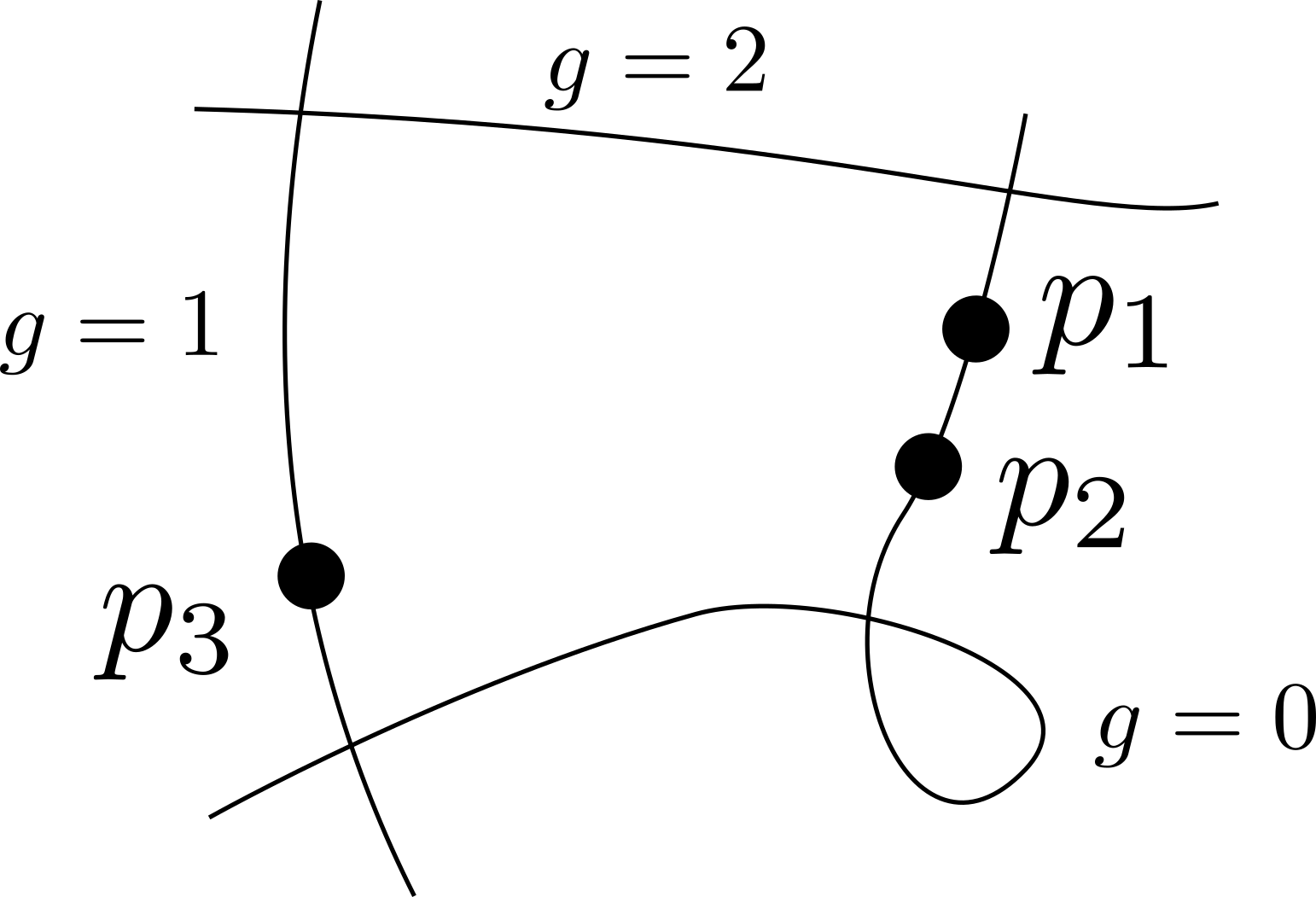}\hspace{3cm}
 \begin{tikzpicture}[baseline=4pt,->,>=bad to,node distance=1.3cm,thick,main node/.style={circle,draw,font=\Large,scale=0.5}]
\node at (0,1) [main node] (A) {2};
\node at (1,0) [scale=.3,draw,circle,fill=black] (B) {};
\node at (-1,0) [main node] (C)  {1};
\node at (-1.7,0)  (p3)  {3};
\node at (1.3,-.5)  (p2)  {2};
\node at (.7,-.5)  (p1)  {1};
\draw [-] (A) to  (B);
\draw [-] (C) to  (B);
\draw [-] (A) to (C);
\draw [-] (p3) to (C);
\draw [-] (p2) to (B);
\draw [-] (p1) to (B);
\draw [-] (B) to [out=30, in=-30,looseness=30] (B);
\end{tikzpicture}
    \label{fig:curve}
    \caption{A curve and its dual graph. When the genus of a vertex is 0 we shall depict it as a black dot.}
\end{figure}

The following definition formalizes the notation that a stable graph $\Gamma$ is a specialization of a stable graph $A$, recording the data of this specialization. 
  
   \begin{definition}\label{def:morgraph}
  Let $A$ and $\Gamma$ be  genus $g$ stable graphs with set of legs $L_\Gamma=L=L_A$. An $A$-structure on $\Gamma$ is a triple
  \[
  ( \alpha \colon V_\Gamma \twoheadrightarrow V_A,\; \beta\colon H_A\hookrightarrow H_\Gamma ,\; \gamma\colon H_\Gamma \backslash \im \beta \rightarrow V_A)
  \]
 which satisfies the following conditions:
 \begin{enumerate}[i]
  \item the map $\beta$ commutes with the involutions, i.e.\  $\beta\circ \iota_A =\iota_\Gamma \circ \beta$, 
  
  \item the map $\alpha$ respects the leg assignments, i.e.\  $\alpha \circ \zeta_\Gamma = \zeta_A$,
  
  \item if $h\in \im \beta$ and $a_\Gamma(h)=v$ then $a_A(\beta^{-1}(h))=\alpha(v)$,
  
  \item if $h\in H_\Gamma\backslash \im \beta$ and $a_\Gamma(h)=v$ then $\alpha(v)=\gamma(h)$,
  
  \item if $v\in V_A$ then 
  \[
   (\alpha^{-1}(v),\; \gamma^{-1}(v),\; \beta(a_A^{-1}(v))\cup \zeta_A^{-1}(v),\; g_\Gamma, a_\Gamma, \iota_\Gamma, \zeta)
  \]
 is a stable graph of genus $g(v)$ (where  $g_\Gamma$, $a_\Gamma$ and $\iota_\Gamma$ are restricted to the appropriate subsets and $\zeta$ is defined by $\zeta_\Gamma$ on $\zeta_A^{-1}(v)$ and by $a_\Gamma$ on $\beta(a_A^{-1}(v))$).
 \end{enumerate}
 \end{definition}

 \begin{remark}
  If $\Gamma$ has an $A$-structure $f=(\alpha,\beta,\gamma)$ and $A$ has a $B$-structure $g=(\alpha',\beta',\gamma')$ it is easy to check that there exists a unique $\gamma'': H_\Gamma \setminus \im (\beta \circ \beta') \to V_B$ making $(\alpha'\circ \alpha,\beta\circ\beta' ,\gamma'')$ a $B$-structure on $\Gamma$. Indeed, $\gamma''$ is given as $\alpha' \circ \gamma$ on $H_\Gamma \setminus \im \beta$  and as $\gamma' \circ \beta^{-1}$ on $\im \beta \setminus \im \beta \circ \beta'$. We can therefore define a \emph{morphism of $n$ pointed genus $g$ stable graphs} $\Gamma\rightarrow A$  as an $A$-structure on $\Gamma$. An \emph{isomorphism} of stable graphs $f: A\rightarrow B$ is thus a $B$-structure  $f=(\alpha,\beta,\gamma)$ on $A$ and an $A$-structure  $g=(\alpha',\beta',\gamma')$ on $B$ such that $g \circ f=(\text{id}_{V_A},\text{id}_{H_B},\text{id}_\emptyset)$ and $f \circ g=(\text{id}_{V_B},\text{id}_{H_A},\text{id}_\emptyset)$.
   
 We have formed an essentially finite category of stable graphs of genus $g$ with set of legs $L$. 
 \end{remark}

 \begin{example}
 Let 
 \begin{center}
\begin{tabular}{cc}
$\Gamma =
\begin{tikzpicture}[->,>=bad to,baseline=-3pt,node distance=1.3cm,thick,main node/.style={circle,draw,font=\Large,scale=0.5}]
\node at (0,0) [main node] (A) {1};
\node at (1.3,0) [main node] (B)  {1};
\node at (2.6,0) [main node] (C)  {1};
\node at (.4,.5) (h1) {\tiny $h_1$};
\node at (.9,.5) (h2) {\tiny $h_2$};
\node at (1.7,.25) (h3) {\tiny $h_3$};
\node at (2.2,.25) (h4) {\tiny $h_4$};
\node at (.4,-.5) (h5) {\tiny $h_5$};
\node at (.9,-.5) (h6) {\tiny $h_6$};
\node at (-.3,-.3) (v1) {\tiny $v_1$};
\node at (1.6,-.3) (v2) {\tiny $v_2$};
\node at (2.9,-.3) (v3) {\tiny $v_3$};
\draw [-] (A) to [out=40, in=140] (B);
\draw [-] (A) to [out=-40, in=-140] (B);
\draw [-] (B) to (C);
\end{tikzpicture}$, &
$A =
\begin{tikzpicture}[->,>=bad to,baseline=-3pt,node distance=1.3cm,thick,main node/.style={circle,draw,font=\Large,scale=0.5}]
\node at (0,0) [main node] (A) {2};
\node at (1.3,0) [main node] (B) {1};
\node at (.3,-.3) (v1) {\tiny $v'_1$};
\node at (1.6,-.3) (v2) {\tiny $v'_2$};
\node at (-.7,.4) (h1) {\tiny $h'_1$};
\node at (-.7,-.4) (h2) {\tiny $h'_2$};
\node at (.4,.25) (h5) {\tiny $h'_3$};
\node at (.9,.25) (h6) {\tiny $h'_4$};
\draw [-] (A) to [out=150, in=-150,looseness=10] (A);
\draw [-] (A) to (B);
\end{tikzpicture} $
\end{tabular}
\end{center}
 One $A$-structure on $\Gamma$ is
 \begin{align*}
 \alpha\colon v_1, v_2&\mapsto v'_1 \\
  v_3& \mapsto v'_2\\
  \beta\colon h'_i &\mapsto h_i\\
  \gamma \colon h_5,h_6 &\mapsto v'_1.
 \end{align*}

 In total there are 4 different $A$-structures which can be given to $\Gamma$. Indeed we can send $h'_1$ to $h_1$, $h_2$, $h_5$, or $h_6$ and each of these choices completely determines an $A$-structure on $\Gamma$.
  \end{example}

 \begin{notation}\label{par:gluing}
  For a stable graph $\Gamma$ we define 
  \[
   \M_\Gamma := \prod_{v\in V} \M_{g(v),n(v)}.
  \]
There is a natural \emph{gluing morphism} \[\xi_{\Gamma}\colon \M_\Gamma \rightarrow \M_{g,n}\] defined by $\Gamma$. It glues the curve
 \[\left( \coprod_{v\in V} C_v\rightarrow \Base \;\; ; \;\; (\sigma_h\colon \Base \rightarrow C_{a(h)})_{h \in H},\; (\sigma'_{l}\colon \Base\rightarrow C_{\zeta(l)})_{l \in L}\right) \in \M_\Gamma \]
 over $\Base$ together by gluing the section $\sigma_h$ to $\sigma_{\iota(h)}$. 
 The morphism $\xi_\Gamma$ is a representable morphism of Deligne-Mumford stacks (see \cite[Proposition 10.25]{Arbarello2011}) and since both its domain and codomain are smooth and complete it is an lci and proper morphism.
 
   If we have a morphism of stable graphs $\Gamma \rightarrow A$ we can construct a corresponding map of moduli spaces 
  \[
   \xi_{\Gamma\rightarrow A} \colon \M_\Gamma \rightarrow \M_A
  \]
 defined as a composition of gluing morphisms component wise.
  \end{notation}
  
  \begin{remark}
   The image of $\xi_\Gamma$ in $\M_{g,n}$ equals the closure of the locus of all curves in $\M_{g,n}$ with dual graph isomorphic to $\Gamma$. The generic degree of $\xi_\Gamma$ equals the order of the automorphism group of $\Gamma$.
  \end{remark}

 \begin{prgrph}\label{pr:setup}
   Let $A$ and $B$ be $n$ pointed genus $g$ stable graphs. We will compute the intersection 
  \[
   \xi_{A*}([\M_A])\cdot \xi_{B*}([\M_B])
  \]
 as an explicit sum of classes of $\M_{g,n}$. By the excess intersection formula (see for example \cite[Proposition 17.4.1]{fulton1984}) we have to identify the fiber product
 \begin{center}
  \begin{tikzcd}
   \F_{A,B} \arrow[r,"p_2"]\arrow[d,"p_1"] & \M_B\arrow[d,"\xi_B"]\\
   \M_{A} \arrow[r,"\xi_A", swap] & \M_{g,n}
  \end{tikzcd}
 \end{center}
 and the top Chern class of the excess intersection bundle $E=p_1^*N_{\xi_A}/N_{p_2}$. We will prove $\F_{A,B}$ is isomorphic to a disjoint union of stacks $\M_\Gamma$ where $\Gamma$ is a graph which appears as a specialization of both $A$ and $B$. 
 \end{prgrph}

 \begin{prgrph} 
 We will say that $\Gamma$ has an \emph{$(A,B)$-structure $(f,g)$}, or that \emph{$\Gamma$ is an $(A,B)$-graph}, if $\Gamma$ has both an $A$-structure $f=(\alpha_A,\beta_A,\gamma_A)$ and a $B$-structure $g=(\alpha_B,\beta_B,\gamma_B)$. We say that $(f,g)$ is \emph{a generic $(A,B)$-structure} on $\Gamma$ if every edge of $\Gamma$ corresponds to an edge of $A$ or an edge of $B$, i.e.\  if 
 \[
  \beta_A(H_A)\cup \beta_B(H_B)=H_\Gamma.
 \]
 
 We say that two stable graphs $\Gamma$, $\Gamma'$ with $(A,B)$-structures $(f,g)$ and $(f',g')$ are isomorphic $(A,B)$-graphs if there exists an isomorphism $\tau:\Gamma\rightarrow \Gamma'$ such that the following diagram commutes
 \begin{center}
  \begin{tikzcd}
   & B \\
   \Gamma \arrow[ru,"g"]\arrow[r,"\tau"]\arrow[rd,"f",swap] & \Gamma'\arrow[u,"g'",swap]\arrow[d,"f'"]\\
   & A
  \end{tikzcd}.
 \end{center}
 \end{prgrph}


 \begin{example}
  Let 
  \[
   A=B =
   \begin{tikzpicture}[->,>=bad to,baseline=-3pt,node distance=1.3cm,thick,main node/.style={circle,draw,font=\Large,scale=0.5}]
\node at (0,0) [main node] (A) {3};
\draw [-] (A) to [out=30, in=-30,looseness=10] (A);
\node at (.7,.4) (h1) {\tiny $h_1$};
\node at (.7,-.4) (h2) {\tiny $h_2$};
\end{tikzpicture}.
  \]
There are three stable graphs which admit a generic $(A,B)$-structure:  
  \begin{center}
   \begin{tabular}{c@{\hskip 2cm}c@{\hskip 2cm}c}
  $\Gamma_1 = 
     \begin{tikzpicture}[->,>=bad to,baseline=-3pt,node distance=1.3cm,thick,main node/.style={circle,draw,font=\Large,scale=0.5}]
\node at (0,0) [main node] (A) {3};
\draw [-] (A) to [out=30, in=-30,looseness=10] (A);
\node at (.7,.4) (h'_1) {\tiny $h'_1$};
\node at (.7,-.4) (h'_2) {\tiny $h'_2$};
\end{tikzpicture}$  
  & $\Gamma_2 = \begin{tikzpicture}[->,>=bad to,baseline=-3pt,node distance=1.3cm,thick,main node/.style={circle,draw,font=\Large,scale=0.5}]
\node at (0,0) [main node] (A) {2};
\node at (1.3,0) [main node] (B) {1};
\node at (.4,.4) (h5) {\tiny $h'_1$};
\node at (.9,.4) (h6) {\tiny $h'_2$};
\node at (.4,-.4) (h5) {\tiny $h'_3$};
\node at (.9,-.4) (h6) {\tiny $h'_4$};
\draw [-] (A) to [out=30, in=150] (B);
\draw [-] (A) to [out=-30, in=-150] (B);
\end{tikzpicture} $
&
  $\Gamma_3 = 
     \begin{tikzpicture}[->,>=bad to,baseline=-3pt,node distance=1.3cm,thick,main node/.style={circle,draw,font=\Large,scale=0.5}]
\node at (0,0) [main node] (A) {2};
\draw [-] (A) to [out=30, in=-30,looseness=10] (A);
\draw [-] (A) to [out=150, in=-150,looseness=10] (A);
\node at (-.7,.4) (h'_1) {\tiny $h'_1$};
\node at (-.7,-.4) (h'_2) {\tiny $h'_2$};
\node at (.7,.4) (h'_3) {\tiny $h'_3$};
\node at (.7,-.4) (h'_4) {\tiny $h'_4$};
\end{tikzpicture}$ 
   \end{tabular}
  \end{center}
  The stable graph $\Gamma_1$ has two isomorphism classes of generic $(A,B)$-structures $(f,g)$: Set $f=(\alpha_f,\beta_f,\gamma_f)$ and $g=(\alpha_g,\beta_g,\gamma_g)$.   Up to isomorphism we can assume that $\beta_f(h_1)=h_1'$, there are then two possible nonisomorphic choices for the images of $\beta_g(h_1)$, namely $h'_1$ and $h'_2$.
  
  The graph $\Gamma_2$ has four isomorphism classes of generic $(A,B)$-structures: There is always an isomorphism such that $\beta_f$ sends the edge $(h_1,h_2)$ to $(h'_1,h'_2)$. We then have $\beta_f(h_1)=h'_1$ or $\beta_f(h_1)=h'_2$ and these choices are nonisomorphic. Since the $(A,B)$-structure is generic $\beta_g$ must send $(h_1,h_2)$ to $(h'_3,h'_4)$ and we again have 2 choices.
  
  The stable graph $\Gamma_3$ has only one isomorphism class of stable $(A,B)$-structures $(f,g)$.  Up to isomorphism we have $\beta_f(h_1)=h_1'$ and $\beta_g(h_1)=h_3'$.
 \end{example}

 \begin{notation}
   Let $A$ and $B$  be stable $n$ pointed genus $g$ graphs. We will denote by $\mathfrak{G}_{A,B}$ a set of representatives of the set of all isomorphism classes of generic $(A,B)$-graphs.
   We set 
   \[
 \X = \coprod_{(\Gamma,f,g)\in \mathfrak{G}_{A,B}} \M_{\Gamma}.
   \]
 \end{notation}

  \begin{proposition}[see Proposition 9, \cite{Graber2001}]\label{intprop}
  There is a natural isomorphism $\X\xrightarrow{\sim} \F_{A,B}$.
 \end{proposition}
 
 In the following we recall the proof of the above proposition, since 
 our proofs in Section \ref{sect:fibreprod} will follow a similar strategy and we introduce some essential terminology on the way.

 We start by giving a different modular interpretation of $\M_\Gamma$ for any stable graph $\Gamma$ (see  \cite[page 315]{Arbarello2011} or \cite[Section A2]{Graber2001}):

 \begin{definition}\label{def:curstruc}
  Let $\Gamma=(V,H,L,g,\iota,a,\zeta)$ be an $n$-pointed stable graph of genus $g$ and let~$C$ be an $n$-pointed stable curve
  \[
   \pi\colon C\rightarrow \Base, \qquad \qquad s_i\colon \Base \rightarrow C\quad i=1,...,n
  \]
 of genus $g$ over a connected base $\Base$. A \emph{$\Gamma$-marking on $C$} is the following additional data: (this is called a $\Gamma$-structure in \cite{Graber2001})
 \begin{enumerate}[i]
  \item $\# E$ additional disjoint sections $\sigma_1,...,\sigma_{e(\Gamma)}$ of $\pi$ with image in the singular locus of $C$,
  
  \item $\# H$ sections $\tilde{\sigma}_{1,1},\tilde{\sigma}_{1,2},\tilde{\sigma}_{2,1},...,\tilde{\sigma}_{e(\Gamma),2}$ of the normalization $\tilde{C}$ of $C$ along the sections $\sigma_i$,
  
  \item $\# V$ disjoint connected components $C_v$ of $C\backslash \{\sigma_i\}$ whose union is  $C\backslash \{\sigma_i\}$ and such that each $C_v$ remains connected upon pullback along any morphism $\Base'\rightarrow\Base$ of base schemes (we shall call such components \emph{$\pi$-relative components of $C\backslash \{\sigma_i\}$}),  
  
  \item a choice of isomorphism between $\Gamma$ and the stable graph \label{enum:curstruc}
  \[
   (\{C_v\},\{\tilde{\sigma}_{i,j}\}, \{s\},g,\iota,\alpha,\zeta)
  \]
where $g(C_v)$ is the arithmetic genus of $C_v$, the involution $\iota$ is defined by $\iota(\tilde{\sigma}_{i,1})=\tilde{\sigma}_{i,2}$, $\alpha$ maps $\tilde{\sigma}_{i,j}$ to the $\pi$-relative component corresponding to the component of $\tilde{C}$ it lies on and $\zeta$ maps $s_i'$ to the $\pi$-relative component it lies on.
 \end{enumerate}
 We will denote the curve $C$ together with the data of a $\Gamma$-marking on $C$ by $C_\Gamma$. 
 \end{definition}

 \begin{prgrph}
  The data of a $\Gamma$-marking on a stable curve can be pulled back under any morphism of connected base schemes. We can therefore define a stack $\M_\Gamma'$ whose objects consist of stable curves with $\Gamma$-marking and whose morphisms respects the $\Gamma$-marking under pullback. 
 \end{prgrph}

 \begin{proposition}[see Proposition 8, \cite{Graber2001}]\label{prop:Gcurve}
  There exists a natural isomorphism between $\M_\Gamma$ and $\M_\Gamma'$.
 \end{proposition}

 \begin{proof}
  We can construct a natural morphism from $\M_\Gamma$ to $\M_\Gamma'$ by assigning the canonical $\Gamma$-marking to the universal curve over $\M_\Gamma$. In the other direction given a $\Base$ valued point of $\M_\Gamma'$ we naturally obtain a collection of $v(\Gamma)$ stable curves by analysing the $\pi$-relative components of $C$. Since we have a bijection between these curves and $v(\Gamma)$ and a bijection between the sections of the normalization of $C$ and the sections of curves in $\M_\Gamma$ we obtain a $\Base$ valued point of $\M_\Gamma$. It is straightforward to check that this correspondence induces a bijection on the space of morphisms between corresponding objects. 
 \end{proof}

\begin{proof}[Proof of Proposition \ref{intprop}]
 In this proof we will always use the modular interpretation of curves with a $\Gamma$-structure for $\M_{\Gamma}$ and write $\M_{\Gamma}$ for $\M_{\Gamma}'$ everywhere.
 
 
 Let $u\colon \X\rightarrow \M_A$ be the map defined as $\xi_{f\colon \Gamma \rightarrow A}\colon \M_\Gamma \rightarrow \M_A$ on the connected component $\M_\Gamma$ of $\X$ indexed by $(\Gamma,f,g)$. Similarly define $v\colon \X\rightarrow \M_B$ to be the map  $\xi_{g\colon \Gamma \rightarrow B}\colon \M_\Gamma \rightarrow \M_B$ on the connected component of $\X$ indexed by $(\Gamma,f,g)$.   An object $(C_\Gamma,f,g)$ of~$\X$, over a connected base scheme $S$, consists of a graph $\Gamma$ with a generic $(A,B)$-structure~$(f,g)$ together with a stable curve $C$ over $S$ endowed  with a $\Gamma$-marking. Let $(C_\Gamma,f,g)$ be one such object of $\X$ over $S$. By definition we have $\xi_A(u(C_\Gamma,f,g))=C=\xi_B(v(C_\Gamma,f,g))$ a natural isomorphism $\xi_A\circ u\Rightarrow \xi_B \circ v$ is therefore given by the identity. We have the following diagram:
  \begin{center}
   \begin{tikzcd}
    \X \arrow[ddr,"u"{name=U},swap,bend right] \arrow[drr,"v",bend left] \arrow[dr,"q" description] & &\\
   &     \F_{A,B}\arrow[dr,phantom,"\Rightarrow"] \arrow[r,"p_2"]\arrow[d,"p_1",swap]  & \M_B\arrow[d,"\xi_B"]\\
   &  \M_A\arrow[r,"\xi_A",swap]&\M_{g,n}.
   \end{tikzcd}
  \end{center}
where the map $q$ is given by the strict universal property of the fiber product. It sends the object $(C_\Gamma,f,g)$ over $S$ to the object $(C_A,C_B,\id_C)$ over $S$ and a morphism $C_\Gamma'\rightarrow C_\Gamma$ over $S' \rightarrow S$ to the induced pair of morphisms $(C_A'\rightarrow C_A, C_B'\rightarrow C_B)$.

We want to prove that $q$ is an isomorphism. We will do so by defining a map $r\colon \F_{A,B}\rightarrow \X$ and showing that $r\circ q$ and $q\circ r$ are naturally isomorphic to the respective identities on $\X$ and on $\F_{A,B}$. 

Let $(D_A,C_B,\alpha\colon D\xrightarrow{\sim} C)$ be an object of $\F_{A,B}$ over $S$. Following the notation of Definition \ref{def:curstruc} we will define an $A$-marking on $C$ by passing through the isomorphism $\alpha$: 
\begin{enumerate}[i]
 \item if $\{\sigma_i\}$ are the $e(A)$ sections of $\varpi\colon D\rightarrow S$ in the singular locus of $D$ defined by the $A$-marking, we get sections $\sigma'_i:=\alpha\circ \sigma_i$ in the singular locus of  $C$,
 
 \item the pullback of $\alpha$ along the partial normalization $\tilde{C}\rightarrow C$ defines a map $\tilde{\alpha}\colon \tilde{D}\rightarrow \tilde{C}$, in this way we obtain sections $\{\tilde{\sigma}'_{i,j}\}:=\{\tilde{\alpha}\circ \tilde{\sigma}_{i,j}\}$ in the partial normalization $\tilde{C}$ of $C$ at $\{\sigma'_i\}$,
 
 \item if $D_{v,A}$ are the $\varpi'$ relative components of $D$ then the $\pi$-relative components are given by $C_{v,A}:=\alpha(D_{v,A})$,
 
 \item we obtain an isomorphism of stable graphs by the composition 
 \begin{center}
  \begin{tikzcd}
    (\{C_{v,A}\},\{\tilde{\sigma}'_{i,j}\}, \{\alpha\circ s_k \},g',\iota',a',\zeta') \arrow[r] & (\{D_{v,A}\},\{\tilde{\sigma}_{i,j}\}, \{s_k\},g,\iota,a,\zeta) \arrow[r,"\lambda"] & A
  \end{tikzcd}
 \end{center}
 where $\lambda$ is the isomorphism of stable graphs defined by the $A$-structure on $D$.
\end{enumerate}
    Let $\tau_i$ be the sections on $C$ defined by the $B$-structure, $C_{v,B}$ the $\pi$-relative components defined by the $B$-structure. The curve $C$ now comes with the following structure:
 \begin{figure}[H]
  \centering
    \includegraphics[width=.4\textwidth]{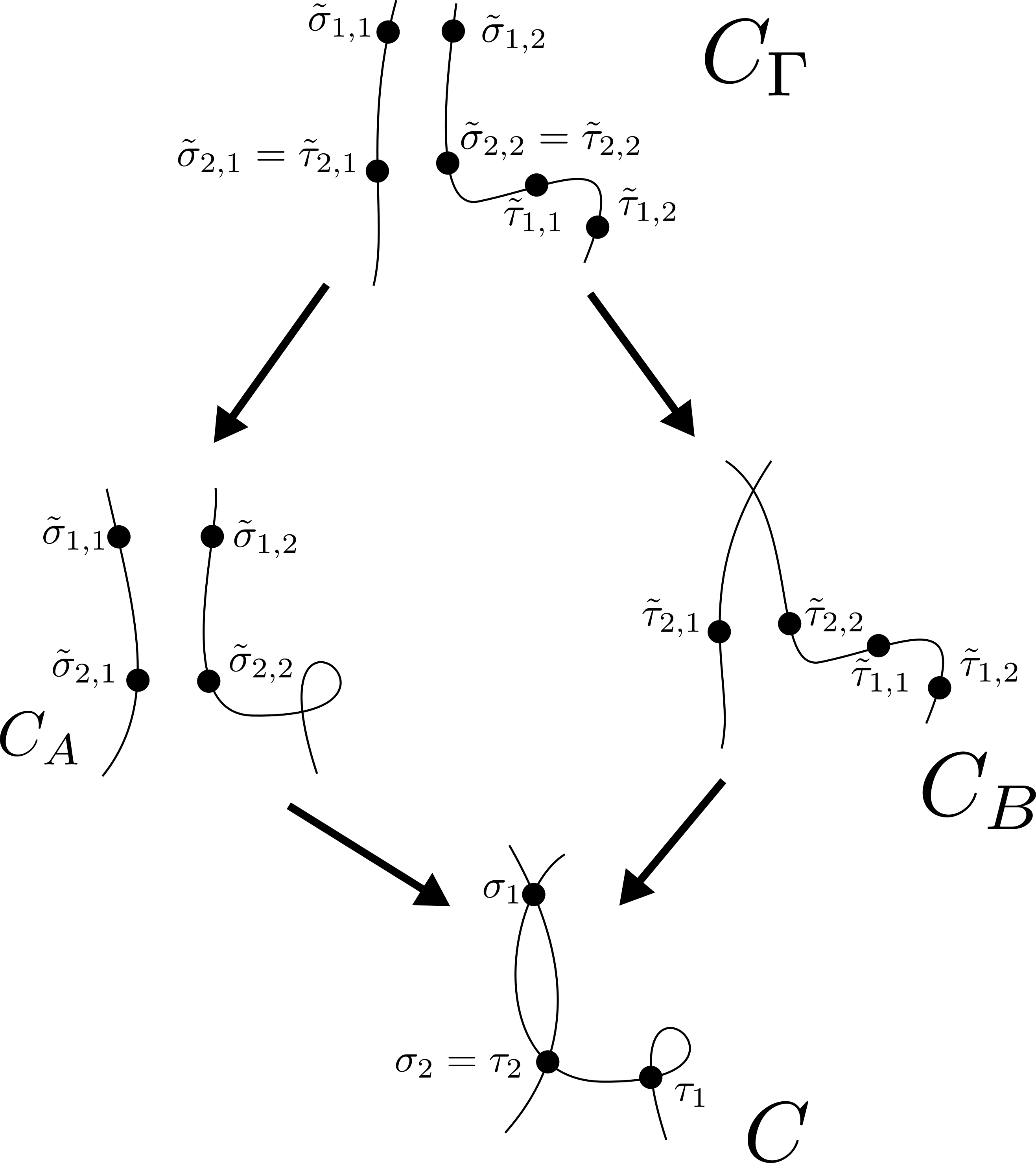}
    \label{fig:intprod}
\end{figure}

\begin{enumerate}[i]
 \item a set of sections $E:=\{\sigma'_{i}\}\cup \{\tau_i\}$ of $\pi$ in the singular locus of $C$,
 
 \item a set of sections $H:=\{\tilde{\sigma}'_{i,j}\} \cup \{\tilde{\tau}_{i,j}\}$ in the partial normalization of $C$ at $E$,
 
 \item a set of $\pi$-relative components $V$ of $C\backslash \{\sigma'_i\} \cup \{\tau_i\}$.
\end{enumerate}
This data defines a stable graph
\[
 \Gamma:=(V,H,\{\alpha(s_i)\}, g,\iota,a,\zeta)
\]
 as in Definition \ref{def:curstruc}.\ref{enum:curstruc} (where the $\alpha(s_i)$ are the $n$ sections of $C$ outside of the singular locus corresponding to the marked points). This data defines a $\Gamma$-marking on $C$.

 This $\Gamma$ has an $(A,B)$-structure, indeed let $\alpha_A\colon V \rightarrow V_A,$ be the map of $\pi$-relative components given by the inclusion $C\backslash\{\sigma'_i\}\cup \{\tau_i\}\hookrightarrow C\backslash\{\sigma'_i\}$,  let $\beta_A\colon H_A\hookrightarrow H$ be the obvious inclusion of sections and let $\gamma_A\colon H\backslash\im \beta \rightarrow H_A$ be the map that sends $\tilde{\tau}_{i,j}$ to the $\pi$-relative component $C_{v,A}$ in which $\tau_i$ lies. In this way we have constructed an $A$-structure $f=(\alpha_A,\beta_A,\gamma_A)$ on $\Gamma$ and we can define a $B$-structure $g$ on $\Gamma$ in the same way. Clearly $H=\beta_A(H_A) \cup \beta_B(H_B)$, so the $(A,B)$-structure is generic. In other words we have defined an object $(C_\Gamma,f,g)$ of $\X$ over $S$. This completes the definition of the functor $r$ on the objects of $\F_{A,B}$.

 Let $(\lambda_1\colon D'_A \rightarrow D_A, \lambda_2\colon C_B'\rightarrow C_B)$ be a morphism in $\F_{A,B}$ over $\lambda\colon S' \rightarrow S$. Let $C'_{\Gamma'}$ and $C_{\Gamma}$ be the curves with $\Gamma'$ and $\Gamma$-structure defined as above by $D_{A}'$, $C_{B}'$ and respectively $D_A$, $C_B$. The maps $\lambda_1$ and $\lambda_2$ together define an isomorphism of stable graph $\Gamma'\rightarrow \Gamma$ which commutes with the respective $(A,B)$-structures $(f',g')$ and $(f,g)$ of these graphs. In other words this defines an isomorphism of $(A,B)$-graphs and a map 
 \[
  (C'_{\Gamma'},f',g')\rightarrow (C_\Gamma,f,g).
 \]
 This completes the definition of the functor $r$ on the morphisms of $\F_{A,B}$.
 
 It remains to check that $r$ and $q$ are inverses of each other. Let $(C_\Gamma,f,g)$ be an object of $\X$. Then $q(C_\Gamma,f,g)=(C_A,C_B, \id_C)$. If we pass this through the above construction we see that $r(C_A,C_B, \id_C)=(C_{\Gamma'},f',g')$ where $\Gamma$ and $\Gamma'$ are isomorphic $(A,B)$-graphs. In other words $r\circ q$ is the identity on $\X$.
 
 In the other direction let $(D_A,C_B, \alpha\colon D\rightarrow C)$ be a $S$-point of $\F_{A,B}$. We have \[q(r(D_A,C_B, \alpha))= q((C_\Gamma,f,g))=(C_A,C_B,\alpha).\] An isomorphism between $D_A$ and $C_A$ is given by passing the $A$-structure through $\alpha$ as above when we define the functor $r$. It is clear that this defines an isomorphism of objects $(D_A,C_B, \alpha)\rightarrow (C_A,C_B,\alpha)$. It follows that $q\circ r$ is naturally isomorphic to the identity on $\F_{A,B}$. 
 \end{proof}

 \begin{prgrph}\label{prgrph:excess}
  Let $A$ and $B$ be $n$ pointed genus $g$ stable graphs. To identify the excess bundle $E:=p_1^*N_{\xi_A}/N_{p_2}$ on $\F_{A,B}$ of the intersection between $\M_A$ and $\M_B$ (see Paragraph \ref{pr:setup}) we can compute the excess bundle on the connected components of $\X$.  For any $(\Gamma,f,g) \in \mathfrak{G}_{A,B}$ consider the diagram
  \begin{center}
   \begin{tikzcd}
    \M_{\Gamma} \arrow[d,"\xi_{\Gamma\rightarrow A}",swap]\arrow[r,"\xi_{\Gamma \rightarrow B}"] &\M_B\arrow[d,"\xi_B"]\\
    \M_A\arrow[r,"\xi_A"] &\M_{g,n}
   \end{tikzcd}.
  \end{center}
We want to identify
\begin{align*}
E_\Gamma:= \xi_{\Gamma\rightarrow A}^*N_{\xi_A} /N_{\xi_{\Gamma\rightarrow B}}.
\end{align*}
and to compute the top Chern class $c_{\text{top}}(E_\Gamma)$.

Let $\pi\colon \M_{g,n+1} \rightarrow \M_{g,n}$ be the forgetful morphism, let $\omega_\pi$ be the dualizing sheaf and let $\sigma_i$ be the sections $\M_{g,n}\rightarrow \M_{g,n+1}$ given by the $n$ markings. Set 
\[
 \L_i = \sigma_i^* (\omega_\pi).
\]
Recall (see for example \cite[Section 13.3]{Arbarello2011}) that the normal bundle $N_{\xi_A}$ can be identified with 
\begin{align*}
 N_{\xi_A} = \bigoplus_{(h,h')\in E_A} \L^\vee_h\otimes \L^\vee_{h'},
\end{align*}
similarly 
\begin{align*}
 N_{\xi_{\Gamma \rightarrow B}} = \bigoplus_{(h,h')\in E_B\backslash \im \beta_B} \L^\vee_{h}\otimes \L^\vee_{h'}.
\end{align*}
It follows that 
\begin{align*}
 E_\Gamma = \bigoplus_{(h,h') \in \im \beta_A \cap \im \beta_B} \L^\vee_h\otimes \L^\vee_{h'}.
\end{align*}
Since $c_1(\L_h)=\psi_h$, the top Chern class of $E$ is just the product over the relevant $\psi$-classes. 

 In conclusion we get:
\end{prgrph}

\begin{proposition}\label{intbound}
 Let $A$ and $B$ be stable graphs, then 
 \begin{align*}
 \xi^*_A\xi_{B*}([\M_B]) = \sum_{\Gamma \in \mathfrak{G}_{A,B}}\xi_{\Gamma\rightarrow A*}\left( \prod_{(h,h')\in \im\beta_A \cap \im \beta_B} (-\psi_h -\psi_{h'}) \cdot [\M_\Gamma]\right)
 \end{align*}
\end{proposition}
\begin{proof}
This now follows directly from the excess intersection formula (see \cite[Proposition 17.4.1]{fulton1984}, Proposition \ref{intprop} and Paragraph \ref{prgrph:excess}.
\end{proof}

\subsection{Decorated Stratum Classes}

We will now define decorated stratum classes and compute the product of two such classes.

\begin{prgrph}
 For a given space $\M_{g,n}$ let  $\pi:\M_{g,n+1} \to \M_{g,n}$ be the map forgetting the marking $n+1$ and stabilizing, with sections $\sigma_i : \M_{g,n} \to \M_{g,n+1}$ corresponding to the markings $i$ for $i=1, \ldots,n$. Let $\omega_\pi$ be the relative dualizing sheaf.
 
 Recall the $\psi$-classes and (Arbarello-Cornalba) $\kappa$ classes on $\M_{g,n}$ are defined by
 \[\psi_i = \sigma_i^* c_1(\omega_\pi) \in A^1(\M_{g,n}),\qquad \kappa_j = \pi_* (\psi_{n+1})^{j+1} \in A^j(\M_{g,n}).\]
 
 Now let $A=(V,H,L,g,\iota,a,\zeta)$ be a stable graph and $v\in V$. Consider the projection map 
 \begin{align*}
  p_v\colon \M_{A} = \prod_{w\in V} \M_{g(w),n(w)} \rightarrow \M_{g(v),n(v)}.
 \end{align*}
 We will set $\kappa_{v,i} := p_v^*(\kappa_i)\in A^i(\M_A)$ and $\psi_{v,i} := p_v^* (\psi_{i})\in A^1(\M_A)$. 
 \end{prgrph}
 
 \begin{definition}
  A \emph{decorated stable graph} $A_\theta$  is a stable graph $A$ together with a decoration 
   \begin{align*}
  \theta = \prod_{v\in V} \left( \prod_{i\in a^{-1}(v)\cup \zeta^{-1}(v)} \psi_{v,i}^{a_i} \prod_{j=1}^m \kappa_{v,j}^{b_j} \right)\in A^\bullet (\M_A).
 \end{align*}
 \end{definition}
 
 \begin{remark}
  Restricted to each vertex  the decoration $\theta$ is just a monomial in $\psi$ and $\kappa$ classes. Given another decoration $\theta'$ the product $\theta\cdot \theta'$ in $A^\bullet(\M_A)$ is given by
  \[
   \theta \cdot \theta' = \prod_{v\in V} \left( \prod_{i\in a^{-1}(v)\cup \zeta^{-1}(v)} \psi_{v,i}^{a_i+a'_i} \prod_{j=1}^m \kappa_{v,j}^{b_j+b'_j} \right)
  \]
 \end{remark}


 \begin{notation}\label{pr:decboun}
     We set
  \[
   [A_\theta]:=\frac{1}{|\aut A|}\xi_{A*}(\theta)\in A^\bullet(\M_{g,n})
  \]
  and will call $[A_\theta]$ a \emph{decorated stratum class}. If $\theta=1$ we will drop it from the notation.

We will draw decorated stratum classes by attaching a number of arrowheads to each half edge or leg $i$ equal to $a_i$ and by attaching the monomial
\[
 \prod_{j=1}^m \kappa_{v,j}^{b_j}
\]
 to each vertex $v$.
 \end{notation}

 \begin{example}
 Let $A$ be the graph 
     \begin{tikzpicture}[->,>=bad to,baseline=-3pt,node distance=1.3cm,thick,main node/.style={circle,draw,font=\Large,scale=0.5}]
\node [main node] (A) {3};
\node[main node] (B) [left of=A] {2};
\draw [-] (A) to (B);
\end{tikzpicture}.
If $v_1$ is the vertex of genus 2 and $v_2$ is the vertex of genus 3 and we have a decoration $\theta = \psi^2_{v_2,h} \kappa_{v_1,1}^2$ we will draw the decorated graph $A_\theta$ as 
  \begin{center}
    \begin{tikzpicture}[->,>=bad to,baseline=-3pt,node distance=1.3cm,thick,main node/.style={circle,draw,font=\Large,scale=0.5}]
\node [main node] (A) {3};
\node[main node] (B) [left of=A] {2};
\node [above left] at (B) {$\kappa^2_1$};
\draw [<<-] (A) to (B);
\end{tikzpicture}.
  \end{center}
 \end{example}

 \begin{warning}
  There are two conflicting notations in the literature for decorated stratum classes. One is as given in \ref{pr:decboun} the other one is without dividing by the size of the automorphism group of $A$. 
 
  The advantage to our definition is that the class $[A]$ is the Poincar\'{e} dual of the closure of the locus of all stable curves with dual graph isomorphic to $A$. In particular $[A]$ corresponds to an actual closed integral substack of $\M_{g,n}$. The advantage of not dividing by the order of the automorphism group is that it makes calculations slightly easier. 
 \end{warning}

\begin{remark}
 The codimension (or degree) of $\psi_i$ in $A^\bullet(\M_{g,n})$ is 1, the codimension of $\kappa_j$ is~$j$ and the codimension of $[A]\in A^\bullet(\M_{g,n})$ equals the number of edges of $A$. Therefore if $A_\theta$ is a decorated boundary graph with decoration
 \[
  \theta = \prod_{v\in V} \left( \prod_{i\in a^{-1}(v)\cup \zeta^{-1}(v)} \psi_{v,i}^{a_i} \prod_{j=1}^m \kappa_{v,j}^{b_j} \right)
 \]
then  
\[
 \codim [A_\theta] = \# E_A + \sum_{i\in H \cup L} a_i + \sum_{v\in V,j} j b_{v,j}.
 \]
\end{remark}

  To determine the product $[A_\theta]\cdot [B_\lambda]$ for two decorated stable graphs $A_\theta$, $B_\lambda$ we need to know the pullback of $\theta$ under $\xi_{\Gamma \rightarrow A}$. Since the pullback is a ring homeomorphism this can be done by pulling back the individual $\psi$ and $\kappa$ classes in $\theta$.

\begin{lemma}\label{lem:pullpsikappa}
 Let $f=(\alpha,\beta,\gamma)\colon \Gamma \rightarrow A$ be a map of stable graphs, then
 \begin{align*}
 \xi^*_{f\colon\Gamma \rightarrow A} (\psi_{v,i}) &= \psi_{a_\Gamma\circ \beta(i),\beta(i)}, \\
  \xi^*_{f\colon \Gamma \rightarrow A} (\kappa_{v,i}) &= \sum_{w\in \alpha^{-1}(v)} \kappa_{w,i}.
 \end{align*}
\end{lemma}

\begin{proof}
 The first of these is trivial, the second is  \cite[Lemma 4.31]{Arbarello2011}.
\end{proof}

\begin{corollary}\label{cor:pulltheta}
 Let $A$ and $\Gamma$ be stable graphs, let $f$ be an $A$-structure on $\Gamma$ and let $\theta$ be a decoration on $A$. We have
 \[
   \xi^*_{f}(\theta) = \prod_{v\in V_A} \left( \prod_{i=1}^{n(v)} \psi_{a_\Gamma \circ \beta_f(i),\beta_f (i)}^{a_i} \prod_{j=1}^m (\sum_{w\in \alpha_f^{-1}(v)} \kappa_{w,j})^{b_j} \right)
 \]
\end{corollary}

 \begin{theorem}\label{th:prod}
 Let $A_\theta$ and $B_\lambda$ be decorated stable $n$ pointed genus $g$ graphs. Then
 \begin{align*}
 [A_\theta]\cdot [B_\lambda] =\\
 \frac{1}{|\aut A| \cdot |\aut B|}  & \sum_{(\Gamma,f,g) \in \mathfrak{G}_{A,B}}\xi_{\Gamma*} \left(\xi^*_{f}(\theta) \cdot \xi^*_g(\lambda)  \prod_{(h,h')\in \im\beta_A \cap \im \beta_B} (-\psi_h -\psi_{h'}) \right).
 \end{align*}
 \end{theorem}
 
 \begin{proof}
  This follows by pushing forward the expression of Proposition \ref{intbound}.
 \end{proof}

\begin{remark} \label{Rmk:intnumb}
 If $A_\theta$ and $B_\lambda$ define classes in complementary degrees for $\M_{g,n}$, i.e. $[A_\theta]\cdot[B_\lambda]$ is a zero cycle, we can compute the degree of this zero cycle using Theorem \ref{th:prod}. Indeed, the formula above reduces this to the question of computing the intersection numbers of $\kappa$ and $\psi$-classes on the factors $\M_{g(v),n(v)}$ of the spaces $M_\Gamma$. These intersection numbers are governed by the KdV hierarchy as shown by Kontsevich \cite{kontsevich1992} after a conjecture by Witten. Thus in principle, we are able to compute intersection numbers of decorated strata classes.
\end{remark}
 
\begin{application} \label{app:tautkunneth}
 As an application of the intersection numbers above, we remark that for some $g,n$ it is possible to express the class of the diagonal $[\Delta] = [\Delta_{\M_{g,n}}] \in H^{6g-6+2n}(\M_{g,n} \times \M_{g,n})$ in terms of tautological classes. Indeed, assume that $H^*(\M_{g,n})=RH^*(\M_{g,n})$ and let $(e_i)_i$ be a basis of the tautological ring as a $\mathbb{Q}$-vector space. By Remark \ref{Rmk:intnumb} we can compute the pairing matrix $\eta_{i,j}=\deg(e_i \cdot e_j)$. Let $(\eta^{i,j})_{i,j}$ be the inverse matrix. Then we have a Kunneth decomposition of the diagonal
 \[[\Delta] = \sum_{i,j} \eta^{i,j} e_i \otimes e_j \in H^{*}(\M_{g,n} \times \M_{g,n}) = RH^*(\M_{g,n}) \otimes RH^*(\M_{g,n}).\]
 The assumption that all cohomology is tautological is for instance satisfied for all spaces $\M_{0,n}$ (see \cite{Keel}) and $\M_{1,n}$ for $n \leq 10$ (see \cite[Proposition 2]{Graber2001} and \cite{petersen2014}). It is also known for a few cases with $g \geq 2$ where we know that all even cohomology is tautological and where the ranks of these even cohomology groups sum up to the Euler characteristic of $\M_{g,n}$, implying the vanishing of the odd cohomology.
 
 We are going to see an application of these tautological Kunneth decompositions in Section \ref{sec:lowergenus}.
\end{application}

\section{Stacks of Pointed Admissible \texorpdfstring{$G$}{G}-covers}\label{sec:hurspace}
In the following section we introduce the stacks $\H_{g,G,\xi}$ parametrizing finite covers $\varphi: C \to D$ of (connected) curves together with a $G$-action on $C$, such that $C$ has genus $g$ and  such that $\varphi$ is the quotient map under the $G$-action. Further included is the data of a tuple of smooth points $p_{i,a} \in C$, containing all ramification points of $\varphi$, such that the behaviour of the $G$-action at these points is described by a monodromy datum $\xi$. 

This problem has been studied in a number of contexts. In \cite{Harris1982} Harris and Mumford first introduced the notion of an admissible cover, a technical condition we need to require if the curves $C,D$ above become nodal. This allowed them to write down a compact moduli space for such covers. However, they did not require the data of a $G$-action on $C$ but considered all covers $\varphi$ of a given degree.

The paper \cite{Abramovich2001} by Abramovich, Corti and Vistoli instead notes that if we take the quotient stack $[C/G]$ instead of the quotient $D=C/G$ in the realm of schemes, the data of the quotient morphism $\varphi: C \to [C/G]$ is equivalent to a stable map $[C/G] \to BG$ of the (stacky) curve $[C/G]$ into the classifying stack $BG$ of the finite group $G$. In \cite{Jarvis2005}, Jarvis, Kaufmann and Kimura take a variant of this construction. The crucial difference is that they consider the data of marked points on $C$, not on $[C/G]$. This is also the convention of our paper and it will be the case that our space $\H_{g,G,\xi}$ is a union of connected components of their moduli space. The only further restriction we impose is that the domain $C$ of the cover is a connected curve.

Finally, there is the book \cite{BertinRomagny} by Bertin and Romagny, which focuses on the data of the $G$-action on the curve $C$. This action of course determines the data of the quotient map $\varphi: C \to C/G$. This perspective and their results will be useful, since later in our paper we are interested in the map $\phi : \H_{g,G,\xi} \to \M_{g,r}$ which forgets the data of the group action and cover and just remembers the curve $C$ with the markings $p_{i,a}$. Here is where we crucially use the connectedness of $C$, since $C$ with the marked points $p_{i,a}$ is a stable curve.

Our goal in the following is to give a self-contained introduction to the theory of admissible $G$-covers and to define the stack $\H_{g,G,\xi}$. We prove properties and functorialities of this stack which we will use later, often citing relevant results from the literature. We comment on the detailed relation between our convention and those in the references above in Remark \ref{Rmk:relntoLit}. 



\subsection{Group Actions on Smooth Curves}
As a warm-up we start in the situation of a finite group $G$ acting effectively on a smooth curve $C$, i.e.\  via an injective group homomorphism $G \hookrightarrow \aut(C)$. Then the quotient $D=C/G$ exists (as a scheme) and it is itself a smooth curve. The corresponding quotient map $\varphi: C \to D$ is a \emph{Galois cover}, i.e.\  the group of automorphisms of $C$ over $D$ acts transitively on the fibres and is in fact  equal to $G$. 

A general point of $D$ has exactly $|G|$ preimages under $\varphi$, corresponding to the fact that the general point of $C$ has trivial stabilizer in $G$. The ramification points $p$ of $\varphi$ are exactly those points which have a nontrivial stabilizer $H=G_p$. We recall  the following result about the local action of $H$ around $p$. 
\begin{lemma} \label{Lem:stabcyclic}
 Let the finite group $G$ act on the smooth curve $C$ effectively and let $p \in C$ be a point. Then the stabilizer $H=G_p$ of $p$ is a cyclic group. The induced action of a generator of $H$ on $T_p C \cong \mathbb{C}$ is given by multiplication with a primitive $e$-th root of unity, where $e=|H|$.
\end{lemma}
\begin{proof}
 The derivative of the action of $G_p$ at the point $p$ induces a linear action $\theta: G_p \to \mathrm{GL}(T_p C) = \mathbb{C}^*$. We claim that $\theta$ is injective. Once we show this, the Lemma is proven, since any finite subgroup of $\mathbb{C}^*$ of order $e$ is a subgroup of the $e$th roots of unity and thus must be equal to this group.
 
 Now assume $\theta(h)=1$ for $h \in G_p$, then in local coordinates $z$ around $p$ the action $\sigma_h : C \to C$ of $h$ on $C$ looks like
 \[\sigma_h(z) = z + a_2 z^2 + a_3 z^3 + \ldots .\]
 By composing this power series with itself, it is easy to see that for all $m \geq 0$ the $m$-fold composition of $\sigma_h$ looks like
 \[(\sigma_h)^{\circ m}(z) = \sigma_{h^m}(z) = z + (m a_2) z^2 + O(z^3).\]
 For $m=\mathrm{ord}_G(h)$ we have $\sigma_{h^m}(z)=z$, so it follows $m a_2 = 0$, hence $a_2=0$. With similar arguments one shows $a_i=0$ for all $i \geq 2$, i.e.\  $\sigma_h = \mathrm{id}_C$. Since the action is effective, this implies $h=e_G$, so $\theta$ is injective as claimed.
\end{proof}
It follows that there exists a distinguished generator $h$ of $H$ acting by the root $\zeta_e = \exp(2 \pi i/e)$. We call $h$ the \emph{monodromy of $G$ at $p$}.

Now note that for $q=\varphi(p)$ we have a bijection
\[G/H \to \varphi^{-1}(q), gH \mapsto gp.\]
If we had chosen a different point $p'=tp \in \varphi^{-1}(q)$, $t \in G$, we would have obtained a stabilizer $G_{p'}=t H t^{-1}$ with the distinguished generator $h'=t h t^{-1}$. Overall we see that to a branch point $q$ of $\varphi$ we can uniquely associate a conjugacy class $[h]$ of $G$. If $q_1, \ldots, q_b \in D$ are the branch points of $\varphi$, we can write a formal sum
\[[h_1] + \ldots + [h_b]\]
of the corresponding conjugacy classes. In \cite{BertinRomagny} this is called the Hurwitz datum associated to $\varphi$. By making this a formal sum and by having conjugacy classes $[h_i]$ instead of group elements, this is independent of a choice of ordering of branch points $q_i$ as well as a distinguished point in the preimage $\varphi^{-1}(q_i)$. 

For our purposes it will be more convenient to include this ordering and the choice of a preimage. For us a \emph{monodromy datum} will be an ordered tuple \[\xi = (h_1, \ldots, h_b)\]
of elements $h_i \in G$.

We immediately note that formally it makes sense to include entries of the form $e_G$ in $\xi$, where $e_G \in G$ is the neutral element. This corresponds to tracking a point $q \in C$ such that $G$ acts freely on the fibre $\varphi^{-1}(q)$. 


Knowing the full branching behaviour of $\varphi$ it is easy to write down the Riemann-Hurwitz formula 
\begin{equation} \label{eqn:RiemHurw}
 2g-2 = |G| (2g'-2) + \sum_{i=1}^b \left(|G| - \frac{|G|}{\mathrm{ord}_G(h_i)} \right).
\end{equation}
determining the genus $g'$ of $D$.

\subsection{Admissible Covers}
As a next step we need to understand what happens when the curves $C$ and $D$ degenerate, obtaining nodal singularities. In this case, we must require the map $\varphi: C \to D$ to be an \emph{admissible cover}. In order to obtain a suitable stack, it will be natural to define the notion of a family of admissible covers parametrized by a scheme $S$. The following definition is essentially the same as in \cite[Definition 2.1]{Jarvis2005}, with the crucial difference that we require the curve $C$ to be connected.

\begin{definition} \label{def:admGcover}
An \emph{admissible $G$-cover} over a scheme $S$ is a map $\varphi: C \to D$ between connected, nodal curves over $S$ together with an action $G \curvearrowright C$ of $G$ on $C$ and markings $q_1, \ldots, q_b : S \to D$, such that
\begin{itemize}
 \item $(D,q_1, \ldots, q_b)$ is a stable curve,
 \item $\varphi$ is finite, mapping every node of $C$ to a node of $D$,
 \item the action of $G$ preserves $\varphi$ and the restriction of $\varphi$ over the set $D_{\text{gen}}$ - of unmarked, smooth points of $D$ - is a principal $G$-bundle,
 \item the local picture of $C \to D \to S$ at a point of $C$ over a node of $D$ is (analytically isomorphic to) that of 
 \[\mathrm{Spec} A[\zeta,\eta]/(\zeta \eta - a) \to \mathrm{Spec} A[x,y]/(xy-a^e) \to \mathrm{Spec} A,\]
 at $(\zeta, \eta)=(0,0)$ for some $e \geq 1$, where $\varphi^* x = \zeta^e$ and $\varphi^* y = \eta^e$,
 \item the local picture of $C \to D \to S$ at a point of $C$ over a marking of $D$ is (analytically isomorphic to) that of 
 \[\mathrm{Spec} A[\zeta] \to \mathrm{Spec} A[x] \to \mathrm{Spec} A,\]
 at $\zeta=0$ for some $e \geq 1$, where $\varphi^* x = \zeta^e$,
 \item for each geometric nodal point $P \in C$ with image point $s \in S$, the stabilizer $G_P$ is cyclic and its action on the fiber $C_s$ is \emph{balanced}, i.e.\  the local picture around $P$ looks like $\mathrm{Spec} A[\zeta,\eta]/(\zeta \eta - a)$ and for a generator $h$ of $G_P$ the action is given by
    \[h . (\zeta, \eta) = (\mu  \zeta, \mu^{-1} \eta)\]
    with a primitive root $\mu$ of unity of order equal to the order of $h$.
\end{itemize}
\end{definition}
Given an admissible $G$-cover, it is natural to place markings on $C$ at the preimages of the markings $q_i$ of the curve $D$. While in \cite{Jarvis2005} exactly one marking $p_i$ is chosen in the preimage of every $q_i$, it will be more convenient for us to have individual markings at \emph{all} preimage points of $q_i$. Note that these two data are equivalent: given any marking $p_i \in \varphi^{-1}(q_i)$ with stabilizer $\langle h_i \rangle$ under $G$, the remaining points in $\varphi^{-1}(q_i)$ are given by the $G$-orbit of $p_i$ and naturally bijective to $G/\langle h_i \rangle$ via $a \mapsto a p_i$. Below we denote the markings $a p_i$ by $p_{i,a}$.
\begin{definition} \label{def:pointedadmGcover}
A \emph{pointed admissible $G$-cover} with \emph{monodromy data $\xi=(h_1, \ldots, h_b) \in G^b$} over a scheme $S$ is an admissible $G$-cover 
\[\varphi: G \curvearrowright C \to (D,q_1, \ldots, q_b) \to S\]
together with sections
\[p_{i,a} : S \to C, i=1, \ldots, b, a \in G/\langle h_i \rangle,\]
such that
\begin{itemize}
 \item $\varphi^{-1}(q_i) = \{p_{i,a} : a \in G/\langle h_i \rangle \}$ for $i=1, \ldots, b$ and for $h \in G$ we have $h p_{i,a} = p_{i,ha}$,
 \item the monodromy of the $G$-action at the marking $p_i=p_{i,[e_G]}$ corresponding to the neutral element $e_G \in G$ is given by $h_i$ for all $i$, i.e.\  the generator $h_i$ of the stabilizer $G_{p_i}$ of $p_i$ acts on $T_{p_i} C$ by multiplication with the root $\zeta_e = \exp(2 \pi i/e)$ of unity, where $e=\operatorname{ord}_G(h_i)$.
\end{itemize}
\end{definition}
The reason we mark the other preimages of $q_i$ as well is the following.
\begin{lemma}
 Given a pointed admissible $G$-cover $\varphi: G \curvearrowright (C, (p_{i,a})_{i,a}) \to (D,q_1, \ldots, q_b)$, the curve $(C, (p_{i,a})_{i,a})$ is stable.
\end{lemma}
\begin{proof}
This is essentially clear: a rational component $C'$ of $C$ must map to a rational component $D'$ of $D$. Then $D'$ contains at least $3$ special points and each has some preimage in $C'$ which is again a special point, so $C'$ is stable.
\end{proof}

\begin{definition}
For a genus $g \geq 0$, a finite group $G$ and a monodromy datum $\xi=(h_1, \ldots, h_n)$ with elements $h_1, \ldots, h_n \in G$, 
define the stack $\H_{g,G,\xi}$ whose objects over a scheme $S$ are pointed admissible covers
\begin{equation} \label{eqn:fullvarphi} \varphi: G \curvearrowright (C,(p_{i,a})_{\substack{i=1, \ldots, b\\ a \in G/\langle h_i \rangle}}) \to (D,q_1, \ldots, q_b) \to S.\end{equation}
Isomorphisms between such objects are $G$-equivariant isomorphisms of the curves $C$ fixing the markings.

Define the stack $\Ho_{g,G,\xi}$ as the open substack of $\H_{g,G,\xi}$ where the curve $C$ (and $D$) is smooth.
\end{definition}

In the future we will abbreviate the notation in (\ref{eqn:fullvarphi}) to $(C \to D, (p_{i,a})_{i,a}, (q_i)_i)$ or even $(C \to D)$, with the understanding that the data of the group action on $C$ and the markings is implicit.

\begin{remark} \label{Rmk:relntoLit}
 We want to outline here how the space $\H_{g,G,\xi}$ compares to other constructions in the literature.
 
 As remarked before, our definition is closest to the one from \cite{Jarvis2005}. Given $g,G,\xi=(h_1, \ldots, h_b)$ denote by $g'$ the genus of the curve $D$, which can be computed via the  Riemann-Hurwitz formula (\ref{eqn:RiemHurw}). Then the space $\H_{g,G,\xi}$ is (canonically isomorphic to) the union of the components of the space $\M_{g',b}^G(h_1, \ldots, h_b)$ in \cite{Jarvis2005} where the domain $C$ of the admissible cover $\varphi: C \to D$ is connected. 
 
 On the other hand, in \cite[Definition 6.2.3.]{BertinRomagny} Bertin and Romagny also define a stack $\H_{g,G,[h_1]+ \ldots + [h_b]}$, which we denote by $\H^{\text{BR}}_{g,G,\xi}$ for clarity. Instead of pointed admissible $G$-covers, it just parametrizes admissible $G$-covers $\varphi: G \curvearrowright C \to (D,q_1, \ldots, q_b) \to S$ such that for all $i$ some preimage $p_i$ of $q_i$ has monodromy given by $h_i$. In other words there is no distinguished preimage of $q_i$. By comparison of definitions, one sees that this has a relation to the space $\mathcalorig{A}_{g',b}^{\mathbf{d},\mathbf{e}}$ defined in \cite[Appendix B.2.]{Abramovich2001} as
 \[\mathcalorig{A}_{g',b}^{\mathbf{d},\mathbf{e}} = \bigcup_{\xi} \H^{\text{BR}}_{g,G,\xi} .\]
 Here the union on the right goes over all monodromy data $\xi$ giving rise to the ramification multiplicities specified by the vectors $\mathbf{d},\mathbf{e}$ of positive integers. By the results from the Appendix and Theorem 4.3.2. of \cite{Abramovich2001} there exist natural isomorphisms and inclusions as open and closed substacks
 \[\mathcalorig{A}_{g',b}^{\mathbf{d},\mathbf{e}} \cong \mathcalorig{A}dm_{g',b}^{d,e}(G) \subset \mathcalorig{A}dm_{g',b}(G) \cong \mathcalorig{B}_{g',b}^{\text{bal}}(G) = \M_{g',b}(BG).\]
 The stack on the right is the stack of balanced twisted stable maps to the classifying stack $BG$ of the finite group $G$. The basic idea how to go from $\H^{\text{BR}}_{g,G,\xi}$ to $\M_{g',b}(BG)$ is that given a curve $C \to \mathrm{Spec}(\mathbb{C})$ with action by $G$ we can take the quotient of $G$ on both sides (acting trivially on $\mathrm{Spec}(\mathbb{C})$) and obtain a morphism $[C/G] \to BG$. The inverse operation is taking the fibre product via the universal torsor $\mathrm{Spec}(\mathbb{C}) \to BG$.
 
 From the proof of \cite[Theorem 2.4]{Jarvis2005} it follows that the canonical map $\H_{g,G,\xi} \to \H^{\text{BR}}_{g,G,\xi}$, forgetting the distinguished markings $p_{i,a}$ in $C$ and only remembering the markings $q_i$ in $D$, is \'etale, since it is a union of components of an iterated fibre product of \'etale morphisms to $\H^{\text{BR}}_{g,G,\xi}$.
\end{remark}

\subsection{Properties of \texorpdfstring{$\H_{g,G,\xi}$}{Hg,G,xi}}
In this section we collect results about the stacks $\H_{g,G,\xi}$ we are going to use later. We recall that $g \geq 0$ is the genus of the domain curve $C$ of the admissible covers, $G$ a finite group, $\xi=(h_1, \ldots, h_b)$ a monodromy datum of elements $h_i \in G$. We require that the number $g'$ determined by the formula (\ref{eqn:RiemHurw}) is a nonnegative integer, giving the genus of the curve $D$ in the admissible cover $\varphi: C \to D$. Let furthermore $n=\sum_{i=1}^b |G|/ \mathrm{ord}_G(h_i)$ be the total number of markings $p_{i,a}$ in $C$. 

\begin{theorem}\label{Thm:Hurwitzstack}
 The stack $\H_{g,G,\xi}$ is a smooth proper Deligne-Mumford stack of dimension $3g'-3+b$ over $\mathbb{C}$. 
 There exist natural maps 
 \begin{equation} \label{eqn:ideltadiagram}
  \begin{tikzcd}
   \H_{g,G,\xi} \arrow[r,"\phi"] \arrow[d,"\delta"] & \M_{g,r}\\
   \M_{g',b} & ~
  \end{tikzcd}
 \end{equation}
 defined by
 \begin{align*}
  \phi((C \to D, (p_{i,a})_{i,a}, (q_i)_i)) &= (C, (p_{i,a})_{i,a}),\\
  \delta((C \to D, (p_{i,a})_{i,a}, (q_i)_i)) &= (D, (q_i)_i).
 \end{align*}
 The morphism $\phi$ is representable, finite, unramified and a local complete intersection.
 The morphism $\delta$ is flat, proper and quasi-finite, however not necessarily representable. 
\end{theorem}
\begin{proof}
Since $\H_{g,G,\xi}$ is a union of components of $\M_{g',b}^G(\xi)$ from the paper by Jarvis, Kaufmann and Kimura, the fact that $\H_{g,G,\xi}$ is a smooth, proper  Deligne-Mumford stack together with the properties of $\delta$ follow from \cite[Theorem 2.4]{Jarvis2005}. The fact that $\delta$ is not necessarily representable comes from the possibility of having $G$-equivariant automorphisms of $(C,(p_{i,a})_{i,a})$ over the identity map on $D$ (see the discussion preceding Theorem \ref{thm:condnonempty} for details).

The proof of the properties of $\phi$ is exactly analogous to the proof of  \cite[Proposition 6.5.2.]{BertinRomagny}. 

Indeed, assume we are given an $S$-point $S \to \M_{g,r}$ for a scheme $S$, given by a family of stable curves $(C,p_1, \ldots, p_n) \to S$. Then we first want to show that the pullback of $\phi$ given as
\[\mathcal{F} := \H_{g,G,\xi} \times_{\M_{g,r}} S \to S\]
is finite and unramified, with $\mathcal{F}$ an algebraic space. For $T \to S$ a morphism of schemes, giving a $T$-point of $\mathcal{F}$  over $S$ means specifying a suitable $G$-action on $C_T := C \times_S T$. Once this $G$-action is specified, we obtain the admissible cover as $C_T \mapsto D=C_T/G$. To express the $G$-action, let
\[\mathcal{A} = \{\eta  \in \aut_S(C) : \eta(\{p_1, \ldots, p_n\}) \subset \{p_1, \ldots, p_n\}\}\]
be the group of automorphisms of $C$ over $S$ leaving invariant the \emph{set} of marked points. Then $\mathcal{A} \to S$ is finite and unramified (corresponding to the fact that the quotient $\M_{g,r} / S_n$ is still a separated Deligne-Mumford stack).




We can specify a $G$-action on $C_T$ by giving a $T$-point of the $|G|$-fold product
\[\mathcal{A}_G=\mathcal{A} \times_S \mathcal{A} \times_S \cdots \times_S \mathcal{A}\]
satisfying a number of closed conditions (defining a $G$-action, permuting the markings $p_i$ in the correct way, with the correct local monodromies, etc.). Thus $\mathcal{F}$ is represented by a closed subscheme of $\mathcal{A}_G$ so it is finite and unramified over $S$. We conclude that $\phi$ is representable, finite and unramified. Since $\H_{g,G,\xi}$ and $\M_{g,r}$ are smooth, it is also a local complete intersection.
\end{proof}

We call $\phi$ the \emph{source morphism} and $\delta$ the \emph{target morphism}, since they remember the source and target of the admissible cover, respectively.

\begin{definition} \label{Def:Hbar}
 We denote by $[\H_{g,G,\xi}] \in A_{3g'-3+b}(\M_{g,r})$ the pushforward of the fundamental class of $\H_{g,G,\xi}$ under the map $\phi$.
\end{definition}
These cycles and their intersections with tautological classes are the main object of study of the present paper.

Over the stack $\H_{g,G,\xi}$ we have universal curves $\mathcal{C}, \mathcal{D}$ with sections $\mathcal{p}_{i,a}, \mathcal{q}_i$ and a universal admissible $G$-cover
\begin{equation} \label{eqn:univGalois}
  \begin{tikzcd}
   \mathcal{C} \arrow[r,"\varphi"] \arrow[d,"\pi"] & \mathcal{D}\arrow[dl,"\tilde \pi"]\\
   \H_{g,G,\xi} \arrow[u,bend left,"\mathcal{p}_{i,a}"] \arrow[ur, bend right,swap,"\mathcal{q}_i"] & ~
  \end{tikzcd}.
 \end{equation}
The pointed families of curves $\mathcal{C}, \mathcal{D}$ induce the maps $\phi,\delta$ from Theorem \ref{Thm:Hurwitzstack}. There is a natural isomorphism $\mathcal{C} \cong \H_{g,G,(\xi,e_G)}$, see \cite[Section 6.5.2]{BertinRomagny}.

Given one of the markings $p_{i,a}$ on $\H_{g,G,\xi}$ we define $\psi_{p_{i,a}} \in A^1(\H_{g,G,\xi})$ as the pullback of the corresponding $\psi$-class on $\M_{g,r}$ under $\phi$. This is equivalent to defining it as $c_1(\mathcal{p}_{i,a}^* \omega_\pi)$ by the functoriality of the relative dualizing sheaf.

Similarly, given $\ell \geq 1$ we define $\kappa_\ell \in A^\ell(\H_{g,G,\xi})$ as the pullback of $\kappa_\ell \in A^\ell(\M_{g,r})$ under $\phi$. For notational distinction we will write the $\psi$ and $\kappa$ classes on $\M_{g',b}$ as $\psi'_{q_i}$ and $\kappa'_{\ell}$. Then we have the following convenient comparison result.

\begin{lemma} \label{Lem:psikappacompat}
 Assume the stabilizer of $p_{i,a}$ has order $e_i=\mathrm{ord}_G(h_i)$, then we have
 \begin{equation} \label{eqn:psicompat}
  \psi_{p_{i,a}} = \frac{1}{e_i} \delta^*(\psi'_{q_i}) \in A^1(\H_{g,G,\xi}).
 \end{equation}
 On the other hand, for $\ell \geq 1$ we have
 \begin{equation} \label{eqn:kappacompat}
  \kappa_\ell = |G| \delta^*(\kappa'_{\ell}) \in A^\ell(\H_{g,G,\xi}).
 \end{equation}
\end{lemma}
\begin{proof}
 To prove (\ref{eqn:psicompat}) consider the universal admissible cover (\ref{eqn:univGalois}). Then we have
 \begin{align*}
  \delta^*(\psi'_{q_i}) &= \mathcal{q}_i^*(c_1(\omega_{\tilde \pi})) =(\varphi \circ \mathcal{p}_{i,a})^*(c_1(\omega_{\tilde \pi})) = \mathcal{p}_{i,a}^*(\varphi^* c_1(\omega_{\tilde \pi}))
 \end{align*}
 Since $\varphi$ is ramified along the sections $\mathcal{p}_{j,a}$ with ramification index $e_j-1$, we have
 \begin{equation}
  \omega_\pi = \varphi^* \omega_{\tilde \pi} \otimes \left(\otimes_{j} \mathcal{O}_{\mathcal{C}}(R_j)^{\otimes e_j-1} \right),
 \end{equation}
 where $R_j$ is the union of the images of the sections $\mathcal{p}_{j,a}$. The equation above just says that $\varphi$ is ramified along the divisors $R_j$ with ramification index $e_j-1$. Inserting in our computation above and using the sections $\mathcal{p}_{j,a'}$ are all disjoint from $\mathcal{p}_{i,a}$, except $\mathcal{p}_{i,a}$ itself, we obtain 
 \begin{align*}
  \delta^*(\psi'_{q_i}) &= \mathcal{p}_{i,a}^*\left(c_1(\omega_\pi) - (e_i-1) (\mathcal{p}_{i,a})_* [\H_{g,G,\xi}]\right) = \psi_{q_{i,a}} + (e_i-1)\psi_{q_{i,a}} = e_i \psi_{q_{i,a}}.
 \end{align*} 
 The statement about $\kappa$ classes is proved exactly as in \cite[Th\'eor\`{e}me 10.3.4.]{BertinRomagny}.
\end{proof}

\subsection{Boundary Strata of \texorpdfstring{$\H_{g,G,\xi}$}{Hg,G,xi} and Admissible \texorpdfstring{$G$}{G}-Graphs} \label{Sect:boundaryHurwitz}
The moduli space $\M_{g,n}$ has a stratification indexed by stable graphs $\Gamma$. Similarly, the space $\H_{g,G,\xi}$ has a stratification indexed by stable graphs $\Gamma$ together with a $G$-action. This is described in detail in Chapter 7 of \cite{BertinRomagny}. We summarize the results in our notation.

Given $(C \to D, (p_{i,a})_{i,a}, (q_i)_i) \in \H_{g,G,\xi}$ let $\Gamma$ be the dual graph of $C$. Since the action of $G$ on $C$ permutes the irreducible components and nodes of $C$, respectively, we obtain a group action $G \curvearrowright \Gamma$ on the stable graph. In particular, for each leg \emph{or} half-edge $l$ we can write down its stabilizer $G_l$, which is a cyclic group.

We need to record some finer data describing the action of $G_l$ on $C$ around the point corresponding to $l$, namely the monodromy of the $G$-action there. Let $v \in V(\Gamma)$ be the vertex adjacent to $l$, let $C_v$ be the normalization of the component of $C$ corresponding to $v$ and let $x_l \in C_v$ be the  point corresponding to the leg or half-edge $l$. If $l$ is a leg this is just the position of the corresponding marking, if $l$ is a half-edge this is the corresponding preimage of the node. Elements in the stabilizer $G_l$ preserve the vertex $v$ and correspondingly there is an action $G_l \curvearrowright C_v$ fixing $x_l$. For $e=|G_l|$ there is then a unique element $h_l \in G_l$ acting as multiplication by $\zeta_e = \exp(2 \pi i/e)$ on the tangent space $T_{x_l} C_v$.

\begin{definition} \label{Def:Hstablegraph}
 We call the data
 \begin{equation} \label{eqn:Hstablegraph} (\Gamma,G)=(G \curvearrowright \Gamma, (h_l)_{l \in L(\Gamma) \cup H(\Gamma)})\end{equation}
 the \emph{admissible $G$-graph} associated to $(C \to D, (p_{i,a})_{i,a}, (q_i)_i)$. 
\end{definition}
Note that from the balancing condition of the group action on $C$ it follows that for two half-edges $a,b$ forming an edge we have $h_a=h_b^{-1}$. Moreover, there is no element $t \in G$ preserving the edge but flipping the orientation, i.e.\  with $ta=b$. Such an element would correspond to a stabilizer of a node which exchanges the branches of the node, contradicting the local picture of the action around nodes from Definition \ref{def:admGcover}.


We denote by $\Ho_{g,G,\xi}(\Gamma,G)$ the set of elements in $\H_{g,G,\xi}$ with associated admissible $G$-graph $(\Gamma,G)$ and by $\H_{g,G,\xi}(\Gamma,G)$ its closure. This defines a stratification of $\H_{g,G,\xi}$ similar to the stratification of $\M_{g,n}$ by boundary strata. In the following we want to derive a parametrization of $\H_{g,G,\xi}(\Gamma,G)$ by a generalization of the usual gluing maps associated to the graph $\Gamma$.

Since the general case of these equivariant gluing maps has a fairly technical description, let us first look at a concrete example.
\begin{example}
The space $\H_{2,\mathbb{Z}/2\mathbb{Z},(1,1)}$ parametrizes double covers $C \to E$ from genus $2$ curves $C$ to genus $1$ curves $E$, ramified at two points $p_1, p_2 \in C$. One boundary stratum of this space is given by the locus of curves
\begin{equation} \label{eqn:exaCactionglue} (C,p_1,p_2) = (C_1,q) \cup_{q=q'} (C_2,q',p_1,p_2,q'') \cup_{q''=q} (C_1,q),\end{equation}
where $C_1$ has genus $1$, $C_2$ has genus $0$ and $C_2$ carries a $G=\mathbb{Z}/2\mathbb{Z}$-action leaving $p_1, p_2$ fixed and exchanging $q',q''$. This action extends to all of $C$, where the first component $C_1$ is exchanged with the second one.

The corresponding $G$-admissible stable graph $(\Gamma,G) $ is given by
\[(\Gamma,G)=
\begin{tikzpicture}[->,>=bad to,baseline=-3pt,node distance=1.3cm,thick,main node/.style={circle,draw,font=\Large,scale=0.5}]
\node at (1.4,0) [main node] (B) {0};
\node at (0,.5) [main node] (A1) {1};
\node at (0,-.5) [main node] (A2) {1};
\draw [-] (A1) to (B);
\draw [-] (A2) to (B);
\draw [-] (B) to (1.9,-.1);
\draw [-] (B) to (1.9,.1);
\draw [<->] (.5,-.2) to (.5,.2);
\end{tikzpicture} \]
Then we see a point in the closure $\H_{g,G,\xi}(\Gamma,G)$ can be specified by giving $(C_1,q) \in \M_{1,1} = \H_{1,\{0\},(0)}$ and $(C_1, p_1,p_2,q',q'') \in \H_{0,G,(1,1,0)}$. This means that the corresponding equivariant gluing map \[\xi_{(\Gamma,G)} : \M_{1,1} \times \H_{0,G,(1,1,0)} \to \H_{2,\mathbb{Z}/2\mathbb{Z},(1,1)}\] glues two copies of $C_1$ and one copy of $C_2$ to obtain the curve $C$ from (\ref{eqn:exaCactionglue}) together with the natural $G$-action on $C$.
\end{example} 
 
Now we want to  write down a similar map $\xi_{(\Gamma,G)}$ for a general admissible $G$-graph as in (\ref{eqn:Hstablegraph}). As in the example, the domain of $\xi_{(\Gamma,G)}$ will need one factor $\H_{g(v'),G_{v'},\xi_{v'}}$ for a representative $v'$ of each orbit of the vertices $V(\Gamma)$ under $G$.
 
Because of this, it will be convenient to introduce the quotient graph $\tilde \Gamma=\Gamma/G$. Its sets of vertices, half-edges and legs are the quotient sets of the corresponding sets for $\Gamma$ under the equivalence relation of being in the same $G$-orbit. Denote the quotient maps by 
\begin{equation}
 \pi_V : V(\Gamma) \to V(\tilde \Gamma), \pi_L : L(\Gamma) \to L(\tilde \Gamma),\pi_H : H(\Gamma) \to H(\tilde \Gamma).
\end{equation}
Since the group action respects the incidence of legs and half-edges with vertices and does not reverse the orientation of an edge, this is a well-defined stable graph. In fact, if $(\Gamma,G)$ is the admissible $G$-graph associated to $(C \to D, (p_{i,a})_{i,a}, (q_i)_i)$, then $\tilde \Gamma$ is canonically identified with the dual graph of $(D,(q_i)_i)$. 



Now choose a section of $\pi_V$, i.e.\  a set $V' \subset V(\Gamma)$ of representatives $v \in V(\Gamma)$ for each $[v] \in V(\tilde \Gamma)$. For each ${v'} \in V'$ choose representatives $L'_{v'}, H'_{v'}$ of the legs and half-edges in $\Gamma$ incident to ${v'}$ up to the action of $G_{v'}$. With these choices made, there are natural $G$-equivariant identifications
\begin{align*}
 \pi_V^{-1}([{v'}]) &\cong G/G_{v'} && ({v'} \in V'),\\
 \pi_L^{-1}([l]) &\cong G/G_l && (l \in L'_{v'}),\\
 \pi_H^{-1}([h]) &\cong G/G_h && (h \in H'_{v'}).
\end{align*}

For every ${v'}\in V'$ let
\begin{equation}
 \xi_{v'} = (h_l : l \in L'_{v'} \cup H'_{v'}) 
\end{equation}
be a monodromy datum. We will construct a natural morphism
\begin{equation}
 \xi_{(\Gamma,G)} : \H_{(\Gamma,G)} = \prod_{v' \in V'} \H_{g(v'),G_{v'},\xi_{v'}} \to \H_{g,G,\xi},
\end{equation}
whose image is $\H_{g,G,\xi}(\Gamma,G)$.

The idea is the following: for a curve $C$ with group action by $G$ and admissible $G$-graph $\Gamma$, look at the orbit of a vertex ${v'} \in V'$. This orbit corresponds to an orbit of the component $C_{v'}$ of $C$ associated to ${v'}$. Since $G$ acts on $C$ by isomorphisms, all components in the $G$-orbit of $C_{v'}$ must be isomorphic. On the other hand, the stabilizer $G_{v'}$ of ${v'}$ acts on $C_v$ with monodromy datum $\xi_{v'}$. The idea of the map $\xi_{(\Gamma,G)}$ is to take all these restricted actions $G_{v'} \curvearrowright C_v$ (together with the marked points from $\xi_{v'}$) and to reconstruct the curve $C$ from this.

The first step is to go from $C_{v'}$ to the orbit of components of $C$, which were all isomorphic to $C_{v'}$. So for each ${v'} \in V'$ and for an element $(G_{v'} \curvearrowright C_{v'}, (p^{v'}_{i,a})_{i,a}) \in \H_{g({v'}),G_{v'},\xi_{v'}}$ there is a (usually disconnected) curve
\[\widehat C_{v'} = \bigcup_{[g] \in G/G_{v'}} C_{v'}\]
together with a $G$-action on $\widehat C_{v'}$ such that the induced action on the connected components is the left-multiplication of $G$ on $G/G_{v'}$ and such that on the component $C_{v'} \subset \widehat C_{v'}$ corresponding to the trivial element $[e_G]$ the induced action of $G_{v'}$ is the given one. 

Note that for the action of $G$ on $\widehat C_{v'}$ there is a choice involved. This is the same kind of ambivalence that appears when defining the induced representation of $G_{v'}$ in $G$ by choosing a set of representatives of $G/G_{v'}$. In the end, the resulting map $\xi_{(\Gamma,G)}$ does not depend on this choice.

The map $\xi_{(\Gamma,G)}$ is defined by taking the disjoint union of all curves $\widehat C_{v'}$, ${v'} \in V'$, and glueing them together to a curve $C$ according to the graph $\Gamma$. Since the actions of $G$ on the curves $\widehat C_v$ are compatible with this gluing, the new curve $C$ inherits a $G$-action. Taking the quotient map $C \to D=C/G$ gives the desired admissible $G$-cover. This ends the construction of $\xi_{(\Gamma,G)}$.

From the construction above it follows that the composition of $\xi_{(\Gamma,G)}$ with the map $\phi$ forgetting the $G$-action factors through the gluing map $\xi_\Gamma$ associated to the stable graph $\Gamma$ as
\begin{equation}
  \begin{tikzcd}
   \H_{(\Gamma,G)} = \prod_{v' \in V'} \H_{g(v),G_v,\xi_v} \arrow[r,"\xi_{(\Gamma,G)}"]\arrow[d,"\phi_{(\Gamma,G)}"] & \H_{g,G,\xi}\arrow[d,"\phi"]\\
   \M_\Gamma = \prod_{v \in V(\Gamma)} \M_{g(v),n(v)} \arrow[r,"\xi_\Gamma", swap] & \M_{g,r}
  \end{tikzcd}
  \end{equation}
For the map $\phi_{(\Gamma,G)}$ note that the set $V(\Gamma)$ is a disjoint union of the $G$-orbits of the elements $v' \in V'$. Then $\phi_{(\Gamma,G)}$ is a product of maps $\phi$ followed by diagonals $\Delta : \M_{g(v'),n(v')} \to \prod_{v \in Gv'}  \M_{g(v'),n(v')}$. These diagonal maps arise since in the construction above, when defining $\widehat C_v$ we take the disjoint union of identical curves $C_v$, indexed by $G/G_v$, which is canonically bijective to the orbit of $v$.

Putting everything together, we see that the map $\phi_{(\Gamma,G)}$ factors as
\begin{equation} \label{eqn:phiGammaGfactor}
  \phi_{(\Gamma,G)} \colon 
  \underbrace{\prod_{v' \in V'} \H_{g(v),G_v,\xi_v}}_{=\H_{(\Gamma,G)}} \xrightarrow{\prod \phi} \prod_{v' \in V'} \M_{g(v'),n(v')} \xrightarrow{\prod \Delta} \underbrace{\prod_{v' \in V'} \prod_{v \in Gv'}  \M_{g(v'),n(v')}}_{=  \M_\Gamma}.
\end{equation}


Likewise we have a map 
\begin{equation}
    \delta_{(\Gamma,G)} \colon  \H_{(\Gamma,G)} \to  \prod_{v' \in V'} \M_{g'(v'),b(v')}:=\M_{\Gamma/G} ,
\end{equation}
given by taking the target morphism $\delta\colon \H_{g(v'),G_{v'},\xi_{v'}} \to \M_{g'(v'),b(v')}$ for each $v'\in V' \cong V(\Gamma/G)$.

\begin{definition} \label{Def:AutG}
 A \emph{morphism of admissible $G$-graphs} $f: (\Gamma',G) \to (\Gamma,G)$ is a usual morphism of stable graphs $f:\Gamma' \to \Gamma$ (see Definition \ref{def:morgraph}) which is equivariant with respect to the $G$-action and such that for each half-edge $l$ of $\Gamma$ mapping to a half-edge $l'$ of $\Gamma'$ we have $h_{l}=h'_{l'}$, where $h,h'$ are the monodromy data of $(\Gamma,G)$ and $(\Gamma',G)$, respectively.
 
 For an admissible $G$-graph $(\Gamma,G)$ the set of automorphisms of $(\Gamma,G)$ is denoted $\aut(\Gamma,G) \subset \aut(\Gamma)$.
\end{definition}

\begin{proposition}[see \cite{BertinRomagny}, Proposition 7.4.1.]
 The morphism $\xi_{(\Gamma,G)}$ has as image the closure $\H_{g,G,\xi}(\Gamma,G)$  of the set of elements in $\H_{g,G,\xi}$ with associated admissible $G$-graph $(\Gamma,G)$. It is generically of degree  $|\aut(\Gamma,G)|$ onto its image.
\end{proposition}
One quickly checks that the domain of $\xi_{(\Gamma,G)}$ (and thus $\H_{g,G,\xi}(\Gamma,G)$) is of the same dimension as the space $\M_{\Gamma/G}$ associated to the quotient graph $\Gamma/G$. In particular, the boundary divisors of $\H_{g,G,\xi}$ are indexed by the admissible $G$-graphs $(\Gamma,G)$ such that $\Gamma/G$ has exactly one edge. In other words, $\Gamma$ has exactly one $G$-orbit of edges.



\subsection{Abstract Characterization of Admissible \texorpdfstring{$G$}{G}-Graphs}
In Definition \ref{Def:Hstablegraph} we defined the admissible $G$-graph associated to a given pointed admissible $G$-cover $(C \to D, (p_{i,a})_{i,a}, (q_i)_i)$. However, for computations it will be necessary to give a purely combinatorial definition of an admissible $G$-graph. 

It turns out that not every graph $\Gamma$ with a group action by $G$ and monodromy data at the half-edges and legs can occur as the admissible $G$-graph of an admissible $G$-cover. This is related to the fact that not every tuple $(g,G,\xi)$ leads to a nonempty space $\H_{g,G,\xi}$. The question whether $\H_{g,G,\xi}$ is empty (and more generally the degree of the map $\delta : \H_{g,G,\xi} \to \M_{g',b}$) can be answered in terms of the group theory of $G$. 

We begin by defining a group action with stabilizer data on a stable graph.
\begin{definition}
 Let $G$ be a finite group and $\Gamma=(V,H,L,g,\iota,a,\zeta)$ an $n$-pointed stable graph. A \emph{$G$-action on $\Gamma$} is an action of $G$ on the sets $V,H,L$ such that $g: V \to \mathbb{Z}_{\geq 0}$ is $G$-invariant and such that $\iota,a,\zeta$ are $G$-equivariant. 
 
 A \emph{$G$-action with stabilizer datum} on $\Gamma$ is a $G$-action together with a collection $(h_l)_{l \in H \cup L}$ of elements of $G$ such that the stabilizer of each leg or half-edge $l$ of $\Gamma$ is cyclic with generator $h_l$ and such that for each $l \in H$ we have $h_{\iota(l)} = h_l^{-1}$. In other words, for $(l,l')$ forming an edge of $\Gamma$ we have $h_{l'}=h_l^{-1}$.
 
 A \emph{pre-admissible $G$-graph} $(\Gamma,G)$ is a stable graph $\Gamma$ together with a $G$-action and stabilizer datum such that the induced action of $\iota$ on the quotient $H/G$ is fixed point free, i.e.\  for an edge $(l,l')$ there is no element $t \in G$ with $tl=l'$.
\end{definition}
One verifies easily that the admissible $G$-graph associated to a given admissible $G$-cover $(C \to D, (p_{i,a})_{i,a}, (q_i)_i)$ gives a pre-admissible $G$-graph in the sense above. However, not every pre-admissible $G$-graph will correspond to some (nonempty) stratum of a suitable space of pointed admissible $G$-covers.

\begin{definition}
 Let $(\Gamma,G)$ be a pre-admissible $G$-graph. With respect to a choice $l_1, \ldots, l_b \in L$ of representatives of the $G$-action on $L$, the \emph{monodromy datum of $(\Gamma,G)$} is the collection $(h_{l_1}, \ldots, h_{l_b}) \in G^b$.
 
 For each vertex $v$ of $\Gamma$ with stabilizer $G_v$ under $G$ let  $L_v', H_v'$ be representatives of the legs and half-edges of $\Gamma$ incident to $v$ up to the action of $G_v$. Then the tuple $(h_{l})_{l \in L_v' \cup H_v'}$ is called the \emph{local monodromy datum at $v$}.
\end{definition}
\begin{definition} \label{Def:abstHurwitzstgraph}
 An \emph{admissible $G$-graph} is a pre-admissible $G$-graph $(\Gamma,G)$ such that for each vertex $v$ of $\Gamma$, the space $\H_{g(v), G_v, (h_{l})_{l \in L_v' \cup H_v'}}$ is nonempty (for some choice of representatives $L_v',H_v'$ as above\footnote{The nonemptyness of the space of pointed admissible $G$-covers is independent of these choices.}).
\end{definition}
Note that the space $\H_{g(v), G_v, (h_{l})_{l \in L_v' \cup H_v'}}$ is nonempty if and only if its interior part $\Ho_{g(v), G_v, (h_{l})_{l \in L_v' \cup H_v'}}$ is nonempty. Indeed, the balancing condition of admissible covers was chosen in such a way that a nodal admissible cover can always be smoothed in a family (see \cite{Abramovich2001}).

Using the definition of the $G$-equivariant gluing maps $\xi_{(\Gamma,G)}$ above, the following is immediate.
\begin{proposition}
 A pre-admissible $G$-graph $(\Gamma,G)$ with monodromy datum $\xi$ corresponds to a nonempty stratum of the space $\H_{g,G,\xi}$ if and only if $(\Gamma,G)$ is an admissible $G$-graph.
\end{proposition}
\begin{proof}
 If $(\Gamma,G)$ is admissible, the domain of the equivariant gluing map $\xi_{(\Gamma,G)}$ is nonempty and thus a general element of its image will be in the stratum $\Ho_{g,G,\xi}(\Gamma,G)$. On the other hand, if the  stratum corresponding to a pre-admissible $G$ graph $(\Gamma,G)$ is nonempty, then suitable components of the normalization of a general element $C \to D$ in this stratum (as described above) will give elements in the spaces $\Ho_{g(v), G_v, (h_{l})_{l \in L_v' \cup H_v'}}$, which are thus nonempty, hence $(\Gamma,G)$ is admissible.
\end{proof}
Using this result and Definition \ref{Def:abstHurwitzstgraph} we can reduce the classification of the strata in our space of pointed admissible $G$-covers to the question for which $(g,G,\xi)$ the space $\H_{g,G,\xi}$ is nonempty. 

A first obvious criterion is that the Riemann-Hurwitz formula needs to be satisfied, in the sense that for $\xi=(h_1, \ldots, h_b)$ the number
\[g'=\frac{1}{2 |G|} \left( 2g-2 - |G| \sum_{i=1}^b (1-\frac{1}{\operatorname{ord}_G(h_i)}) \right)+1  \]
is a nonnegative integer.

If this is satisfied, it makes sense to ask what the degree of the map $\delta : \H_{g,G,\xi} \to \M_{g',b}$ is. In other words, given a smooth curve $D$ of genus $g'$ with markings $q_1, \ldots, q_b$, how many pointed admissible $G$-covers $(C \to D, (p_{i,a})_{i,a}, (q_i)_i)$  lie above it (counted in a stacky sense, i.e.\  dividing by automorphisms)? This can be answered in terms of classical Hurwitz theory (see e.g. \cite{fultonhurwitz} or \cite{cavalieribook} for a modern treatment). The following discussion is related to results in \cite[Chapter 2.3]{BertinRomagny} concerning the number of connected components of the spaces $\H_{g,G,\xi}$.

Let the fundamental group $\fundgp$ of $D$ punctured at $q_1, \ldots, q_b$ -  with respect to some base point $y \in D^*=D \setminus \{q_1, \ldots, q_b\}$ - be given by
\begin{equation} \label{eqn:fundgens}
  \fundgp = \langle A_1, \ldots, A_{g'}, B_1, \ldots, B_{g'}, \gamma_1, \ldots, \gamma_b \rangle / \langle  \left( \prod_{j=1}^{g'} [A_j, B_j] \right) \gamma_1 \cdots \gamma_b \rangle.
\end{equation}
Here the $A_j, B_j$ are the usual generators of the fundamental group of $D$ and the $\gamma_k$ are loops around the punctures at $q_k$.

Given a $G$-cover $\varphi: C \to D$ whose branch points are contained in $\{q_1, \ldots, q_b\}$, it is \'etale over $D^*$.  For a loop $\gamma$ in $D^*$ based at $y$, we get a permutation of $\varphi^{-1}(y)$ by lifting $\gamma$ under $\varphi$ and sending the start point of the lifted path to the end point. Choose a base point $x \in \varphi^{-1}(y)$, then lifting the path $\gamma$ starting at $x$ there exists a unique $g_\gamma \in G$ such that the lifted path ends at $g_\gamma x$. 
The map 
\begin{equation} \label{eqn:psimonodrommorph}
  \psi : \fundgp \to G, \gamma \mapsto g_\gamma 
\end{equation}
is a group homomorphism\footnote{If composition in $\fundgp$ is defined as usual by concatenation of paths, we actually have to replace $\fundgp$ with the opposite group $\fundgp^{\mathrm{op}}$. Since these groups are isomorphic by $g \mapsto g^{-1}$ this does not change any of the following conclusions, so to keep notation simple we just write $\fundgp$.}. 

If we had chosen a different base point $x'=\eta x \in \varphi^{-1}(y)$ for $\eta \in G$, we would have obtained $g_\gamma' = \eta g_\gamma \eta^{-1}$. Thus $\psi$ is only unique up to conjugation in $G$.

From the presentation in (\ref{eqn:fundgens}) it is clear that $\psi$ is determined by the images $a_j, b_j, \sigma_k \in G$ of the generators $A_j, B_j, \gamma_k$ of $\fundgp$. Conversely, for any choice of such $a_j, b_j, \sigma_k \in G$ satisfying $(\prod_j [a_j, b_j] ) \sigma_1 \cdots \sigma_b = e_G$ we obtain a group homomorphism $\psi$. Thus, at least in principle, we can understand the set of group homomorphisms $\operatorname{Hom}(\fundgp, G)$ very concretely as a subset of the finite set $G^{2 g'+b}$. 

Due to the way we constructed $\psi$ from a $G$-cover $C \to D$, it is contained in a special subset $\operatorname{Hom}_\xi(\fundgp,G)$ of the set of all homomorphisms $\operatorname{Hom}(\fundgp,G)$.
\begin{definition} \label{Def:HomxipiG}
 We define $\operatorname{Hom}_\xi(\fundgp,G)$ to be the set of group homomorphisms $\psi: \fundgp \to G$ satisfying
 \begin{enumerate}
    \item[(1)] The homomorphism $\psi$ is surjective. 
    \item[(2)] The elements $\sigma_k=\psi(\gamma_k)$ describing the monodromy around the points $q_k \in D$ are contained in the conjugacy class $[h_k]$ of the elements $h_k$ from  the monodromy datum $\xi$.
\end{enumerate}
Since these properties are invariant under conjugating $\psi$ with an element of $G$, denote by $\operatorname{Hom}_\xi(\fundgp,G)/G$ the quotient of $\operatorname{Hom}_\xi(\fundgp,G)$ under this conjugation action.
\end{definition}
We claim that the morphism $\psi$ from (\ref{eqn:psimonodrommorph}) is contained in $\operatorname{Hom}_\xi(\fundgp,G)$. Property (1) follows from the fact that the cover $C$ is connected. Indeed, this is the case iff we can connect our base point $x \in C$ to any other point $tx$ in the same fibre by some path $\widehat \gamma$, for $t \in G$. In this case we can choose $\widehat \gamma$ such that it lies completely over $D^*$. But then it is the unique lift of the path $\gamma=\pi \circ \widehat \gamma \in \fundgp$ and so $\psi(\gamma)=g_{\gamma}=t$, hence every element $t\in G$ lies in the image of $\psi$. 
Finally property (2) comes from the definition of the monodromy datum of a cover.

On the other hand, given $\psi \in \operatorname{Hom}_\xi(\fundgp,G)$ we can reverse the construction (by Riemann's existence theorem) and obtain a $G$-cover $\varphi: C \to D$ with ramification behaviour described by $\xi$. Moreover, up to isomorphism, the cover $\varphi$ does not change if we conjugate $\psi$ by an element of $G$.

Note however that for the markings $(p_{i,a})_{i,a}$ in $C$ we can have some additional choices: in the preimage of the marking $q_i \in D$ there is at least one marking $p_i$ with mondoromy $h_i$ (i.e.\  with stabilizer generated by $h_i$ such that $h_i$ acts by multiplication with $\exp(2 \pi i/\operatorname{ord}_G(h_i))$ on $T_{p_i} C$). For a different preimage $\eta p_i$ of $q_i$, the corresponding monodromy is given by $\eta h_i \eta^{-1}$. Thus the set of all points in the fibre with monodromy $h_i$ is given by the quotient of the centralizer $C_G(h_i) = \{\eta \in G: \eta h_i \eta^{-1} = h_i\}$ by the subgroup generated by $h_i$.

A priori it would seem that the choice of markings is parametrized by the product $\prod_{j=i}^b C_G(h_i)/\langle h_i \rangle$. However, we also need to consider the automorphisms of the $G$-cover $C \to D$. First, any isomorphism $\rho: C \to C$ covering the identity of $D=C/G$ is of the form $p \mapsto \eta p$ for some $\eta \in G$, since $C \to D$ is a $G$-Galois cover. This isomorphism is $G$-equivariant iff $\rho(b p) = b \rho(p)$ for all $b \in G$, i.e.\  $\eta$ commutes with all elements $b \in G$. Conversely, for such a $\eta \in C(G)$ the map $\rho(p)=\eta p$ is a $G$-equivariant isomorphism of $C$ over the identity of $D$. 
Thus the set of choices of markings is actually parametrized by $(\prod_{i=1}^b C_G(h_i)/\langle h_i \rangle) / C(G)$.




We have now assembled all relevant ingredients to write down the degree of $\delta$.


\begin{theorem}\label{thm:condnonempty}
 Let $g \geq 0$, $G$ a finite group and $\xi=(h_1, \ldots, h_b)$ a monodromy datum. Then the space $\H_{g,G,\xi}$ is empty unless 
 \[g'=\frac{1}{2 |G|} \left( 2g-2 - |G| \sum_{i=1}^b (1-\frac{1}{\operatorname{ord}_G(h_i)}) \right)+1  \]
is a nonnegative integer. If this number is a nonnegative integer, the degree of the map $\delta : \H_{g,G,\xi} \to \M_{g',b}$ is given by
\begin{equation} \label{eqn:deltadeg}
 \operatorname{deg}(\delta) = \left| \operatorname{Hom}_\xi(\fundgp,G)/G \right| \cdot \left(\prod_{i=1}^b \frac{|C_G(h_i)|}{\operatorname{ord}_G(h_i)} \right)/ \left| C(G) \right|,
\end{equation}
for $\operatorname{Hom}_\xi(\fundgp,G)/G$ as in Definition \ref{Def:HomxipiG}.

In particular, the space $\H_{g,G,\xi}$ is empty iff $\operatorname{Hom}_\xi(\fundgp,G) = \emptyset$.
\end{theorem}
\begin{proof}
 In the section above we have shown that the preimage of a general $\mathbb{C}$-point $(C,q_1, \ldots, q_n) \in \M_{g',b}$ is isomorphic to the quotient of
 \[\left(\operatorname{Hom}_\xi(\fundgp,G)/G \right)\times  \prod_{i=1}^b C_G(h_i)/\langle h_i \rangle\]
 by the action of $C(G)$. This gives the multiplicity from above.
\end{proof}
For general finite groups $G$ and monodromy data $\xi$ it can be difficult to compute this degree. However, for a cyclic group $G=\mathbb{Z}/m \mathbb{Z}$ it is possible to write it very explicitly.

\begin{proposition}\label{prop:condcyclic}
 Let $m \geq 1$ 
 and let $G=\mathbb{Z}/m \mathbb{Z}$. Let $g \geq 0$, $\xi=(h_1, \ldots, h_b)$ a monodromy datum such that 
 \[g'=\frac{1}{2 |G|} \left( 2g-2 - |G| \sum_{i=1}^b (1-\frac{1}{\operatorname{ord}_G(h_i)}) \right)+1  \]
is a nonnegative integer. Let $m_0=\operatorname{gcd}(m,h_1, \ldots, h_b)$ and let $p_1, \ldots, p_s$ be the distinct prime factors of $m_0$.

Then the space $\H_{g,G,\xi}$ is empty unless $h_1 + \ldots + h_b=0 \in \mathbb{Z}/m \mathbb{Z}$. If this sum vanishes, the degree of $\delta$ satisfies
\begin{equation} \label{eqn:deltadegcyclic}
 \operatorname{deg}(\delta) = m^{2g'-1} \prod_{j =1}^s\left(1-\frac{1}{(p_j)^{2g'}} \right)\cdot \prod_{i=1}^b \operatorname{gcd}(m,h_i),
\end{equation}
where $\operatorname{gcd}(m,0)=m$. In particular, assuming $h_1 + \ldots + h_b=0 \in \mathbb{Z}/m\mathbb{Z}$, the space $\H_{g,G,\xi}$ is always nonempty for $g' \geq 1$ and is nonempty for $g'=0$ iff $m_0=1$.
\end{proposition}
\begin{proof}
First we compute $\operatorname{Hom}_\xi(\fundgp,G)$. As we saw, a group homomorphism $\psi: \fundgp \to G$ is uniquely determined by the images $a_j, b_j, \sigma_k$ of the generators $A_j, B_j, \gamma_k$ of $\fundgp$. The second condition of Definition \ref{Def:HomxipiG} forces $\sigma_k$ to lie in the conjugacy class $[h_k]$. Since $G$ is abelian we have $\sigma_k = h_k$ for all $k$. On the other hand, the homomorphism $\psi$ is surjective onto $G$ iff the elements $a_j, b_j, \sigma_k$ generate $G$. In other words if and only if
\begin{equation} \label{eqn:gcdcondition} 1=\operatorname{gcd}(m,a_1, \ldots, a_{g'}, b_1, \ldots b_{g'}, \sigma_1, \ldots, \sigma_b) = \operatorname{gcd}(a_1, \ldots, a_{g'}, b_1, \ldots b_{g'}, m_0). \end{equation}
Each $a_j,b_j$ can be chosen from $0, 1, \ldots, m-1$, so there are $m^{2g'}$ choices in total. We need to substract those where all the numbers $a_j,b_j$ and $m_0$ share a common factor. For each number $q$ dividing $m$ there are $m/q$ numbers among $0, \ldots, m-1$ divisible by $q$, namely $0,q, \ldots, q (m/q  -1)$. Thus there are $(m/q)^{2g'}$ choices of $a_j, b_j$ with common factor $q$. 

We see that from the set of all $m^{2g'}$ choices of $a_j, b_j$ we must substract  
\begin{equation} \label{eqn:exclusionabset} \bigcup_{j=1}^s \left\{(a_1, \ldots, b_{g'}) : p_j | \operatorname{gcd}(a_1, \ldots, b_{g'},m_0) \right\},\end{equation}
which is exactly the set of choices not satisfying (\ref{eqn:gcdcondition}). Applying the inclusion-exclusion principle for computing the cardinality of (\ref{eqn:exclusionabset}) we obtain
\begin{align*}
|\operatorname{Hom}_\xi(\fundgp,G)| &= m^{2g'}-\sum_{\emptyset \neq S \subset \{1, \ldots, s\}} (-1)^{|S|+1} \left( \frac{m}{\prod_{j \in S} p_j} \right)^{2g'}\\
&= m^{2g'} \sum_{S \subset \{1, \ldots, s\}} (-1)^{|S|}  \frac{1}{\prod_{j \in S} (p_j)^{2g'}} \\
&= m^{2g'} \sum_{S \subset \{1, \ldots, s\}} \prod_{j \in S} \left(- \frac{1}{( p_j)^{2g'}} \right)\\
&=m^{2g'} \prod_{j =1}^s\left(1-\frac{1}{(p_j)^{2g'}} \right)
\end{align*}

To compute the remaining factors from (\ref{eqn:deltadeg}) coming from the choice of markings and from automorphisms, we note that $C_G(h_i)=C(G)=G$ since $G$ is abelian. Note also that $|G|/\operatorname{ord}_G(h_i) = \operatorname{gcd}(m,h_i)$. This explains the remaining terms in (\ref{eqn:deltadegcyclic}) and finishes the proof.
\end{proof}

  \section{Pullbacks by Boundary Morphisms}\label{sec:theint}

 Let $A$ be an $n$ pointed graph of genus $g$ and let $G$ be a finite group. In this section we will compute the intersection between $[\M_A]$ and $[\H_{g,G,\xi}]$ using the excess intersection formula (cf. \cite[Proposition 17.4.1]{fulton1984}). For this we need to identify the fiber product $\M_A \times_{\M_{g,r}}\H_{g,G,\xi}$ and compute the top Chern class of the resulting excess bundle. 
 
 For the intersection of $[\M_A]$ and $[\H_{g,G,\xi}]$ recall that $\H_{g,G,\xi}$ is stratified according to admissible $G$-graphs $(\Gamma,G)$. The stratum corresponding to $(\Gamma,G)$ is contained in the image of $\M_A$ iff $\Gamma$ has an $A$-structure. This leads to the following definition.
 
\subsection{The Fibre Product} \label{sect:fibreprod}

\begin{definition} \label{Def:genericequivAstructure}
 Let $(\Gamma,G)$ be an admissible $G$-graph and $A$ a stable graph. An \emph{$A$-structure $f$ on $(\Gamma,G)$} is an $A$-structure $f=(\alpha,\beta,\gamma)$ on $\Gamma$. We say that the $A$-structure on $(\Gamma,G)$ is \emph{generic} if 
 \[G \im \beta = H_\Gamma.\]
 An \emph{isomorphism of $A$-structures $(\Gamma',G,f')\rightarrow (\Gamma,G,f)$} is an isomorphism of admissible $G$-graphs $(\Gamma',G)\rightarrow (\Gamma,G)$ such that the induced diagram of stable graphs 
      \begin{center}
   \begin{tikzcd}
    & A\\
    \Gamma'\arrow[r] \arrow[ru,"f'"] & \Gamma \arrow[u,"f",swap]
   \end{tikzcd}
  \end{center}
  commutes. 
\end{definition}

\begin{notation}
  For a space $\H_{g,G,\xi}$ such that the number of markings $p_{i,a}$ is $n$ and a fixed stable graph $A$ of genus $g$ with $n$ legs we will denote  by $\mathfrak{H}_{A;G,\xi}$ a set of representatives of isomorphism classes of generic $A$-structures $(\Gamma,G,f)$ on admissible $G$-graphs $(\Gamma,G)$ for $\H_{g,G,\xi}$.
  
  For a generic $A$-structure $(\Gamma,G,f)$ there exists a natural map 
  \[\phi_{(\Gamma,G,f)}: \H_{(\Gamma,G)} \xrightarrow{\phi_{(\Gamma,G)}} \M_{\Gamma} \xrightarrow{\xi_{f: \Gamma \to A}} \M_{A},\]
  where the first map was defined above and the second map is the partial gluing map induced by the $A$-structure on $\Gamma$.
 \end{notation}

\begin{proposition}\label{prop:fibreprod}
 There is a natural Cartesian diagram
  \begin{equation*}
  \begin{tikzcd}
   \coprod_{(\Gamma,G,f) \in \mathfrak{H}_{A;G,\xi}} \H_{(\Gamma,G)} \arrow[r,"\coprod \xi_{(\Gamma,G)}"]\arrow[d,"\coprod \phi_{(\Gamma,G,f)}",swap] & \H_{g,G,\xi} \arrow[d,"\phi"]\\
   \M_A\arrow[r,"\xi_A"] & \M_{g,r}
   \end{tikzcd}.
  \end{equation*}
\end{proposition}

 \begin{example}
  Let $G=\Z/2\Z$ and let $\xi= (1,...,1)=(1^8)$ be the monodromy datum of length $8$ prescribing $8$ fixed points under $G$. For the stable graph
  \[
A = \begin{tikzpicture}[->,>=bad to,baseline=-3pt,node distance=1.3cm,thick,main node/.style={circle,draw,font=\Large,scale=0.5}]
\node[main node] (A) {2};
\node at (0,-.7) (B) {...};
\node at (-.6,-.7) (1) {1};
\node at (.6,-.7) (2) {8};
\node at (-.5,.3) (h1) {\footnotesize $h_1$};
\node at (-.6,-.2) (h2) {\footnotesize $h_2$};
\draw [-] (A) to [out=200, in=160,looseness=10] (A);
\draw [-] (A) to (0,-.5);
\draw [-] (A) to (.5,-.5);
\draw [-] (A) to (-.5,-.5);
\draw [-] (A) to (.2,-.5);
\draw [-] (A) to (-.2,-.5);
\end{tikzpicture}
  \] 
 we will compute the set $\mathfrak{H}_{A;G,\xi}$. In other words we want to know which admissible $G$-graphs $(\Gamma,G)$ admit a generic $A$ structure.  We claim that all such admissible $G$-graphs have the following form: 
   \begin{center}
\begin{tabular}{c@{\hspace{2cm}}c}
$    \begin{tikzpicture}[->,>=bad to,baseline=-3pt,node distance=1.3cm,thick,main node/.style={circle,draw,font=\Large,scale=0.5}]
    \begin{scope}
\node [main node] (A) {2};
\node at (-1.5,0) (t1) {$(\Gamma_{1},G)=$};
\node at (.4,.4) (h1) {\footnotesize $\tilde{h}_1$};
\node at (.9,.4) (h2) {\footnotesize $\tilde{h}_2$};
\node at (1,0) [scale=.3,draw,circle,fill=black] (B)  {};
\draw [-] (A) to (.1,-.5);
\draw [-] (A) to (-.1,-.5);
\draw [-] (A) to (.5,-.5);
\draw [-] (A) to (-.5,-.5);
\draw [-] (A) to (.3,-.5);
\draw [-] (A) to (-.3,-.5);
\draw [-] (B) to (1.2,.2);
\draw [-] (B) to (1.2,-.2);
\draw [-] (A) to [out=30, in=150] (B);
\draw [-] (A) to [out=-30, in=-150] (B);
\end{scope}
\begin{scope}[shift ={(.65,-1.2)}]
\draw [->,line width=.7mm] (0,.5) to (0,-.5);
\end{scope}
  \begin{scope}[shift ={(0,-2)}]
\node at (0,0) [scale=.3,draw,circle,fill=black] (A)  {};
\node at (1,0) [scale=.3,draw,circle,fill=black] (B)  {};
\node at (-1.5,0) (t2) {$\Gamma_{1}/G=$};
\draw [-] (A) to (.1,-.5);
\draw [-] (A) to (-.1,-.5);
\draw [-] (A) to (.5,-.5);
\draw [-] (A) to (-.5,-.5);
\draw [-] (A) to (.3,-.5);
\draw [-] (A) to (-.3,-.5);
\draw [-] (B) to (1.2,.2);
\draw [-] (B) to (1.2,-.2);
\draw [-] (B) to (A);
\end{scope}
\end{tikzpicture}
$
&
$
\begin{tikzpicture}[->,>=bad to,baseline=-3pt,node distance=1.3cm,thick,main node/.style={circle,draw,font=\Large,scale=0.5}]
\begin{scope}
\node[main node] (A) {1};
\node at (-1.5,0) (t1) {$(\Gamma_{2},G)=$};
\node at (.4,.4) (h1) {\footnotesize $\tilde{h}_1$};
\node at (.9,.4) (h2) {\footnotesize $\tilde{h}_2$};
\node at (1.3,0) [main node] (B) {1};
\draw [-] (A) to [out=30, in=150] (B);
\draw [-] (A) to [out=-30, in=-150] (B);
\draw [-] (A) to (.5,-.5);
\draw [-] (A) to (-.5,-.5);
\draw [-] (A) to (.2,-.5);
\draw [-] (A) to (-.2,-.5);
\draw [-] (B) to (1.8,-.5);
\draw [-] (B) to (.8,-.5);
\draw [-] (B) to (1.5,-.5);
\draw [-] (B) to (1.1,-.5);
\end{scope}
\begin{scope}[shift ={(.65,-1.2)}]
\draw [->,line width=.7mm] (0,.5) to (0,-.5);
\end{scope}
  \begin{scope}[shift ={(0,-2)}]
\node at (0,0) [scale=.3,draw,circle,fill=black] (A)  {};
\node at (1,0) [scale=.3,draw,circle,fill=black] (B)  {};
\node at (-1.5,0) (t2) {$\Gamma_{2}/G=$};
\draw [-] (A) to (.5,-.5);
\draw [-] (A) to (-.5,-.5);
\draw [-] (A) to (.2,-.5);
\draw [-] (A) to (-.2,-.5);
\draw [-] (B) to (1.8,-.5);
\draw [-] (B) to (.8,-.5);
\draw [-] (B) to (1.5,-.5);
\draw [-] (B) to (1.1,-.5);
\draw [-] (B) to (A);
\end{scope}
\end{tikzpicture}
$
\end{tabular}.
\end{center}
where the generator $\tau=1$ of $G$ acts as the involution fixing all the legs and all the vertices and switching the edges. 

To see this first note that by definition  for any admissible $G$-graph $(\Gamma,G)$ with monodromy datum $\xi$, the genus of $\Gamma/G$ must equal $0$. It immediatly follows that $\Gamma$ cannot have any loops. Since $A$ has only one edge, the cardinality of $G$ is 2 and the $A$ structure is generic $\Gamma$ has at most 2 edges. The possible admissible $G$-graphs are now completely determined by the condition that  for each vertex $v$ the involution $\tau$ must fix exactly $2g(v)+2$ legs or half-edges incident to $v$. 

There are in total $\binom{8}{6}$ admissible $G$-graphs of the form $(\Gamma_1,G)$ given by the different distributions of the legs. Likewise there are $\frac{1}{2}\binom{8}{4}$ admissible $G$-graphs of the form $(\Gamma_2,G)$.

For each $A$-structure $f=(\alpha,\beta,\gamma)$ on $(\Gamma_1,G)$ we can compose $f$ with an isomorphism of $(\Gamma_1,G)$ such that the edge $(h_1,h_2)$ is mapped to $(\tilde{h}_1,\tilde{h}_2)$. There are two possible choices, either $\beta(h_1)=\tilde{h}_1$ or $\beta(h_1)=\tilde{h}_2$. After these choices the $A$-structure $f$ is completely determined so there are 2 isomorphism classes of generic $A$-structure on $(\Gamma_1,G)$. The same argument holds for $(\Gamma_2,\tau_2)$ and there are 2 isomorphism classes of $A$-structures on $(\Gamma_2,G)$.

Granted Proposition \ref{prop:fibreprod} it now follows that we have a fibered diagram.
\begin{equation*}
  \begin{tikzcd}
    (\H_{(\Gamma_1,G)})^{\amalg 2 \binom{8}{6}}\coprod (\H_{(\Gamma_2,G)})^{\amalg \binom{8}{4}} \arrow[r]\arrow[d] & \H_{3,G,(1^8)} \arrow[d,"\phi"]\\
   \M_A\arrow[r,"\xi_A"] & \M_{3,8}
   \end{tikzcd}.
  \end{equation*}
An easy dimension count shows that there is no excess of intersection. We have therefore computed the intersection between $[\phi(\H_{3,G,(1^8)})]$ and $[\xi_A(\M_A)]$ in $A^\bullet(\M_{3,8})$.
\end{example}

  For the proof of Proposition \ref{prop:fibreprod}, we will proceed in the same way as in the proof of Proposition \ref{intprop}.

 \begin{definition}\label{def:admmark}
  Let $C\rightarrow D$ be an admissible cover of $\H_{g,G,\xi}$  over a connected base scheme~$\Base$. A $(\Gamma,G)$-marking on $C\rightarrow D$ is a $\Gamma$-marking on $C$ (see Definition \ref{def:curstruc}) such that the action of $G$ on $\Gamma$ is induced by the action of $G$ on $C$. We will denote an admissible cover with $(\Gamma,G)$ marking by $(C\rightarrow D)_{(\Gamma,G)}$.
 \end{definition}

  A $(\Gamma,G)$-marking on $C\rightarrow D$ defines a $\Gamma$-marking on $C$ and a $\Gamma/G$-marking on $C/G$. The data of a $(\Gamma,G)$-marking can be pulled back along morphisms of base schemes $\Base'\rightarrow \Base$ (by definition of the pullback of an admissible cover along such a morphism and the pullback of the $\Gamma$-structure on $C$).   We define a stack $\H_{(\Gamma,G)}'$ of admissible covers with a $(\Gamma,G)$-marking.

 \begin{proposition}\label{prop:admmark}
  There is a natural isomorphism $\H_{(\Gamma,G)}'\simeq \H_{(\Gamma,G)}$.
 \end{proposition}

 \begin{proof}
The proof is completely analogous to that of Proposition \ref{prop:Gcurve}.
 \end{proof}

 \begin{proof}[Proof of Proposition \ref{prop:fibreprod}]
   We use the modular interpretation of  $\H_{(\Gamma,G)}$ given by admissible covers with a $(\Gamma,G)$-marking from Proposition~\ref{prop:admmark} and the modular interpretation of $\M_A$ in terms of curves with an $A$-marking  (Definition~\ref{def:curstruc}).  
  
  An object $((C\rightarrow D)_{(\Gamma,G)}, f)$ of $\coprod_{(\Gamma,G,f) \in \mathfrak{H}_{A;G,\xi}} \H_{(\Gamma,G)} $ over a (connected) scheme $\Base$ consists of an admissible $G$-graph $(\Gamma,G)$ with a generic $A$-structure $f$ and a Hurwitz cover  $C\rightarrow D$ with a $(\Gamma,G)$-marking. By definition we have 
  \[
  \xi_A\circ \phi_{(\Gamma,G,f)} ((C\rightarrow D)_{(\Gamma,G)}, f) = C = \phi \circ \xi_{(\Gamma,G)} ((C\rightarrow D)_{(\Gamma,G)}, f).
  \]
  We have a natural isomorphism $\xi_A\circ \phi_{(\Gamma,G,f)} \Rightarrow \phi \circ \xi_{(\Gamma,G)}$  given by the identity, so the strict universal property of fibre products (of stacks) gives us a morphism 
  \[q\colon  \coprod_{(\Gamma,G,f) \in \mathfrak{H}_{A;G,\xi}} \H_{(\Gamma,G)} \rightarrow \M_A\times_{\M_{g,r}} \H_{g,G,\xi}\]. 
  In other words we have a commutative diagram
    \begin{center}
   \begin{tikzcd}
     \coprod_{(\Gamma,G,f) \in \mathfrak{H}_{A;G,\xi}} \H_{(\Gamma,G)}  \arrow[ddr,"\coprod \phi_{(\Gamma,G,f)}"{name=U},swap,bend right] \arrow[drr,"\coprod \xi_{(\Gamma,G)}",bend left] \arrow[dr,"q" description] & &\\
   &      \M_A\times_{\M_{g,r}} \H_{g,G,\xi}  \arrow[dr,phantom,"\Rightarrow"] \arrow[r,"p_2"]\arrow[d,"p_1",swap]  &  \H_{g,G,\xi} \arrow[d,"\phi"]\\
   &  \M_A\arrow[r,"\xi_A",swap]&\M_{g,r}.
   \end{tikzcd}
  \end{center}
  We will construct an  inverse functor $r$ to $q$.

 Let $(C'_A,C\rightarrow D,\alpha\colon C'\xrightarrow{\sim} C)$ be an object of $\M_A\times_{\M_{g,r}} \H_{g,G,\xi}$ over a base scheme $\Base$ and let $G$ be the group associated to the admissible cover $C\rightarrow D$. As explained in the proof of Proposition \ref{intprop} the $A$-structure on $C'$ again passes through $\alpha$ to define an $A$-structure on the curve $C$. We now have the following data on $C\rightarrow D$
 \begin{enumerate}[i]
  \item a set of sections 
  \[
  E:=\bigcup_{a\in G} \{a(\sigma_i)\}_i
  \]
  in the singular locus of $C$ where $\sigma_i$ are the sections defined by the $A$-structure on $C'$,
  
  \item a set of sections 
  \[H:=\bigcup_{a\in G} \{a(\tilde{\sigma}_{i,j})\}_{i,j}\]
  in the normalization $\tilde{C}$ of $C$,
  
  \item a set $V$ of $\pi$-relative components of $C\backslash E$ 
 \end{enumerate}
 This data defines a stable graph 
 \[
  \Gamma := (V,H,L,g,\iota,a,\zeta)
 \]
 where $g$ is the genus function that assigns the geometric genus to each element of $V$, $\iota$ is the involution defined by $\iota(a(\tilde{\sigma}_{i,1}))=a(\tilde{\sigma}_{i,2})$ for all $a\in G$, and $a\colon H\rightarrow V$ is the function that sends a section to the $\pi$-relative component it lies on. 
 The action of $G$ on $C$ (and on the normalization $\tilde{C}$) induces a $G$-action (with stabilizer datum) on $\Gamma$. Since $C \to D$ is an object of $\H_{g,G,\xi}$, it follows that $(\Gamma,G)$ is an admissible $G$-graph.
  
  Let $\beta\colon H_A\hookrightarrow H$ be the obvious inclusion. Let $\alpha\colon V\rightarrow V_A$ be the map sending a $\pi$ relative component of $D_{\Gamma}$ to the $\pi$ relative component of $D_A$ it is mapped to under the partial normalization defined by the sections $E\backslash \im \beta$. Let $\gamma\colon H\backslash \im \beta \rightarrow V_A$ be the map which sends a section $a(\sigma_{i,j})$, where $a$ is not the trivial element, to the $\pi$-relative component of $C_A$ it is mapped to under the partial normalization map. This defines an $A$-structure $f$ on $(\Gamma,G)$. This $A$-structure is generic since $G\beta(H_A)=H$ by definition of $H$ (where $H_A$ is the set of half edges of $A$).
  
   \begin{figure}[H]
  \centering
    \includegraphics[width=.5\textwidth]{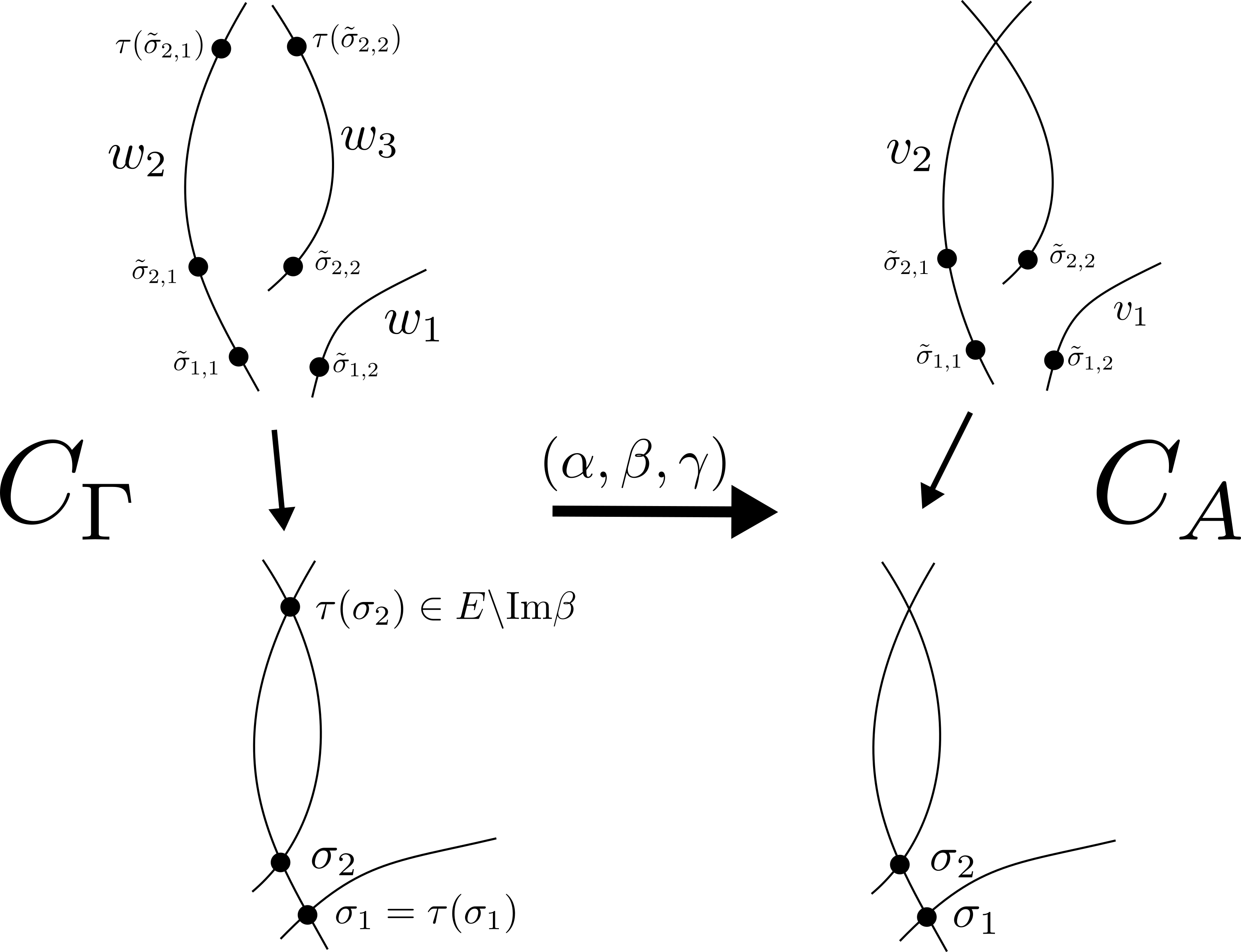}
    \label{fig:structures}
    \caption{The curve $C$ with $A$ structure and $\Gamma$ structure when $G\simeq \Z/2\Z$ with generator~$\tau$.}
\end{figure}
  
  Note that $f$ is unique up to isomorphisms of $A$-structures on $(\Gamma,G)$. We therefore have a well defined object of $ \coprod_{(\Gamma,G,f) \in \mathfrak{H}_{A;G,\xi}} \H_{(\Gamma,G)} $. This defines the functor $r$ on the objects of $\M_A\times_{\M_{g,r}} \H_{g,G,\xi} $.
  
  Let  $(\lambda \colon C'_1\rightarrow C_2', \eta \colon (C_1\rightarrow D_1)\rightarrow (C_2\rightarrow  D_2))$  be a morphism  in $\M_A\times_{\M_{g,r}} \H_{g,G,\xi}$ (note that $\xi_A(\lambda)=\phi (\eta)$). By passing through the above construction we see that we get an isomorphism of admissible $G$-graphs $(\Gamma_1,G)\rightarrow (\Gamma_2,G)$ which commutes with the $A$-structures $f_1$ and $f_2$ on $\Gamma_1, \Gamma_2$. We therefore obtain a morphism
  \[
   ((C_1\rightarrow D_1)_{(\Gamma_1,G)},f_1) \rightarrow ((C_2\rightarrow D_2)_{(\Gamma_2,G)},f_2).
  \]
  This defines a morphism in $ \coprod_{(\Gamma,G,f) \in \mathfrak{H}_{A;G,\xi}} \H_{(\Gamma,G)} $. This completes the definition of the functor \[r\colon \M_A\times_{\M_{g,r}} \H_{g,G,\xi} \rightarrow  \coprod_{(\Gamma,G,f) \in \mathfrak{H}_{A;G,\xi}} \H_{(\Gamma,G)}. \]
  
  It remains to check that $r\circ q$ and $q\circ r$ are naturally isomorphic to the respective identity maps. Let $((C\rightarrow D)_{(\Gamma,G)},f)$ be an object of $\coprod_{(\Gamma,G,f) \in \mathfrak{H}_{A;G,\xi}} \H_{(\Gamma,G)} $. Since $f$ is a generic $A$-structure on $(\Gamma,G)$ it is easy to verify that
  \begin{align*}
   r\circ q ((C\rightarrow D)_{(\Gamma,G)},f)&= r (C_A, (C\rightarrow D), \text{id}_C)\\
   &=((C\rightarrow  D)_{(\Gamma,G)},f).
  \end{align*}

  If $(C'_A, (C\rightarrow D), \alpha)$ is an object of $\M_A\times_{\M_{g,r}} \H_{g,G,\xi}$, then
  \begin{align*}
   q\circ r (C'_A, (C \rightarrow D), \alpha) & = q ((C\rightarrow D)_{(\Gamma,G)},f)\\
   &=  (C_A,(C\rightarrow D), \text{id}_C). 
  \end{align*}
 The isomorphism $\alpha$ induces an isomorphism of curves with $A$ structure $\alpha^{-1}\colon C'_A \rightarrow C_A$ so $(C_A,  (C\rightarrow D),\alpha)$ and $(C_A,  (C\rightarrow D), \text{id}_S)$ are isomorphic by $(\alpha^{-1},\text{id}_{C\rightarrow D})$. We thus have a natural isomorphism of functors $q\circ r \xRightarrow{\sim} \text{id}$.
  \end{proof}

 \subsection{The Excess Bundle}
 
 Let $A$ be a stable $n$ pointed genus $g$ graph. We have shown in Proposition \ref{prop:fibreprod} that the diagram
 \begin{equation}
 \label{diag:excessfib}
  \begin{tikzcd}
   \coprod_{(\Gamma,G,f)\in \mathfrak{H}_{A,G,\xi}} \H_{(\Gamma,G)} \arrow[r,"\coprod \xi_{(\Gamma,G)}"]\arrow[d,"\coprod \phi_{(\Gamma,G,f)}",swap] &  \H_{g,G,\xi}\arrow[d,"\phi"] \\
   \M_A \arrow[r,"\xi_A"] &\M_{g,r}.
  \end{tikzcd}
  \end{equation}
is Cartesian. We can compute the excess bundle separately on each component $\H_{(\Gamma,G)}$  of the disjoint union above. Each component $\H_{(\Gamma,G)}$ comes with an $A$-structure $f=(\alpha,\beta,\gamma)$ on $\Gamma$ so we have to compute
  \[
   E_f = \phi_f^* N_{\xi_A} /N_{\xi_{(\Gamma,G)}},
  \]
where we abbreviate $\phi_{(\Gamma,G,f)}$ as $\phi_f$ from now on.

 
 
 Let us slightly abuse notation and also denote by $\beta$ the induced map of edges $E_A\rightarrow E_\Gamma$. Let $N$ be a set of representatives for the orbits of the action of $G$ on $E$ such that $N\subset \im \beta$ (this is possible since the $A$-structure $f$ is generic). Then:

 \begin{proposition}\label{prop:excess}
 With the above notation we have on $\H_{(\Gamma,G)}$ that
\begin{equation}\label{eq:topexcess}
 c_{\text{top}} E_f = \prod_{(h,h')\in \im \beta\setminus N}  -\psi_h - \psi_{h'}.
\end{equation}
 \end{proposition}
Some remarks concerning this formula: given a graph $(\Gamma,G)$ and a half-edge $h$, for any element $t \in G$ we have $\psi_{th}=\psi_{h}$, since the corresponding cotangent spaces are related by the action via $t$. This shows that the above formula is independent of the choice of representatives $N$ made before. 

In fact, we can interpret this formula as follows: for any orbit of an edge $e=(h,h')$ in $\Gamma$ let $k=|\im \beta \cap (G e)|$ be the number of edges in this orbit coming from the graph $A$. By the assumption that the $A$-structure on $(\Gamma,G)$ is generic, we have $k \geq 1$. Then the orbit of $e$ contributes a factor $(-\psi_h - \psi_{h'})^{k-1}$ to the Chern class $c_{\text{top}} E_f$ of the excess bundle. In particular, there is excess if and only if there is an orbit of edges containing at least two edges coming from $A$.

For the proof of Proposition \ref{prop:excess} we will use the following facts:

\begin{lemma}
  Recall that fibre products (of stacks) commute are compatible with composition. That is, if we have maps $f_1\colon X_1\rightarrow Z$, $f_2\colon X_2\rightarrow X_1$ and $g\colon Y\rightarrow Z$ then there is an isomorphism $X_2\times_{X_1} (X_1\times_Z Y)\simeq X_2\times_Z Y$.  Moreover if the morphisms $f_1$ and $f_2$ are lci, then in the resulting diagram
 \begin{center}
  \begin{tikzcd}
 X_2\times_{X_1} (X_1\times_Z Y)\simeq X_2\times_Z Y \arrow[r,"h_2"]\arrow[d,"g_2"] & X_1\times_Z Y \arrow[r,"h_1"]\arrow[d,"g_1"] & Y \arrow[d,"g"] \\
 X_2\arrow[r,"f_2"] & X_1 \arrow[r,"f_1"] & Z
  \end{tikzcd}
 \end{center}
 we have 
\[
 c_{\textup{top}} (g_2^*N_{f_1\circ f_2} / N_{h_1\circ h_2}) = c_{\textup{top}} (g_2^*N_{f_2} / N_{h_2}) \cdot h_2^*c_{\textup{top}} (g_1^*N_{f_1} / N_{h_1}).
\]
\end{lemma}

 \begin{proof}[Proof of Proposition \ref{prop:excess}]
 We will argue by induction on the edges of $A$. We will show that whenever multiple edges of $A$ land in the same $G$ orbit of edges of $\Gamma$ there is a contribution to the excess bundle and then identify this contribution for each such edge of $A$. 
 
 We can decompose the map $\xi_A$ into a sequence of gluing morphisms $\xi_i$ each gluing a single edge. In this way we obtain a sequence of fibered diagrams
 \begin{equation}\label{eq:decomposition}
  \begin{tikzcd}
  F \arrow[r]\arrow[d,"\coprod \phi_{f}"] & ... \arrow[r] & F_{i} \arrow[r,"\eta_{i}"] \arrow[d,"\phi_{i}"] & F_{i-1}\arrow[r] \arrow[d,"\phi_{i-1}"] & ... \arrow[r] & \H_{g,G,\xi}\arrow[d,"\phi"]\\
  \M_A\arrow[r] & ...\arrow[r] &  \M_{A_{i}}\arrow[r,"\xi_{i}"] & \M_{A_{i-1}} \arrow[r] &... \arrow[r] &\M_{g,r} 
    \end{tikzcd}.
 \end{equation}
 Abusing notation we then have
 \[
  c_{\textup{top}} (\phi_f^* N_{\xi_A} / N_{\xi_{(\Gamma,G)}}) = \prod_{i=1}^{\# E_A} c_{\textup{top}} (\phi^*_iN_{\xi_{i}}/N_{\eta_i}).
 \]
 It therefore suffices to determine $c_{\textup{top}} (\phi^*_iN_{\xi_{A_i}}/N_{\eta_i})$ for each $i$. 
 
 By Proposition \ref{prop:fibreprod} the fibre products $F_{i}$ of Diagram (\ref{eq:decomposition}) take the form
 \[
 F_i=\coprod_{(\Gamma_i,G,f_i)\in \mathfrak{H}_{A_i,G,\xi}} \H_{(\Gamma_i,G)} 
 \]
 Let $ \H_{(\Gamma',G)} \subset F_i$ be one of the components. The map 
 \[
  \phi_{(\Gamma_i,G,f_i)}\colon  \H_{(\Gamma_i,G)} \rightarrow \M_{A_i}
 \]
 factors as $\xi_{\Gamma_i\rightarrow A_i}\circ \phi_{\Gamma_i}$  where $\phi_{\Gamma_i}\colon \H_{\Gamma_i,G} \hookrightarrow \M_{\Gamma_i}$ is the source map and $\xi_{\Gamma_i\rightarrow A_i}\colon \M_{\Gamma_i}\rightarrow \M_{A_i}$ is the gluing morphism. 
 
 We can then form the fibered diagram
   \begin{equation}\label{eq:excessdecomp}
  \begin{tikzcd}
   \coprod_{B\in \mathfrak{G}_{A_{i+1},\Gamma_i}} \coprod_{(\Gamma_{i+1},G,f_{i+1})\in \mathfrak{H}_{B,G,\xi}} \H_{(\Gamma_{i+1},G)} \arrow[r,"\coprod \xi'_{i+1}"]\arrow[d,"\coprod \phi_f"] & \H_{(\Gamma_i,G)}\arrow[d,"\phi_{\Gamma_i}"]\\
   \coprod_{B\in \mathfrak{G}_{A_{i+1},\Gamma_i}} \M_{B}\arrow[d] \arrow[r] &\M_{\Gamma_i}\arrow[d,"\xi_{\Gamma_i\rightarrow A_{i}}"]\\
   \M_{A_{i+1}} \arrow[r,"\xi_{i+1}"]& \M_{A_i}
  \end{tikzcd}.
 \end{equation}
 where the lower square is Cartesian by Proposition \ref{intprop} and the upper square is Cartesian by Proposition \ref{prop:fibreprod}.
 
  Since $\codim_{\M_{A_i}} \im (\xi_{i+1})  =1$ the codimension of $\xi'_{i+1}( \H_{(\Gamma_{i+1},G)})$ in $ \H_{(\Gamma_{i},G)}$ is either 1 or 0. If it is 1 then there is no excess on this irreducible component.
 
 If the codimension is 0 then $ \H_{(\Gamma_{i+1},G)}= \H_{(\Gamma_{i},G)}$, i.e.\ $(\Gamma_{i+1},G)=(\Gamma_i,G)$. Restricting Diagram \eqref{eq:excessdecomp} we get
    \begin{center}
  \begin{tikzcd}
    \H_{(\Gamma_{i+1},G)} \arrow[r,"\textup{id}_{\H} "]\arrow[d,"\phi_{\Gamma_{i+1}}"] &  \H_{(\Gamma_{i},G)}\arrow[d,"\phi_{\Gamma_i}"]\\
   \M_{\Gamma_i}\arrow[d,"\xi_{\Gamma_i\rightarrow A_{i+1}}"] \arrow[r,"\textup{id}_{\M}"] &\M_{\Gamma_i}\arrow[d]\\
   \M_{A_{i+1}} \arrow[r,"\xi_{i+1}"]& \M_{A_i}
  \end{tikzcd}.
 \end{center}
 We deduce that 
 \[
  \phi_{\Gamma_{i+1}}^*\xi_{\Gamma_i\rightarrow A_{i+1}}^*N_{\xi_{i+1}}/N_{\textup{id}_{\H}}= \phi_{\Gamma_{i+1}}^* \left( \xi_{\Gamma_i\rightarrow A_{i+1}}N_{\xi_{i+1}}/N_{\textup{id}_{\M}} \right) = \phi^*_{\Gamma_{i+1}}(\L_{h} \otimes \L_{h'})
 \]
 where $(h,h')$ is the edge of $\Gamma$ corresponding to the edge of $A_{i+1}$ glued together by the morphism $\xi_{i+1}$ and where the second equality is due to \ref{prgrph:excess}. 
 
 Now note that this situation only occurs when $A_{i+1}$ is a specialization of $\Gamma_{i}$. In other words this situation occurs exactly if and only if the node glued together by $\xi_{i+1}$ is already a node of $\Gamma$, which can only happen if and only if the edge glued by $\xi_{i+1}$ is in the $G$ orbit of one of the edges glued by $\M_{A_i}\rightarrow \M_{g,r}$.
 \end{proof}

From the excess intersection formula, Proposition \ref{prop:fibreprod} and Proposition \ref{prop:excess} we now obtain:
 
 \begin{theorem}\label{th:main}
  We have
  \begin{align*}
   \xi_A^*\phi_{*} ([\H_{g,G,\xi}])=  \sum_{ (\Gamma,G,f) \in \mathfrak{H}_{A;G,\xi}} 
  \phi_{f*}( c_{\textup{top}}E_f
   \cap
   [\H_{(\Gamma,G)}]).
  \end{align*}
    where $c_{\text{top}}(E_f)$ is as in Proposition \ref{prop:excess}.
 \end{theorem}
Note that for an $A$-structure $f=(\alpha, \beta, \gamma)$ appearing in the above sum and a factor $(- \psi_h - \psi_{h'})$ from $c_{\text{top}}(E_f)$ (where $(h,h')$ is an edge of $\Gamma$ in $\im \beta$), we have, by definition of $\psi$ classes on spaces of admissible covers
\[- \psi_h - \psi_{h'} = \phi_f^* (- \psi_{\beta^{-1}(h)} - \psi_{\beta^{-1}(h')}).\]
In other words the class $c_{\text{top}}(E_f)$ is a pullback of a suitable combination of $\psi$-classes on $\M_A$. By the projection formula and slight abuse of notation, we can therefore also write
\begin{align*}
   \xi_A^*\phi_{*} ([\H_{g,G,\xi}])=  \sum_{ (\Gamma,G,f) \in \mathfrak{H}_{A;G,\xi}} 
  c_{\textup{top}}E_f \cap \phi_{f*} [\H_{(\Gamma,G)}].
  \end{align*}
 


 \subsection{Pushing Forward by the Target Morphism}

 Let 
 \begin{equation}\label{diag:deltapsi}
     \begin{tikzcd}
     \H_{g,G,\xi} \arrow[r,"\phi"]\arrow[d,"\delta"] & \M_{g,r}\\
     \M_{g',b}
     \end{tikzcd}
 \end{equation}
 be the diagram from Theorem \ref{Thm:Hurwitzstack} and let $[A_\theta]$ be a decorated stratum class on $\M_{g,r}$. In certain situations it is useful to compute the pull-push 
 \[
 \delta_* \phi^* [A_\theta]
 \]
 in terms of decorated stratum classes on $\M_{g',b}$. We can now do this by forming the diagram
  \begin{center}
     \begin{tikzcd}
    & \coprod_{(\Gamma,G,f)\in \mathfrak{H}_{A,G,\xi}}\H_{(\Gamma,G)} \arrow[r,"\coprod \phi_{f}"]\arrow[ddl,bend right,"\coprod \delta_{(\Gamma,G)}",swap] \arrow[d,"\coprod \xi_{(\Gamma,G)}",swap] & \M_{A} \arrow[d,"\xi_A"]\\
    & \H_{g,G,\xi} \arrow[r,"\phi"]\arrow[d,"\delta"] & \M_{g,r}\\
    \coprod_{(\Gamma,G,f)\in \mathfrak{H}_{A,G,\xi}}\M_{\Gamma/G}\arrow[r,"\coprod \xi_{\Gamma/G}",swap] & \M_{g',b}
     \end{tikzcd}.
 \end{center}
 To express $\delta_* \phi^* [A_\theta]$ we first observe that
 \[\phi^* [A_\theta] = \phi^* \xi_{A*} \theta = \sum_{(\Gamma,G,f) \in \mathfrak{H}_{A,G,\xi}} \xi_{(\Gamma,G)*} \big(c_{\textup{top}}(E_f) \cdot \phi_f^* \theta\big).\]
 Here we use that the upper right square is Cartesian by Proposition \ref{prop:fibreprod} and the excess bundle of the pullback along $\phi$ is given, for each $(\Gamma,G,f) \in \mathfrak{H}_{A,G,\xi}$,  in terms of a pullback of $\psi$ classes on $\M_\Gamma$, as $c_{\textup{top}}(E_f)$ in Proposition \ref{prop:excess}.
 
 As a next step, the commutativity of the left part of the diagram above implies that
 \[\delta_* \phi^* [A_\theta] = \sum_{(\Gamma,G,f)} \delta_* \xi_{(\Gamma,G)*} \big(c_{\textup{top}}(E_f) \cdot \phi_f^* \theta\big) =\sum_{(\Gamma,G,f)} \xi_{\Gamma/G *} \delta_{(\Gamma,G)*} \big(c_{\textup{top}}(E_f) \cdot \phi_f^* \theta\big) .\]
 
 The term $c_{\textup{top}}(E_f) \cdot \phi_f^* \theta $ is a linear combination of monomials $\alpha_j$ in $\kappa$ and $\psi$-classes on $\H_{(\Gamma,G)}$. Using \eqref{diag:deltapsi} and Lemma \ref{Lem:psikappacompat} we can find for any such monomial $\alpha_j$ a monomial term\footnote{That is, a rational coefficient times a monomial in $\kappa$ and $\psi$-classes.} $\tilde \alpha_j$ in $\kappa$ and $\psi$-classes on $\M_{\Gamma/G}$ such that $\alpha_j = \delta_{(\Gamma,G)}^* \tilde \alpha_j$. Let $F$ be the $\mathbb{Q}$-linear extension of the map $\alpha_j \mapsto \tilde \alpha_j$, such that $\alpha = \delta_{(\Gamma,G)}^* F(\alpha)$ for any polynomial $\alpha$ in $\kappa$ and $\psi$-classes on $\M_{(\Gamma,G)}$. Then we have
 \[\delta_{(\Gamma,G)*} c_{\textup{top}}(E_f) \cdot \phi_f^* \theta = \delta_{(\Gamma,G)*} \delta_{(\Gamma,G)}^* F(c_{\textup{top}}(E_f) \cdot \phi_f^* \theta) = \deg \delta_{(\Gamma,G)} \cdot F(c_{\textup{top}}(E_f) \cdot \phi_f^* \theta)\]
 again by the projection formula.

 Combining these results we get the following combinatorial description of $\delta_*\phi^* [A_\theta]$ by the excess intersection formula and the projection formula:
 \begin{theorem}\label{th:mainpush}
 With the above notation we have 
  \begin{align*}
   \delta_*\phi^{*}\xi_{A*} ([\theta])=  \sum_{ (\Gamma,G,f) \in \mathfrak{H}_{A;G,\xi}} 
   \deg \delta_{(\Gamma,G)} \; \xi_{\Gamma/G*}(F(c_{\textup{top}}E_f \cdot \phi_f^*(\theta))
   \cap
   [\M_{\Gamma/G}]).
  \end{align*}
 \end{theorem}
 Note that $\deg \delta_{(\Gamma,G)}$ is given componentwise in Theorem \ref{thm:condnonempty} (and more explicitly for $G$ cyclic in Proposition \ref{prop:condcyclic}).

 \begin{corollary} \label{Cor:pullpushdelta}
  If $\H_{g,G,\xi}$ is a stack of pointed admissible $G$-covers with source and target morphisms $\phi$ and $\delta$, then in particular Theorem \ref{th:mainpush} implies that the pull-push 
  \[
   \delta_*\phi^* \colon A^\bullet(\M_{g,r})\rightarrow A^\bullet(\M_{g',b})
  \]
 sends tautological classes to tautological classes.
 \end{corollary}

 \subsection{Forgetting Points} \label{Sect:forgetting}
 
 In practice we are usually interested in the image of $\H_{g,G,\xi}$ under the composition
 \begin{equation}\label{eq:phifor}
  \phi_{(h_{k_1},...,h_{k_j})} \colon \H_{g,G,\xi} \xrightarrow{\phi} \M_{g,r} \xrightarrow{\pi} \M_{g,n}
 \end{equation}
 where $\pi$ is the map forgetting the markings associated to $h_{k_1},...,h_{k_j}$ and their $G$-orbits and stabilizing the resulting curve. For example the hyperelliptic locus $\Hyp_g\subset \M_g$ can be defined in our terms as $\phi_{(1^{2g+2})}(\H_{g,\Z/2\Z,(1^{2g+2})})$.
 
 We are therefore interested in computing the intersection between such a locus and a decorated stratum class. If $A$ is an $n$-pointed stable graph we can form the commutative diagram
 \begin{equation}\label{diag:forgetpoints}
     \begin{tikzcd}
     \coprod_{B\in \mathfrak{A}_A} \M_B \arrow[r,"\coprod \xi_B"]\arrow[d,"\coprod \pi_B"] &\M_{g,r}\arrow[d,"\pi"]\\
     \M_A \arrow[r,"\xi_A"] & \M_{g,n}
     \end{tikzcd}
 \end{equation}
 where $\mathfrak{A}_A$ is the set of all isomorphism classes of $r$-pointed graphs $B$ such that after forgetting $r-n$ of the legs of $B$ we get $A$ (recall here that if $f=(\alpha,\beta,\gamma)\colon A'\rightarrow A$ is an isomorphism of stable graphs then the leg assignment must commute with $\alpha$, i.e. $\alpha \circ \zeta_{A'}=\zeta_A$, in particular a permutation of the set of legs of a stable graph $A$ does not neccesarily induce an isomorphism of $A$). We define the map $\pi_B\colon \M_B\rightarrow \M_A$ as the map forgetting the relevant points (and stabilizing) componentwise. 
 
 The diagram \eqref{diag:forgetpoints} is not Cartesian. 
 Indeed $\M_{g,r}$ has connected fibers over $\M_{g,n}$ but $\coprod_{B\in \mathfrak{A}}\M_B to \M_A$ does not have connected fibers. 
 However the map 
 \[
 \xi\colon \coprod_{B\in \mathfrak{A}} \M_B \rightarrow \M_A\times_{\M_{g,n}}\M_{g,r}
 \]
 is proper and generically 1-to-1 (unwrapping definitions we see that $\xi$ is surjective and $\coprod \xi_B$ is generically injective so $\xi$ is as well). It follows that $\pi^*([A])= \sum_{B\in \mathfrak{A}_A} [B]$.
 
 To summarize, we have the following diagram.
  \begin{equation*}
  \begin{tikzcd}
   \coprod_{B\in \mathfrak{A}_A} \coprod_{(\Gamma,G,f)\in \mathfrak{H}_{B,G,\xi}}\H_{(\Gamma,G)}\arrow[r,"\coprod \xi_{(\Gamma,G)}"]\arrow[d,swap,"\coprod \phi_f"]  &\H_{g,G,\xi}\arrow[d,"\phi"]\\
   \coprod_{B\in \mathfrak{A}_A} \M_B \arrow[d,swap,"\coprod \pi_B"] \arrow[r,"\coprod \xi_B"] & \M_{g,r}\arrow[d,"\pi"]\\
   \M_A \arrow[r,"\xi_A"] &\M_{g,n} 
  \end{tikzcd}
 \end{equation*}
 
 Then putting everything together, we get the following generalization of Theorem \ref{th:main}.
 
 \begin{theorem} \label{th:maingen}
 With notation as above we have
  \begin{align} \label{eqn:bdrypullHbarOrig}
   \xi_A^*\pi_*\phi_{*} ([\H_{g,G,\xi}])= \sum_{B \in \mathfrak{A}_A}  \sum_{ (\Gamma,G,f) \in \mathfrak{H}_{B;G,\xi}} (\pi_B)_* c_{\textup{top}}E_f
   \cap
  \phi_{f*}( 
   [\H_{(\Gamma,G)}]).
  \end{align} 
 \end{theorem}
 
 Next we want to generalize Theorem \ref{th:mainpush}. That is, looking at the diagram 
  \begin{equation}
     \begin{tikzcd}
     \H_{g,G,\xi} \arrow[r,"\phi"]\arrow[d,"\delta"] & \M_{g,r}\arrow[r,"\pi"] & \M_{g,n}\\
     \M_{g',b}
     \end{tikzcd}
 \end{equation}
 we want to compute the pull-push $\delta_* \phi^* \pi^* [A_\theta]$ for decorated stratum classes $[A_\theta]$ on $\M_{g,n}$. 
 
 From the diagram \eqref{diag:forgetpoints} it follows that
 \begin{equation} \label{eqn:tautforgetpullback} \pi^* [A_\theta] = \pi^* \xi_{A*} \theta = \sum_{B \in \mathfrak{A}_A} \xi_{B*} \pi_B^* \theta.\end{equation}
 The maps $\pi_B$ are just products of the usual forgetful maps between moduli spaces of stable curves. The pullback of monomials $\theta$ in $\kappa$ and $\psi$-classes is determined by the following well known result
 \begin{proposition} \label{Pro:tautforgetpull}
  Let $\pi\colon \M_{g,n+1}\rightarrow \M_{g,n}$ be the map forgetting marking $n+1$. Then we have
\begin{align*}
  \psi_i &=\pi^{*} (\psi_i) + D_{i,n+1}\\
  \kappa_a &= \pi^* (\kappa_a) -\psi_{n+1}^a
 \end{align*}
where $D_{i,n+1}$ is the divisor which has the markings $i$ and $n+1$ on an irreducible component $R$ of arithmetic genus 0 and all other markings on a component of genus $g$ not containing $R$.
 \end{proposition}
Again, putting things together we can first use equation \eqref{eqn:tautforgetpullback} and Proposition \ref{Pro:tautforgetpull} to compute the pullback $\pi^* [A_\theta]$ in terms of decorated stratum classes $[B_{\theta'}]$ on $\M_{g,r}$ and then  apply Theorem \ref{th:mainpush}  to each of these $[B_{\theta'}]$ and push the result forward through $\delta$.


\subsection{Example Computations}  
In the following we apply the techniques from the last sections to compute some examples of pullbacks for hyperelliptic and bielliptic cycles.

  \begin{notation}
  We will denote by $\Hyp_{g,\hypbiram,2\hypbipair},\B_{g,\hypbiram,2\hypbipair}\subset \M_{g,\hypbiram+2\hypbipair}$ the spaces
  \begin{align*}
      \phi_{(1^{2g+2-\hypbiram})}(&\H_{g,\Z/2\Z,(1^{2g+2},0^\hypbipair)}),\\
      \phi_{(1^{2g-2-\hypbiram})}(&\H_{g,\Z/2\Z,(1^{2g-2},0^\hypbipair)}),
  \end{align*}
  respectively. Note that on the smooth locus these are the hyperelliptic and bielliptic curves of genus $g$ with $\hypbiram$ points fixed by the covering involution and $\hypbipair$ pairs of points conjugate under the covering involution. We will drop $\hypbiram$ and $\hypbipair$ from the notation when they are both zero.
  \end{notation}

  \begin{example}\label{ex:B21H21}
  Let $A$ be the stable graph 
  \begin{center}
   \begin{tikzpicture}[->,>=bad to,baseline=-3pt,node distance=1.3cm,thick,main node/.style={circle,draw,font=\Large,scale=0.5}]
\node [main node] (A) {2};
\node at (1,0) [main node] (B) {2};
\node at (0,.35) (v1) {$v_1$};
\node at (1,.35) (v2) {$v_2$};
\draw [-] (A) to (B);
\end{tikzpicture}
  \end{center}
we will use Theorem \ref{th:main} to compute $\xi_A^* ([\B_4])=\xi_A^*[\phi_{(1^6)}(\H_{4,\Z/2\Z,(1^6)})]$.  We form the diagram
\begin{center}
 \begin{tikzcd}
\coprod_{B\in \mathfrak{A}_A}\coprod_{(\Gamma,G,f)\in \mathfrak{H}_{(B,G,(1^6))}} \H_{(\Gamma,G)}\arrow[r]\arrow[d] & \H_{4,\Z/2\Z,(1^6)}\arrow[d]\\
\coprod_{B\in \mathfrak{A}_A}\M_B\arrow[d]\arrow[r] & \M_{4,6}\arrow[d]\\
\M_A \arrow[r] &\M_{4}
 \end{tikzcd}
\end{center}
where the upper square is Cartesian. The set $\mathfrak{A}_A$ consists of all possible distributions of legs over the vertices of $A$. We claim that from the nonemptyness condition of admissible $G$ graphs it follows that any graph $B$ in $\mathfrak{A}_A$, such that there exists an admissible $G$ graph $(\Gamma,G)$ with generic $B$ structure, will have one leg attached to one of the vertices and all other legs to the other vertex. Indeed since $A$ has a separating edge, any $B\in \mathfrak{A}_A$ has a seperating edge. It follows that for any $(\Gamma,G)$  with a generic $B$ structure the graph $\Gamma$ and thus $\Gamma/G$ must also have a seperating edge. The graph $\Gamma/G$ without the legs is  therefore of the form
\begin{center}
    \begin{tikzpicture}[->,>=bad to,baseline=-3pt,node distance=1.3cm,thick,main node/.style={circle,draw,font=\Large,scale=0.5}]
\node [main node] (A) {1};
\node at (1,0) [main node] (B) {0};
\draw [-] (A) to (B);
\node at (0,.35) (v1) {$v_1$};
\node at (1,.35) (v2) {$v_2$};
\end{tikzpicture}.
\end{center}
From Riemann-Hurwitz it now follows that to satisfy the nonemptyness condion the vertex of genus 1 must have 2 legs and/or half edges and the vertex of genus 0 must have 6 legs and/or half edges attached to it. This proves the claim. The set of graphs in $\mathfrak{A}_A$ having nonzero intersection with $\phi(\H_{4,\Z/2\Z,(1^6)})$ thus consists of graphs of the following form:
\begin{center}
\begin{tabular}{c@{\hskip 1cm}c}
$B_1=\begin{tikzpicture}[->,>=bad to,baseline=-3pt,node distance=1.3cm,thick,main node/.style={circle,draw,font=\Large,scale=0.5}]
\node [main node] (A) {2};
\node at (1,0) [main node] (B) {2};
\node at (0,.35) (v1) {$v_1$};
\node at (1,.35) (v2) {$v_2$};
\draw [-] (A) to (B);
\draw [-] (A) to (-.3,0);
\draw [-] (B) to (1.5,0);
\draw [-] (B) to (1.4,.2);
\draw [-] (B) to (1.4,-.2);
\draw [-] (B) to (1.3,.5);
\draw [-] (B) to (1.3,-.5);
\end{tikzpicture}$
&
$B_2=
\begin{tikzpicture}[->,>=bad to,baseline=-3pt,node distance=1.3cm,thick,main node/.style={circle,draw,font=\Large,scale=0.5}]
\node [main node] (A) {2};
\node at (1,0) [main node] (B) {2};
\node at (0,.35) (v1) {$v_1$};
\node at (1,.35) (v2) {$v_2$};
\draw [-] (A) to (B);
\draw [-] (B) to (1.3,0);
\draw [-] (A) to (-.5,0);
\draw [-] (A) to (-.4,.2);
\draw [-] (A) to (-.4,-.2);
\draw [-] (A) to (-.3,.5);
\draw [-] (A) to (-.3,-.5);
\end{tikzpicture}$
\end{tabular}
\end{center}
The corresponding admissible $G$-graphs  in $\mathfrak{H}_{B_i,\Z/2\Z,(1^6)}$ are, respectively, of the form
\begin{center}
\begin{tabular}{c@{\hskip 1cm}c}
$(\Gamma_1,G)=
\begin{tikzpicture}[->,>=bad to,baseline=-3pt,node distance=1.3cm,thick,main node/.style={circle,draw,font=\Large,scale=0.5}]
\node [main node] (A) {2};
\node at (1,0) [main node] (B) {2};
\node at (-.8,.5) (v1) {$\H_{2,\Z/2/Z,(1^2)}$};
\node at (1.8,.7) (v2) {$\H_{2,\Z/2/Z,(1^6)}$};
\draw [-] (A) to (B);
\draw [-] (A) to (-.3,0);
\draw [-] (B) to (1.5,0);
\draw [-] (B) to (1.4,.2);
\draw [-] (B) to (1.4,-.2);
\draw [-] (B) to (1.3,.5);
\draw [-] (B) to (1.3,-.5);
\end{tikzpicture}$
&
$(\Gamma_2,G)=
\begin{tikzpicture}[->,>=bad to,baseline=-3pt,node distance=1.3cm,thick,main node/.style={circle,draw,font=\Large,scale=0.5}]
\node [main node] (A) {2};
\node at (1,0) [main node] (B) {2};
\node at (-.8,.7) (v1) {$\H_{2,\Z/2/Z,(1^6)}$};
\node at (1.8,.5) (v2) {$\H_{2,\Z/2/Z,(1^2)}$};
\draw [-] (A) to (B);
\draw [-] (B) to (1.3,0);
\draw [-] (A) to (-.5,0);
\draw [-] (A) to (-.4,.2);
\draw [-] (A) to (-.4,-.2);
\draw [-] (A) to (-.3,.5);
\draw [-] (A) to (-.3,-.5);
\end{tikzpicture}$
\end{tabular}
\end{center}
with the obvious  $B_i$-structure on $(\Gamma_i,G)$. Note that the top Chern class of the excess bundle is trivial. We therefore have 
\begin{align*}
 \xi_A^*\phi_{(1^6)*}[\H_{4,\Z/2\Z,(1^6)}]= \binom{6}{1} \bigg(& \phi_{(1)*}[\H_{2,\Z/2\Z,(1^2)}] \otimes \phi_{(1^5)*}[\H_{2,\Z/2\Z,(1^6)}]\\ &+ \phi_{(1^5)*}[\H_{2,\Z/2\Z,(1^6)}] \otimes \phi_{(1)*}[\H_{2,\Z/2\Z,(1^2)}] \bigg).
\end{align*}
 For smooth bielliptic curves of genus 4 the bielliptic involution is unique and generically such curves have no further automorphisms. Therefore the degree of $\phi_{(1^6)*}\colon \H_{4,\Z/2\Z,(1^6)}\rightarrow \M_4$ onto its image is $6!$, the number of possible orderings of the $6$ ramification points of the bielliptic map. Similarly the degree of  $\phi_{(1)*}\colon \H_{2,\Z/2\Z,(1^2)}\rightarrow \M_{2,1}$ to its image is $1$ and the degree of  $\phi_{(1^5)*}\colon \H_{2,\Z/2\Z,(1^2)}\rightarrow \M_{2,1}$ is $5!$. Thus
 \begin{align*}
  \xi_A^*[\B_4] =& \frac{1}{6!} \xi_A^*\phi_{(1^6)*}[\H_{4,\Z/2\Z,(1^6)}]\\
  =&  \frac{\binom{6}{1}}{6!} \big( \phi_{(1)*}[\H_{2,\Z/2\Z,(1^2)}] \otimes \phi_{(1^5)*}[\H_{4,\Z/2\Z,(1^2)}] \\ &+ \phi_{(1^5)*}[\H_{4,\Z/2\Z,(1^6)}] \otimes \phi_{(1)*}[\H_{4,\Z/2\Z,(1^2)}] \big)\\
  =& \frac{1}{5!}\left( [\B_{2,1}]\otimes 5![\Hyp_{2,1}] + 5![\Hyp_{2,1}]\otimes [\B_{2,1}]\right)\\  
  =&[\B_{2,1}]\otimes [\Hyp_{2,1}] + [\Hyp_{2,1}]\otimes [\B_{2,1}] \in A^3(\M_{2,1}\times \M_{2,1}).
 \end{align*}
 This computation will turn out useful later when we compute $[\B_4]$ in Theorem \ref{th:B4}.
 \end{example}

 \begin{example}
 Let 
\[
 A =
\begin{tikzpicture}[->,>=bad to,baseline=-3pt,node distance=1.3cm,thick,main node/.style={circle,draw,font=\Large,scale=0.5}]
\node [main node] (A) {1};
\draw [-] (A) to [out=30, in=-30,looseness=5] (A);
\draw [-] (A) to [out=150, in=-150,looseness=5] (A);
\end{tikzpicture}
\]
and consider the maps $\phi_{(1^8)}\colon \Hyp_{3,\Z/2\Z,(1^8)}\rightarrow \M_3$, $\delta \colon \H_{3,\Z/2\Z,(1^8)}\rightarrow \M_{0,8}$. We will compute $\delta_*\phi_{(1^8)}^*([A])$ using Theorem \ref{th:mainpush}. We will go step by step through the different parts of the equation
\begin{itemize}
 \item First note that $\# \aut A =8$ so $[A]=\tfrac{1}{8}\xi_{A*}([\M_A])$. 

 \item The set $\mathfrak{A}$, introduced in \eqref{diag:forgetpoints},  consists of all possible distributions of the legs $\{1,...,8\}$ over the vertices of $A$. Since $A$ has only one vertex, the set $\mathfrak{A}_A$ consist of a single element.
\[
 \mathfrak{A} = \left\{ 
 B:=\begin{tikzpicture}[->,>=bad to,baseline=-3pt,node distance=1.3cm,thick,main node/.style={circle,draw,font=\Large,scale=0.5},shift={(0,.3)}]
\node [main node] (A) {1};
\node at (0,-.7) (B) {...};
\node at (-.6,-.7) (1) {1};
\node at (.6,-.7) (2) {8};
\node at (-.4,.3) {\footnotesize $i_1$};
\node at (-.55,.-.1) {\footnotesize $i_2$};
\node at (.4,.3) {\footnotesize $i_3$};
\node at (.55,-.1) {\footnotesize $i_4$};
\draw [-] (A) to (0,-.5);
\draw [-] (A) to (.5,-.5);
\draw [-] (A) to (-.5,-.5);
\draw [-] (A) to (.2,-.5);
\draw [-] (A) to (-.2,-.5);
\draw [-] (A) to [out=30, in=-30,looseness=5] (A);
\draw [-] (A) to [out=150, in=-150,looseness=5] (A);
\end{tikzpicture}
 \right\}
\]

\item 
There are two possible types of admissible $\Z/2\Z$-graphs with a generic $B$-structure.
\begin{center}
\begin{tabular}{c@{\hspace{2cm}}c}
$(\Gamma_1,\Z/2\Z)=
    \begin{tikzpicture}[->,>=bad to,baseline=-3pt,node distance=1.3cm,thick,main node/.style={circle,draw,font=\Large,scale=0.5}]
  \begin{scope}
\node [main node] (A) {1};
\node at (-1,0) [scale=.3,draw,circle,fill=black] (B)  {};
\node at (1,0) [scale=.3,draw,circle,fill=black] (C)  {};
\node at (-.8,.4) {\footnotesize $h_1$};
\node at (-.3,.4) {\footnotesize $h_2$};
\node at (-.8,-.4) {\footnotesize $h_3$};
\node at (-.5,-.4) {\footnotesize $h_4$};
\node at (.3,.4) {\footnotesize $h_5$};
\node at (.8,.4) {\footnotesize $h_6$};
\node at (.5,-.4) {\footnotesize $h_7$};
\node at (.9,-.4) {\footnotesize $h_8$};
\draw [-] (A) to (0.1,-.5);
\draw [-] (A) to (-0.1,-.5);
\draw [-] (A) to (-.3,-.4);
\draw [-] (A) to (.3,-.4);
\draw [-] (B) to (-1.2,.2);
\draw [-] (B) to (-1.2,-.2);
\draw [-] (C) to (1.2,.2);
\draw [-] (C) to (1.2,-.2);
\draw [-] (A) to [out=30, in=150] (C);
\draw [-] (A) to [out=-30, in=-150] (C);
\draw [-] (B) to [out=30, in=150] (A);
\draw [-] (B) to [out=-30, in=-150] (A);
\end{scope}
\begin{scope}[shift ={(0,-1.2)}]
\draw [->,line width=.7mm] (0,.5) to (0,-.5);
\end{scope}
  \begin{scope}[shift ={(0,-2)}]
\node at (0,-.2) [scale=.3,draw,circle,fill=black] (A)  {};
\node at (-1,0) [scale=.3,draw,circle,fill=black] (B)  {};
\node at (1,0) [scale=.3,draw,circle,fill=black] (C)  {};
\draw [-] (A) to (0.1,-.5);
\draw [-] (A) to (-0.1,-.5);
\draw [-] (A) to (-.4,-.5);
\draw [-] (A) to (.4,-.5);
\draw [-] (B) to (-1.2,.2);
\draw [-] (B) to (-1.2,-.2);
\draw [-] (C) to (1.2,.2);
\draw [-] (C) to (1.2,-.2);
\draw [-] (A) to (C);
\draw [-] (B) to (A);
\end{scope}
\end{tikzpicture}
$
&
$
(\Gamma_2,\Z/2\Z) =
   \begin{tikzpicture}[->,>=bad to,baseline=-3pt,node distance=1.3cm,thick,main node/.style={circle,draw,font=\Large,scale=0.5}]
  \begin{scope}
\node [scale=.3,draw,circle,fill=black] (A)  {};
\node at (-1,0) [scale=.3,draw,circle,fill=black] (B)  {};
\node at (1,0) [main node] (C) {1};
\draw [-] (A) to (0.1,-.2);
\draw [-] (A) to (-0.1,-.2);
\draw [-] (B) to (-1.2,.2);
\draw [-] (B) to (-1.2,-.2);
\draw [-] (C) to (1.5,.1);
\draw [-] (C) to (1.5,-.1);
\draw [-] (C) to (1.3,.4);
\draw [-] (C) to (1.3,-.4);
\draw [-] (A) to [out=30, in=150] (C);
\draw [-] (A) to [out=-30, in=-150] (C);
\draw [-] (B) to [out=30, in=150] (A);
\draw [-] (B) to [out=-30, in=-150] (A);
\end{scope}
\begin{scope}[shift ={(0,-1.2)}]
\draw [->,line width=.7mm] (0,.5) to (0,-.5);
\end{scope}
  \begin{scope}[shift ={(0,-2)}]
\node at (0,-.2) [scale=.3,draw,circle,fill=black] (A)  {};
\node at (-1,0) [scale=.3,draw,circle,fill=black] (B)  {};
\node at (1,0) [scale=.3,draw,circle,fill=black] (C)  {};
\draw [-] (A) to (0.2,-.5);
\draw [-] (A) to (-0.2,-.5);
\draw [-] (B) to (-1.2,.2);
\draw [-] (B) to (-1.2,-.2);
\draw [-] (C) to (1.2,.1);
\draw [-] (C) to (1.2,-.1);
\draw [-] (C) to (1.1,.3);
\draw [-] (C) to (1.1,-.3);
\draw [-] (A) to (C);
\draw [-] (B) to (A);
\end{scope}
\end{tikzpicture}
$
\end{tabular}.
\end{center}
There are $\frac{1}{2}\binom{8}{2,4,2}$ possible nonisomorphic distributions of the legs for the graph $\Gamma_1$ and  $\binom{8}{2,2,4}$ for $\Gamma_2$.  

\item There are 8 isomorphism classes of $B$-structures on $(\Gamma_1,\Z/2\Z)$. For any generic $B$-structure $f'$ on $(\Gamma_1,\Z/2\Z)$ there exists an automorphism of $B$-structures  \[(\Gamma_1,\Z/2\Z,f')\rightarrow (\Gamma_1,\Z/2\Z,f)\] with $f=(\alpha,\beta,\gamma)$  such that the edges of $B$ are mapped to the edges $(h_1,h_2)$ and $(h_5,h_6)$ under $\beta$. There are 4 possible choices for the image of the half edge $i_1$ and given this choices there are 2 possible images for the halfedge $i_3$. This completely determines a $B$-structure on $(\Gamma_1,\Z/2\Z)$. 
A similar argument shows that there are 8 isomorphism classes of generic $B$-structures for $(\Gamma_2,\mathbb{Z}/2\mathbb{Z})$.

\item We have $\deg \delta_{(\Gamma_1,\Z/2\Z)}=\deg \delta_{(\Gamma_2,\Z/2\Z)}=2$ by Proposition \ref{prop:condcyclic}.

\item For all generic $B$-structures $f$ on $\Gamma_1$ and $\Gamma_2$, the set $\im \beta_f \cap G(\im \beta_f)$ is empty, so the top Chern class of the excess bundle is 1. 

\item The graph $A$ is undecorated, in other words $\theta=1$. Therefore we have \[\delta_{(\Gamma_i,\Z/2\Z)*}\phi^*_{f_i}\pi_B^*(\theta)=\deg \delta_{(\Gamma_i,\Z/2\Z)}=2\].

\item The image of $\delta_* \phi_(1^8)^*\colon A^\bullet(\M_{3})\rightarrow A^\bullet(\M_{0,8})$ lies in the $\S_8$ invariant part $A^\bullet(\M_{0,8})^{\S_8}$ of $A^\bullet(\M_{0,8})$. Denote by $d_{2,4,2},d_{4,2,2}\in A^\bullet(\M_{0,8})^{\S_8}$ the classes given by taking the sum over all nonisomorphic dual graphs of the form (respectively):
\begin{center}
 \begin{tabular}{c@{\hspace{2cm}}c}
\begin{tikzpicture}[->,>=bad to,baseline=-3pt,node distance=1.3cm,thick,main node/.style={circle,draw,font=\Large,scale=0.5}]
\node at (0,-.2) [scale=.3,draw,circle,fill=black] (A)  {};
\node at (-1,0) [scale=.3,draw,circle,fill=black] (B)  {};
\node at (1,0) [scale=.3,draw,circle,fill=black] (C)  {};
\draw [-] (A) to (0.1,-.5);
\draw [-] (A) to (-0.1,-.5);
\draw [-] (A) to (-.4,-.5);
\draw [-] (A) to (.4,-.5);
\draw [-] (B) to (-1.2,.2);
\draw [-] (B) to (-1.2,-.2);
\draw [-] (C) to (1.2,.2);
\draw [-] (C) to (1.2,-.2);
\draw [-] (A) to (C);
\draw [-] (B) to (A);
\end{tikzpicture}
   &
\begin{tikzpicture}[->,>=bad to,baseline=-3pt,node distance=1.3cm,thick,main node/.style={circle,draw,font=\Large,scale=0.5}]
\node at (0,-.2) [scale=.3,draw,circle,fill=black] (A)  {};
\node at (-1,0) [scale=.3,draw,circle,fill=black] (B)  {};
\node at (1,0) [scale=.3,draw,circle,fill=black] (C)  {};
\draw [-] (A) to (0.2,-.5);
\draw [-] (A) to (-0.2,-.5);
\draw [-] (B) to (-1.2,.2);
\draw [-] (B) to (-1.2,-.2);
\draw [-] (C) to (1.2,.1);
\draw [-] (C) to (1.2,-.1);
\draw [-] (C) to (1.1,.3);
\draw [-] (C) to (1.1,-.3);
\draw [-] (A) to (C);
\draw [-] (B) to (A);
\end{tikzpicture}
 \end{tabular}
\end{center}
Note that there are $\frac{1}{2} \binom{8}{2,4,2}$ irreducible boundary components in $d_{2,4,2}$ and $ \binom{8}{2,2,4}$ in $d_{4,2,2}$.

Putting everything together we have 
\begin{align*}
 \delta_*\phi_{(1^8)}^*([A])&= \frac{1}{8}\cdot 1 \left(\frac{1}{2} \binom{8}{2,4,2}\cdot 16 \cdot 2 \binom{8}{2,4,2}^{-1} d_{2,4,2} + \binom{8}{2,2,4}\cdot 16 \cdot  \binom{8}{2,2,4}^{-1} d_{2,4,2}\right) \\
 &= 2d_{2,4,2} + 2 d_{4,2,2}\in R^2(\M_{0,8})
 \end{align*}
\end{itemize}
 \end{example}

 More examples are given in \cite[Section 2.3]{thesisjason}.

\section{Computing Tautological Expressions for Admissible Cover Cycles}\label{ch:computing}

Let $\phi_{\tilde{\xi}}\colon \H_{g,G,\xi} \rightarrow \M_{g,n}$ be the composition $\pi \circ \phi$ as in \eqref{eq:phifor}, let $\H:= \im \phi_{\tilde{\xi}}$ and assume that the class $[\H]$ is tautological. We can then try to compute $[\H]$ in terms of decorated stratum classes. In this section we will outline a few possible methods using Theorem \ref{th:main} and Theorem \ref{th:mainpush} to do this.

\subsection{Perfect Pairing}\label{sec:perfectpairing}

The intersection pairing $H^k(\M_{g,n})\otimes H^{2\dim \M_{g,n}-k}(\M_{g,n})\rightarrow \C$ is perfect since $\M_{g,n}$ is a smooth complete Deligne-Mumford stack. The Gorenstein conjecture asks whether this pairing restricted to the tautological rings $RH^{2k}(\M_{g,n})\otimes RH^{2\dim \M_{g,n}-2k}(\M_{g,n})$ is perfect as well. As was first shown in \cite{Petersen2014b} the Gorenstein conjecture is false in general.  However for many low values of $g$, $n$ and $k$ it is known that $H^{2k}(\M_{g,n})=RH^{2k}(\M_{g,n})$ and  $H^{2\dim \M_{g,n}-2k}(\M_{g,n})=RH^{2\dim \M_{g,n}-2k}(\M_{g,n})$ so that the pairing in this degree is also perfect when restricted to the tautological ring. In particular this is the case 
\begin{itemize}
    \item for $g=0$ \cite{Keel} and $g=1$ \cite{petersen2014} and arbitrary $n$ as well as $g=2, n \leq 19$ \cite{petersen2016},
    \item for $k=0,1$ \cite{arbarellocornalba},
    \item for $(g,n)=(3,0), (3,1), (3,2), (4,0)$. Here the ranks of the cohomology groups $H^{2k}(\M_{g,n})$ have been computed \cite{Getzler1998, getzlerhodge, bergstromptcts, Bergstrom2007} and agree with the ranks of the intersection pairing on $RH^{2k}(\M_{g,n})$ computed in \cite{Yang2008}.
\end{itemize}

Assume $[\H] \in H^{2k}(\M_{g,n})$ is tautological and that the intersection pairing of $\M_{g,n}$ restricted to the tautological ring is perfect in degree $2k$. Then we can compute this class in the following way. Suppose we have fixed bases $\{D_i\}$ and $\{\hat{D}_j\}$ for $RH^{2k}(\M_{g,n})$ and for $RH^{2\dim \M_{g,n}-2k}(\M_{g,n})$ where $D_i$ and $\hat{D}_j$ are decorated stratum classes. As described in Remark \ref{Rmk:intnumb} we can compute the intersection matrix
\[
 M = (D_i\cdot \hat{D}_j)_{ij}.
\]
By assumption, there is a unique expression $[\H] = \sum a_i D_i$, which we want to compute. It must satisfy
\[
 ([\H] \cdot \hat{D}_j)_j = \left(\sum_i a_i D_i\cdot \hat{D}_j\right)_j.
\] 
We can compute the numbers on the left hand side using
Theorem \ref{th:mainpush}. Indeed, for any decorated stratum class $\hat{D}_j=[A_\theta]$, the theorem allows us to express the zero cycle $\delta_* \phi_{\tilde \xi}^* [A_\theta]$ in terms of tautological classes. Again using Remark \ref{Rmk:intnumb} we can compute the degree of this tautological zero cycle. 

Finally, since $M$ is an invertible matrix by assumption, we can solve this expression for the $a_i$.


\begin{example}
 We can use this method to compute the class of the hyperelliptic locus in the genus 3 case explicitly. 
 
 A basis for $A^1(\M_3)$ is given by $\lambda$, $\delta_0$ and $\delta_1$ (or $\kappa_1$, $\delta_0$ and $\delta_1$, note that $12 \lambda = \kappa_1+\delta_0 +\delta_1$). A basis in terms of decorated stratum classes for $A^5(\M_3)$ is given by (see \cite{Faber1990}) 
 \begin{center}
  \begin{tabular}{c@{\hskip 1 cm}c@{\hskip 1 cm}c}
    $\hat{D}_1 = \Bigg[\begin{tikzpicture}[->,>=bad to,baseline=-3pt,node distance=1.3cm,thick,main node/.style={circle,draw,font=\Large,scale=0.5}]
\node [scale=.3,draw,circle,fill=black] (A) {};
\node [scale=.3,draw,circle,fill=black] [right of =A] (B) {};
\node [scale=.3,draw,circle,fill=black] [right of =B] (C) {};
\draw [-] (A) to [out=150, in=-150,looseness=20] (A);
\draw [-] (C) to [out=30, in=-30,looseness=20] (C);
\draw [-] (A) to (B);
\draw [-] (B) to [out=40, in=140] (C);
\draw [-] (B) to [out=-40, in=-140] (C);
\end{tikzpicture}\Bigg]$
& 
    $ \hat{D}_2 = \Bigg[\begin{tikzpicture}[->,>=bad to,baseline=-3pt,node distance=1.3cm,thick,main node/.style={circle,draw,font=\Large,scale=0.5}]
\node [scale=.3,draw,circle,fill=black] (A) {};
\node [scale=.3,draw,circle,fill=black] [right of =A] (B) {};
\node [scale=.3,draw,circle,fill=black] [right of =B] (C) {};
\draw [-] (A) to [out=150, in=-150,looseness=20] (A);
\draw [-] (C) to [out=30, in=-30,looseness=20] (C);
\draw [-] (B) to [out=-120, in=-60,looseness=20] (B);
\draw [-] (A) to (B);
\draw [-] (B) to (C);
\end{tikzpicture}\Bigg]$
&
  $ \hat{D}_3 = \Bigg[\begin{tikzpicture}[->,>=bad to,baseline=-3pt,node distance=1.3cm,thick,main node/.style={circle,draw,font=\Large,scale=0.5}]
\node [scale=.3,draw,circle,fill=black] (A) {};
\node [scale=.3,draw,circle,fill=black] [right of =A] (B) {};
\node [scale=.3,draw,circle,fill=black] [right of =B] (C) {};
  \node [main node] [left of =A] (D) {1};
\draw [-] (C) to [out=30, in=-30,looseness=20] (C);
\draw [-] (A) to (D);
\draw [-] (C) to (B);
\draw [-] (A) to [out=40, in=140] (B);
\draw [-] (A) to [out=-40, in=-140] (B);
\end{tikzpicture}\Bigg]$
  \end{tabular}
   \end{center}
  The intersection numbers between the bases have been computed in \cite[pg. 418]{Faber1990} (or use Theorem \ref{th:prod}): 
  \begin{center}
  \begin{tabular}{c@{\hskip 1 cm}c@{\hskip 1 cm}c}
   $\lambda \cdot \hat{D}_1 = 0$ &    $\lambda \cdot \hat{D}_2 = 0$ &    $\lambda \cdot \hat{D}_3 = \frac{1}{96}$ \\
   $\delta_0 \cdot \hat{D}_1 = -\frac{1}{4}$ &    $\delta_0 \cdot \hat{D}_2 = 0$ &    $\delta_0 \cdot \hat{D}_3 = \frac{1}{8}$ \\
      $\delta_1 \cdot \hat{D}_1 = \frac{1}{8}$ &    $\delta_1 \cdot \hat{D}_2 = -\frac{1}{16}$ &    $\delta_1 \cdot \hat{D}_3 = -\frac{1}{96}$
  \end{tabular}
 \end{center}
 Recall that we denote by $[\Hyp_3]$ the class $\tfrac{1}{8!} \phi_{(1^6)*}[\H_{3,\Z_2,(1^6)}]$. We leave it as an exercise to the reader to compute the intersection numbers
 \begin{center}
    \begin{tabular}{c@{\hskip 1 cm}c@{\hskip 1 cm}c}
        $[\Hyp_3] \cdot \hat{D}_1 = -\frac{1}{8}$ &    $[\Hyp_3] \cdot \hat{D}_2 =\frac{3}{16}$ &    $[\Hyp_3] \cdot \hat{D}_3 = 0$
    \end{tabular}
 \end{center}
using Theorem \ref{th:main} or Theorem \ref{th:mainpush}. Solving the resulting system of equations we get the well known expression
\begin{align*}
 [\Hyp_3]=  (9\lambda - \delta_0 - 3\delta_1)\in A^1(\M_{3})
\end{align*}
\end{example}

\begin{remark}
 The Sage program \cite{Schmittprogram} has this technique implemented as a function. However it is computationally not the most efficient method. It suffices to compute $[\B_{2,2,0}]$ in a short amount of time but most classes which are more complicated require too much computing time and memory. The problem is that in calculating the intersection of $[\im \phi_{\tilde{\xi}*}]$ with a decorated stratum classes $[A_\theta]$ we first have to pull back $[A_\theta]$ through the forgetful map $\pi \colon \M_{g,r}\rightarrow \M_{g,n}$. If the number $r-n$  is very high then the pullback $\pi^{*}[A_\theta]$ in terms of decorated stratum classes involves a large number of different classes which leads to a large increase in computation time. Moreover, many of these classes live on complicated stable graphs. The classification of all $G$-admissible graphs with large codimension is another computationally intensive operation.
\end{remark}

\subsection{Repeated Pullback to Lower Genus}\label{sec:lowergenus}

A second method for studying cohomology classes such as $\alpha= \phi_{\tilde{\xi}*} [\H_{g,G,\xi}]$ on $\M_{g,n}$ is to pull them back via the boundary gluing maps $\xi_A$ for stable graphs $A$. As the following result of Arbarello and Cornalba shows, knowing all the pullbacks $\xi_A^* \alpha$ is in many cases sufficient to determine $\alpha$ uniquely.
\begin{lemma}[Lemma 2.5, 2.6, \cite{accohomology}]\label{lem:pullboundrange}
  Let $\mathcalorig{A}$ be the set of stable graphs $A$ of genus $g$ with $n$ markings having exactly one edge. Then the pullback
  \[H^k(\M_{g,n}) \xrightarrow{\prod \xi_A^*} \prod_{A \in \mathcalorig{A}} H^k(\M_A)\]
  via all boundary gluing maps $\xi_A : \M_A \to \M_{g,n}$ is injective for $k \leq d(g,n)$ with
  \[d(g,n)=\begin{cases}
  n-4 &\text{if }g=0,\\
  2g-2 &\text{if }n=0,\\
  2g-3+n &\text{if }g>0, n>0.
  \end{cases}
  \]
\end{lemma}

\begin{remark}
 In particular note that $\dim \Hyp_g=2g-1$ and $\dim \B_g=2g-2$ and so $[\Hyp_g]\in H^{2g-4}(\M_g)$ and $[\B_g]\in H^{2g-2}(\M_g)$. Hence the classes of the (unmarked) hyperelliptic and the bielliptic  loci are always in the range of Lemma \ref{lem:pullboundrange}.
\end{remark}

In a lot of ways, computing the pullback by $\xi_A$ gives us finer information than the method of Section \ref{sec:perfectpairing}, since knowing the pullback $\xi_A^*\alpha$ determines the intersection number of $\alpha$ with any decorated stratum class $A_\theta$ on the graph $A$. The only intersection numbers we cannot obtain this way are those with pure polynomials in $\kappa$ and $\psi$-classes on $\M_{g,n}$.

So for every stable graph $A$ we form the diagram
 \begin{equation*}
  \begin{tikzcd}
   \coprod_{B\in \mathfrak{A}} \coprod_{(\Gamma,G,f)\in \mathfrak{H}_{B,G,\xi}}\H_{(\Gamma,G)}\arrow[r,"\coprod \xi_{(\Gamma,G)}"]\arrow[d,swap,"\coprod \phi_f"]  &\H_{g,G,\xi}\arrow[d,"\phi"]\\
   \coprod_{B\in \mathfrak{A}} \M_B \arrow[d,swap,"\coprod \pi_B"] \arrow[r,"\coprod \xi_B"] & \M_{g,r}\arrow[d,"\pi"]\\
   \M_A \arrow[r,"\xi_A"] &\M_{g,n} 
  \end{tikzcd}.
 \end{equation*}
Using Theorem \ref{th:main} we can express the pullback  $\xi_A^*\pi_*\phi_* ([\H_{g,G,\xi}])$ as 
  \begin{align} \label{eqn:bdrypullHbar}
   \xi_A^*\pi_*\phi_{*} ([\H_{g,G,\xi}])= \sum_{B \in \mathfrak{A}}  \sum_{ (\Gamma,G,f) \in \mathfrak{H}_{B;G,\xi}} (\pi_B)_* c_{\textup{top}}E_f
   \cap
  \phi_{f*}( 
   [\H_{(\Gamma,G)}]).
  \end{align}
  
Our goal is to express the right hand side of (\ref{eqn:bdrypullHbar}) entirely in terms of tautological classes. As shown in Proposition \ref{prop:excess}, the cycle $c_{\textup{top}}E_f$ is a combination of $\psi$-classes, which are tautological. Also we know how the pushforward $(\pi_B)_*$ acts on tautological classes. Now observe that the map $\phi_f$ is the composition $\xi_f \circ \phi_{(\Gamma,G)}$ and we also know how the partial gluing map $\xi_f$ acts on tautological classes. Thus effectively, we have to be able to express terms of the form $(\phi_{(\Gamma,G)})_* [\H_{(\Gamma,G)}]$.

Hence we can reduce to finding a tautological expression for $(\phi_{(\Gamma,G)})_* [\H_{(\Gamma,G)}]$. The space $\H_{(\Gamma,G)}$ itself was a product of moduli spaces $\H_{g(v),G_v,\xi_v}$ for a set of representatives $v$ of the $G$-action on $\Gamma$. Since these spaces have smaller dimension than $\H_{g,G,\xi}$, we assume that the corresponding cycles $\phi_* [\H_{g(v),G_v,\xi_v}] \in H^*(\M_{g(v),n(v)})$ are tautological and that we have computed them in terms of the generators of the tautological ring.

From the definition of $\phi_{(\Gamma,G)}$ we saw in (\ref{eqn:phiGammaGfactor}) that it factors as a product of the maps $\phi$ for the spaces $\H_{g(v),G_v,\xi_v}$ composed with diagonal morphisms $\Delta : \M_{g(v),n(v)} \to \prod_{w \in Gv} \M_{g(v),n(v)}$, where $Gv$ is the orbit of $v \in V(\Gamma)$ under the action of $G$. By Application \ref{app:tautkunneth}, we can express the class of the diagonal in terms of (tensor products of) tautological classes in the case that the entire cohomology $H^*(\M_{g(v),n(v)})$ is tautological. As discussed in Application \ref{app:tautkunneth}, this is the case for many small $(g(v),n(v))$. 


Suppose that the cycles $\phi_* [\H_{g(v),G_v,\xi_v}]$ and the classes of the diagonals appearing in (\ref{eqn:bdrypullHbar}) are tautological. Then we can determine  $\xi_A^*\pi_*\phi_{*}([\H_{g,G,\xi}])$ in terms of decorated stratum classes. Using Theorem \ref{th:prod} we can on the other hand compute $\xi^*_A\colon R^\bullet(\M_{g,n})\rightarrow R^\bullet(\M_A)$ explicitly as a linear map in terms of (generating sets of) decorated stratum classes. Thus knowing $\xi_A^*\pi_*\phi_{*}([\H_{g,G,\xi}])$  determines the class $\pi_*\phi_{*}([\H_{g,G,\xi}])$ in terms of decorated stratum classes up to the kernel of $\xi^*_A$. 

Note that in order to compute the kernel of $\xi_A^*$, we need to identify all relations between decorated stratum classes on $\M_{g,n}$ and on the factors of $\M_A$. While in principle these relations can be determined from the intersection form (in case that all cohomology in the corresponding degrees is tautological), this is too computationally intensive in most cases. Instead, we use the implementation of the generalized Faber-Zagier relations from \cite{Pixtprogram}, which are known to be a complete set of relations in many cases.

If we can do this for a number of graphs $A_i$ and if
\[
 \bigcap \ker \xi_{A_i}^* = \{0\}
\]
then we can completely determine the class $\pi_*\phi_* ([\H_{g,G,\xi}])$. In some cases where the intersection of kernels is not trivial, we can obtain some further linear restrictions on $\pi_*\phi_* ([\H_{g,G,\xi}])$ by computing intersection numbers with polynomials in $\kappa$ and $\psi$-classes on $\M_{g,n}$ as described in Section \ref{sec:perfectpairing}.

As an example we will now compute the class of the hyperelliptic locus in genus 5.

\begin{theorem}\label{th:H5}
 In $R^3(\M_5)$ we have 

\input{H5class.tex}
\end{theorem}

 
 \begin{proof}[Proof of Theorem \ref{th:H5}]
  Consider the following two graphs
    \begin{center}
   \begin{tabular}{c@{\hskip 2cm}c}
      $A_1=
   \begin{tikzpicture}[->,>=bad to,baseline=-3pt,node distance=1.3cm,thick,main node/.style={circle,draw,font=\Large,scale=0.5}]
\node [main node] (A) {2};
\node[main node] (B) [left of=A] {2};
\draw [-] (B) to [out=40, in=140] (A);
\draw [-] (B) to [out=-40, in=-140] (A);
\end{tikzpicture}$
  & $A_2=
 \begin{tikzpicture}[->,>=bad to,baseline=-3pt,node distance=1.3cm,thick,main node/.style={circle,draw,font=\Large,scale=0.5}]
\node[main node] (A) {3};
\node[main node] (B) [right of=A] {2};
\draw [-] (A) to (B);
\end{tikzpicture}$.
   \end{tabular}
  \end{center}
  Using Theorem \ref{th:prod} (and \cite{Schmittprogram}) the pullbacks 
  \begin{align*}
  \xi^*_{A_1}\colon A^3(\M_4) & \to A^3 (\M_{2,2}\times \M_{2,2})\\
  \xi^*_{A_2}\colon A^3(\M_4) & \to A^3 (\M_{3,1}\times \M_{2,1})
    \end{align*}
  can be computed in terms of decorated stratum classes. It turns out that $\ker \xi^*_{A_1}\cap \xi^*_{A_2}=\{0\}$. By Theorem \ref{th:main} we have
 \begin{align*}
 \xi_{A_1}^* [\Hyp_5]&=
  -\Bigg[
\begin{tikzpicture}[->,>=bad to,baseline=-3pt,node distance=1.3cm,thick,main node/.style={circle,draw,font=\Large,scale=0.5}]
\node at (1.4,0) [main node] (A) {2};
\node at (0,0) [main node] (B) {2};
\node at (-.2,.5) (Hi1) {$\Hyp_{2,0,2}$};
\node at (1.6,.5) (Hi2) {$\Hyp_{2,0,2}$};
\draw [-] (B) to [out=40, in=140] (A);
\draw [<-] (B) to [out=-40, in=-140] (A);
\draw [<->] (.7,-.2) to (.7,.2);
\end{tikzpicture}\Bigg]
  -\Bigg[
\begin{tikzpicture}[->,>=bad to,baseline=-3pt,node distance=1.3cm,thick,main node/.style={circle,draw,font=\Large,scale=0.5}]
\node at (1.4,0) [main node] (A) {2};
\node at (0,0) [main node] (B) {2};
\node at (-.2,.5) (Hi1) {$\Hyp_{2,0,2}$};
\node at (1.6,.5) (Hi2) {$\Hyp_{2,0,2}$};
\draw [-] (B) to [out=40, in=140] (A);
\draw [->] (B) to [out=-40, in=-140] (A);
\draw [<->] (.7,-.2) to (.7,.2);
\end{tikzpicture}\Bigg]\\
&=-\psi_1\cap[\Hyp_{2,0,2}]\otimes [\Hyp_{2,0,2}]- [\Hyp_{2,0,2}]\otimes \psi_1\cap[\Hyp_{2,0,2}]\in R^\bullet(\M_{2,2}\times \M_{2,2}).
 \end{align*}
The class $[\Hyp_{2,0,2}]\in  A^1(\M_{2,2})$ was originally computed in \cite[Lemma 6]{Belorousski1984}. We have:
\begin{align*}
 [\Hyp_{2,0,2}] = 
 \Bigg[\begin{tikzpicture}[->,>=bad to,baseline=-3pt,node distance=1.3cm,thick,main node/.style={circle,draw,font=\Large,scale=0.5}]
\node [main node] (A) {2};
\node at (.5,.4) (n1) {1};
\node at (.5,-.4) (n2) {2};
\draw [<-] (A) to (n1);
\draw [-] (A) to (n2);
\end{tikzpicture}\Bigg]
+
 \Bigg[\begin{tikzpicture}[->,>=bad to,baseline=-3pt,node distance=1.3cm,thick,main node/.style={circle,draw,font=\Large,scale=0.5}]
\node [main node] (A) {2};
\node at (.5,.4) (n1) {1};
\node at (.5,-.4) (n2) {2};
\draw [-] (A) to (n1);
\draw [<-] (A) to (n2);
\end{tikzpicture}\Bigg]
- 3 
 \Bigg[\begin{tikzpicture}[->,>=bad to,baseline=-3pt,node distance=1.3cm,thick,main node/.style={circle,draw,font=\Large,scale=0.5}]
 \node [scale=.3,draw,circle,fill=black] (A) {};
\node [main node] (B) [left of =A] {2};
\node at (.5,.4) (n1) {1};
\node at (.5,-.4) (n2) {2};
\draw [-] (A) to (B);
\draw [-] (A) to (n1);
\draw [-] (A) to (n2);
\end{tikzpicture}\Bigg]
-\frac{6}{5}
 \Bigg[\begin{tikzpicture}[->,>=bad to,baseline=-3pt,node distance=1.3cm,thick,main node/.style={circle,draw,font=\Large,scale=0.5}]
 \node [main node] (A) {1};
\node [main node] (B) [left of =A] {1};
\node at (.5,.4) (n1) {1};
\node at (.5,-.4) (n2) {2};
\draw [-] (A) to (B);
\draw [-] (A) to (n1);
\draw [-] (A) to (n2);
\end{tikzpicture}\Bigg]
-\frac{1}{5}
 \Bigg[\begin{tikzpicture}[->,>=bad to,baseline=-3pt,node distance=1.3cm,thick,main node/.style={circle,draw,font=\Large,scale=0.5}]
 \node [main node] (A) {1};
\node [main node] (B) [left of =A] {1};
\node at (0,.6) (n1) {1};
\node at (-.65,.6) (n2) {2};
\draw [-] (A) to (B);
\draw [-] (A) to (n1);
\draw [-] (B) to (n2);
\end{tikzpicture}\Bigg]
-\frac{1}{10}
 \Bigg[\begin{tikzpicture}[->,>=bad to,baseline=-3pt,node distance=1.3cm,thick,main node/.style={circle,draw,font=\Large,scale=0.5}]
\node [main node] (A) {1};
\node at (.5,.4) (n1) {1};
\node at (.5,-.4) (n2) {2};
\draw [-] (A) to (n1);
\draw [-] (A) to (n2);
\draw [-] (A) to [out=150, in=-150,looseness=5] (A);
\end{tikzpicture}\Bigg].
\end{align*}
This immediately gives $\psi_1\cap [\Hyp_{2,0,2}]$ as a class in terms of decorated stratum classes. We thus know $\xi_{A_1}^*([\Hyp_5])$ in terms of decorated stratum classes. 

By Theorem \ref{th:main} we have 
\[
 \xi_{A_2}^* ([\Hyp_5]) = [\Hyp_{3,1,0}\times \Hyp_{2,1,0}]\in H^6(\M_{3,1}\times \M_{2,1})
\]
The class $[\Hyp_{2,1,0}]$ has originally been computed in \cite[Theorem 2.2]{Eisenbud1987}. We have
\begin{align*}
 [\Hyp_{2,1,0}]= 3\psi_1 - \frac{6}{5}
  \Bigg[\begin{tikzpicture}[->,>=bad to,baseline=-3pt,node distance=1.3cm,thick,main node/.style={circle,draw,font=\Large,scale=0.5}]
 \node [main node] (A) {1};
\node [main node] (B) [left of =A] {1};
\draw [-] (A) to (B);
\draw [-] (A) to (0,.5);
\end{tikzpicture}\Bigg]
- \frac{1}{10}
 \Bigg[\begin{tikzpicture}[->,>=bad to,baseline=-3pt,node distance=1.3cm,thick,main node/.style={circle,draw,font=\Large,scale=0.5}]
 \node [main node] (A) {1};
\draw [-] (A) to (.5,0);
\draw [-] (A) to [out=150, in=-150,looseness=5] (A);
\end{tikzpicture}\Bigg].
\end{align*}
The class $[\Hyp_{3,1,0}]$ can again be computed using similar pullbacks, or alternatively from an expression for $[\Hyp_4]$ in terms of decorated stratum classes. The class $[\Hyp_4]$ was first computed in \cite[Proposition 5]{Faber2005}. Let
\[
B=
 \begin{tikzpicture}[->,>=bad to,baseline=-3pt,node distance=1.3cm,thick,main node/.style={circle,draw,font=\Large,scale=0.5}]
\node[main node] (A) {3};
\node[main node] (B) [right of=A] {1};
\draw [-] (A) to (B);
\end{tikzpicture}.
\]
Again using Theorem $\ref{th:main}$ we see that 
\[
\xi_A^*([\Hyp_4]= [\Hyp_{3,1,0}]\otimes [\Hyp_{1,1,0}] \in A^2(\M_{3,1})\otimes A^0(\M_{1,1})
\]
From this we can compute:

 \noindent \resizebox{\textwidth}{!}{
\begin{tabular}{c@{\hskip 0cm}r@{\hskip 0cm}l@{\hskip 0cm}r@{\hskip 0cm}l@{\hskip 0cm}r@{\hskip 0cm}l@{\hskip 0cm}r@{\hskip 0cm}l@{\hskip 0cm}r@{\hskip 0cm}l@{\hskip 0cm}r@{\hskip 0cm}l}
 $[\Hyp_{3,1,0}]=$
 &
 $6$
 &
  $ \Bigg[\begin{tikzpicture}[->,>=bad to,baseline=-3pt,node distance=1.3cm,thick,main node/.style={circle,draw,font=\Large,scale=0.5}]
 \node [main node] (A) {3};
\draw [<<-] (A) to (0,.5);
\end{tikzpicture}\Bigg]$
 &
 $-\frac{24}{7}$
 &
 $\Bigg[\begin{tikzpicture}[->,>=bad to,baseline=-3pt,node distance=1.3cm,thick,main node/.style={circle,draw,font=\Large,scale=0.5}]
 \node [main node] (A) {2};
\node [main node] (B) [left of =A] {1};
\draw [-] (A) to (B);
\draw [<-] (A) to (0,.5);
\end{tikzpicture}\Bigg]$
 &
 $-\frac{1}{7}$
 &
  $\Bigg[\begin{tikzpicture}[->,>=bad to,baseline=-3pt,node distance=1.3cm,thick,main node/.style={circle,draw,font=\Large,scale=0.5}]
 \node [main node] (A) {2};
\node [main node] (B) [left of =A] {1};
\draw [<-] (A) to (B);
\draw [-] (A) to (0,.5);
\end{tikzpicture}\Bigg]$
&
$-\frac{1}{10}$
&
 $\Bigg[\begin{tikzpicture}[->,>=bad to,baseline=-3pt,node distance=1.3cm,thick,main node/.style={circle,draw,font=\Large,scale=0.5}]
 \node [main node] (A) {1};
\node [main node] (B) [right of =A] {2};
\draw [->] (A) to (B);
\draw [-] (A) to (0,.5);
\end{tikzpicture}\Bigg]$
&
$-\frac{53}{7}$
&
 $\Bigg[\begin{tikzpicture}[->,>=bad to,baseline=-3pt,node distance=1.3cm,thick,main node/.style={circle,draw,font=\Large,scale=0.5}]
  \node [scale=.3,draw,circle,fill=black] (A) {};
 \node [main node] (B) [left of =A] {2};
\node [main node] (C) [right of =A] {1};
\draw [-] (A) to (B);
\draw [-] (A) to (C);
\draw [-] (A) to (0,.5);
\end{tikzpicture}\Bigg]$
&
$+\frac{48}{35}$
&
 $\Bigg[\begin{tikzpicture}[->,>=bad to,baseline=-3pt,node distance=1.3cm,thick,main node/.style={circle,draw,font=\Large,scale=0.5}]
  \node [main node] (A) {1};
 \node [main node] (B) [left of =A] {1};
\node [main node] (C) [right of =A] {1};
\draw [-] (A) to (B);
\draw [-] (A) to (C);
\draw [-] (A) to (0,.5);
\end{tikzpicture}\Bigg]$
\\
&
$+\frac{54}{35}$
&
 $\Bigg[\begin{tikzpicture}[->,>=bad to,baseline=-3pt,node distance=1.3cm,thick,main node/.style={circle,draw,font=\Large,scale=0.5}]
  \node [main node] (A) {1};
 \node [main node] (B) [right of =A] {1};
\node [main node] (C) [right of =B] {1};
\draw [-] (B) to (C);
\draw [-] (A) to (B);
\draw [-] (A) to (0,.5);
\end{tikzpicture}\Bigg]$
&
$-\frac{2}{7} $
&
$ \Bigg[\begin{tikzpicture}[->,>=bad to,baseline=-3pt,node distance=1.3cm,thick,main node/.style={circle,draw,font=\Large,scale=0.5}]
 \node [main node] (A) {2};
\draw [<-] (A) to (.5,0);
\draw [-] (A) to [out=150, in=-150,looseness=5] (A);
\end{tikzpicture}\Bigg]$
&
$-\frac{6}{7}$
&
$ \Bigg[\begin{tikzpicture}[->,>=bad to,baseline=-3pt,node distance=1.3cm,thick,main node/.style={circle,draw,font=\Large,scale=0.5}]
  \node [scale=.3,draw,circle,fill=black] (A) {};
 \node [main node] (B) [left of = A] {2};
\draw [-] (A) to (.5,0);
\draw [-] (A) to [out=150, in=30] (B);
\draw [-] (A) to [out=-150, in=-30] (B);
\end{tikzpicture}\Bigg]$
&
$+\frac{1}{84}$
&
$ \Bigg[\begin{tikzpicture}[->,>=bad to,baseline=-3pt,node distance=1.3cm,thick,main node/.style={circle,draw,font=\Large,scale=0.5}]
 \node [main node] (A)  {2};
 \node [scale=.3,draw,circle,fill=black] [left of = A] (B) {};
\draw [-] (A) to (.5,0);
\draw [-] (A) to  (B);
\draw [-] (B) to [out=150, in=-150,looseness=30] (B);
\end{tikzpicture}\Bigg]$
& 
$-\frac{53}{84}$
&
$ \Bigg[\begin{tikzpicture}[->,>=bad to,baseline=-3pt,node distance=1.3cm,thick,main node/.style={circle,draw,font=\Large,scale=0.5}]
  \node [scale=.3,draw,circle,fill=black]  (B) {};
 \node [main node] (A) [left of =B] {2};
\draw [-] (B) to (0,.5);
\draw [-] (A) to  (B);
\draw [-] (B) to [out=30, in=-30,looseness=30] (B);
\end{tikzpicture}\Bigg]$
&
$+\frac{4}{35}$
&
$ \Bigg[\begin{tikzpicture}[->,>=bad to,baseline=-3pt,node distance=1.3cm,thick,main node/.style={circle,draw,font=\Large,scale=0.5}]
 \node [main node] (A)  {1};
 \node [main node] [left of = A] (B) {1};
\draw [-] (A) to (0,.5);
\draw [-] (A) to  (B);
\draw [-] (A) to [out=-30, in=30,looseness=5] (A);
\end{tikzpicture}\Bigg]$
\\
&
$+\frac{9}{70}$
&
$ \Bigg[\begin{tikzpicture}[->,>=bad to,baseline=-3pt,node distance=1.3cm,thick,main node/.style={circle,draw,font=\Large,scale=0.5}]
 \node [main node] (A)  {1};
 \node [main node] [right of = A] (B) {1};
\draw [-] (A) to (0,.5);
\draw [-] (A) to  (B);
\draw [-] (B) to [out=30, in=-30,looseness=5] (B);
\end{tikzpicture}\Bigg]$
&
$-\frac{1}{35}$
&
$ \Bigg[\begin{tikzpicture}[->,>=bad to,baseline=-3pt,node distance=1.3cm,thick,main node/.style={circle,draw,font=\Large,scale=0.5}]
 \node [main node] (A)  {1};
 \node [main node] [right of = A] (B) {1};
\draw [-] (A) to (0,.5);
\draw [-] (A) to [out=30, in=150] (B);
\draw [-] (A) to [out=-30, in=-150] (B);
\end{tikzpicture}\Bigg]$
&
$+\frac{1}{105}$
&
$ \Bigg[\begin{tikzpicture}[->,>=bad to,baseline=-3pt,node distance=1.3cm,thick,main node/.style={circle,draw,font=\Large,scale=0.5}]
 \node [main node] (A) {1};
\draw [-] (A) to (.5,0);
\draw [-] (A) to [out=90, in=150,looseness=5] (A);
\draw [-] (A) to [out=-90, in=-150,looseness=5] (A);
\end{tikzpicture}\Bigg]$.
\end{tabular}
}
 \end{proof}

 \begin{remark}\label{Rm:H6}
  In the same way we can find an expression for $[\Hyp_6]\in R^4(\M_4)$ in terms of decorated stratum classes. Let 
    \begin{center}
   \begin{tabular}{c@{\hskip 2cm}c}
      $A_1=
   \begin{tikzpicture}[->,>=bad to,baseline=-3pt,node distance=1.3cm,thick,main node/.style={circle,draw,font=\Large,scale=0.5}]
\node [main node] (A) {3};
\node[main node] (B) [left of=A] {2};
\draw [-] (B) to [out=40, in=140] (A);
\draw [-] (B) to [out=-40, in=-140] (A);
\end{tikzpicture}$
  & $A_2=
 \begin{tikzpicture}[->,>=bad to,baseline=-3pt,node distance=1.3cm,thick,main node/.style={circle,draw,font=\Large,scale=0.5}]
\node[main node] (A) {4};
\node[main node] (B) [right of=A] {2};
\draw [-] (A) to (B);
\end{tikzpicture}$.
   \end{tabular}
  \end{center}
  With the help of a computer one verifies that, restricted to $R^4(\M_6)$, we have 
  \[
  \ker \xi^*_{A_1} \cap \ker \xi^*_{A_2}=\{0\}.
  \]
  We have 
  \[ \xi_{A_1}^* [\Hyp_6]=
  -\Bigg[
\begin{tikzpicture}[->,>=bad to,baseline=-3pt,node distance=1.3cm,thick,main node/.style={circle,draw,font=\Large,scale=0.5}]
\node at (1.4,0) [main node] (A) {2};
\node at (0,0) [main node] (B) {3};
\node at (-.2,.5) (Hi1) {$\Hyp_{3,0,2}$};
\node at (1.6,.5) (Hi2) {$\Hyp_{2,0,2}$};
\draw [-] (B) to [out=40, in=140] (A);
\draw [<-] (B) to [out=-40, in=-140] (A);
\draw [<->] (.7,-.2) to (.7,.2);
\end{tikzpicture}\Bigg]
  -\Bigg[
\begin{tikzpicture}[->,>=bad to,baseline=-3pt,node distance=1.3cm,thick,main node/.style={circle,draw,font=\Large,scale=0.5}]
\node at (1.4,0) [main node] (A) {2};
\node at (0,0) [main node] (B) {3};
\node at (-.2,.5) (Hi1) {$\Hyp_{3,0,2}$};
\node at (1.6,.5) (Hi2) {$\Hyp_{2,0,2}$};
\draw [-] (B) to [out=40, in=140] (A);
\draw [->] (B) to [out=-40, in=-140] (A);
\draw [<->] (.7,-.2) to (.7,.2);
\end{tikzpicture}\Bigg]
\]
and
  \[
 \xi_{A_2}^* ([\Hyp_6]) = [\Hyp_{4,1,0}\times \Hyp_{2,1,0}]
\]
the classes $[\Hyp_{3,0,2}]$ and $[\Hyp_{4,1,0}]$ can readily be computed from our expression for $[\Hyp_5]$. We thus obtain an expression for $[\Hyp_6]$. The set of decorated stratum classes  in $R^{4}(\M_{6})$ up to Pixton's relations is of cardinality $376$ (so conjecturally $\dim R^4(\M_6)=376$). To save ink we shall therefore not write down an explict expression for $[\Hyp_6]$ here, but it is available upon request.

The number of decorated stratum classes in $R^5(\M_7)$ is 7078. Our computer is not powerful enough to compute all the Pixton relations between these decorated stratum classes. Although most likely with more serious computing power computing these relations and computing $\Hyp_7$ is still possible. In general, with our method the main difficulty for computing $[\Hyp_g] \in R^{g-2}(\M_g)$  seems to be the computation of Pixton's relations. That is, in all cases where we can compute these relations - both on $\M_g$ and on the spaces $\M_{g_i,n_i}$ appearing in the boundary divisors - we can also compute $[\Hyp_g]$.
 \end{remark}

\subsection{Example: the Bielliptic Locus in Genus \texorpdfstring{$4$}{4}} \label{prgrph:invB4}
As an application of the methods of the previous section, we now compute the class $[\B_4] \in H^6(\M_4)$ of the locus of bielliptic curves in genus $4$. 

The dimension of $H^6(\M_4)$ is~32 by \cite[Theorem 1]{Bergstrom2007}. By \cite[Page 11]{Yang2008} the dimension of $RH^6(\M_4)$ is also 32, hence the cohomology group is completely tautological. A basis for $H^6(\M_4)$ in terms of decorated stratum classes is given by taking the last 32 generators of the relation given in \cite[Proposition 2]{Yang2008}. 

\begin{theorem}\label{th:B4}
  In $H^6(\M_4)$ we have
  
 \noindent \resizebox{\textwidth}{!}{
\begin{tabular}{c@{\hskip 0cm}r@{\hskip 0cm}l@{\hskip 0cm}r@{\hskip 0cm}l@{\hskip 0cm}r@{\hskip 0cm}l@{\hskip 0cm}r@{\hskip 0cm}l@{\hskip 0cm}r@{\hskip 0cm}l@{\hskip 0cm}r@{\hskip 0cm}l}
 $[\B_4]=$ &$480$ &
 $\Bigg[ \begin{tikzpicture}[->,>=bad to,baseline=-3pt,node distance=1.3cm,thick,main node/.style={circle,draw,font=\Large,scale=0.5}]
\node[main node] (A) {2};
\node[main node] (B) [right of=A] {2};
\draw [<->] (A) to (B);
\end{tikzpicture}\Bigg]$ & $
+180
$ & $\Bigg[ \begin{tikzpicture}[->,>=bad to,baseline=-3pt,node distance=1.3cm,thick,main node/.style={circle,draw,font=\Large,scale=0.5}]
\node[main node] (A) {2};
\node[main node] (C) [left of=A] {1};
\node[main node] (B) [right of=A] {1};
\draw [-] (A) to (B);
\draw [<-] (A) to (C);
\end{tikzpicture}\Bigg]$ & $
-354
$ & $\Bigg[ \begin{tikzpicture}[->,>=bad to,baseline=-3pt,node distance=1.3cm,thick,main node/.style={circle,draw,font=\Large,scale=0.5}]
\node[main node] (A) {1};
\node[main node] (C) [left of=A] {2};
\node[main node] (B) [right of=A] {1};
\draw [-] (A) to (B);
\draw [->] (A) to (C);
\end{tikzpicture}\Bigg]$ & $
-36
$ & $\Bigg[ \begin{tikzpicture}[->,>=bad to,baseline=-3pt,node distance=1.3cm,thick,main node/.style={circle,draw,font=\Large,scale=0.5}]
\node [scale=.3,draw,circle,fill=black] (A) {};
\node[main node] (C) [above left of=A] {1};
\node[main node] (D) [below left of=A] {1};
\node[main node] (B) [right of=A] {2};
\draw [-] (A) to (B);
\draw [-] (A) to (C);
\draw [-] (A) to (D);
\end{tikzpicture}\Bigg]$ & $
-378
$ & $\Bigg[ \begin{tikzpicture}[->,>=bad to,baseline=-3pt,node distance=1.3cm,thick,main node/.style={circle,draw,font=\Large,scale=0.5}]
\node [main node] (A) {1};
\node[main node] (C) [above left of=A] {1};
\node[main node] (D) [below left of=A] {1};
\node[main node] (B) [right of=A] {1};
\draw [-] (A) to (B);
\draw [-] (A) to (C);
\draw [-] (A) to (D);
\end{tikzpicture}\Bigg]$ & $
+\frac{816}{5}
$ & $\Bigg[ \begin{tikzpicture}[->,>=bad to,baseline=-3pt,node distance=1.3cm,thick,main node/.style={circle,draw,font=\Large,scale=0.5}]
\node [main node] (A) {1};
\node[main node] (C) [left of=A] {1};
\node[main node] (D) [left of=C] {1};
\node[main node] (B) [right of=A] {1};
\draw [-] (A) to (B);
\draw [-] (A) to (C);
\draw [-] (C) to (D);
\end{tikzpicture}\Bigg]$ 
\\ & 
$-7
$ & $\Bigg[ \begin{tikzpicture}[->,>=bad to,baseline=-3pt,node distance=1.3cm,thick,main node/.style={circle,draw,font=\Large,scale=0.5}]
\node[main node] (A) {3};
\draw [->>] (A) to [out=200, in=160,looseness=10] (A);
\end{tikzpicture}\Bigg]$ & $
+\frac{7}{3}
$ & $\Bigg[ \begin{tikzpicture}[->,>=bad to,baseline=-3pt,node distance=1.3cm,thick,main node/.style={circle,draw,font=\Large,scale=0.5}]
\node[main node] (A) {2};
\draw [->] (A) to [out=200, in=160,looseness=10] (A);
\draw [-] (A) to [out=20, in=-20,looseness=10] (A);
\end{tikzpicture}\Bigg]$ & $
-\frac{133}{16}
$ & $\Bigg[ \begin{tikzpicture}[->,>=bad to,baseline=-3pt,node distance=1.3cm,thick,main node/.style={circle,draw,font=\Large,scale=0.5}]
\node [scale=.3,draw,circle,fill=black] (A) {};
\node[main node] (B) [left of=A] {3};
\node [above right] at (B) {$\kappa_1$};
\draw [-] (A) to (B);
\draw [-] (A) to [out=20, in=-20,looseness=30] (A);
\end{tikzpicture}\Bigg]$ & $
+\frac{665}{16}
$ & $\Bigg[ \begin{tikzpicture}[->,>=bad to,baseline=-3pt,node distance=1.3cm,thick,main node/.style={circle,draw,font=\Large,scale=0.5}]
\node [scale=.3,draw,circle,fill=black] (A) {};
\node[main node] (B) [left of=A] {3};
\draw [->] (A) to (B);
\draw [-] (A) to [out=20, in=-20,looseness=30] (A);
\end{tikzpicture}\Bigg]$ & $
+\frac{75}{4}
$ & $\Bigg[ \begin{tikzpicture}[->,>=bad to,baseline=-3pt,node distance=1.3cm,thick,main node/.style={circle,draw,font=\Large,scale=0.5}]
\node [main node] (A) {1};
\node[main node] (B) [left of=A] {2};
\draw [-] (A) to (B);
\draw [->] (B) to [out=200, in=160,looseness=10] (B);
\end{tikzpicture}\Bigg]$ & $
+\frac{19}{6}
$ & $\Bigg[ \begin{tikzpicture}[->,>=bad to,baseline=-3pt,node distance=1.3cm,thick,main node/.style={circle,draw,font=\Large,scale=0.5}]
\node [main node] (A) {1};
\node[main node] (B) [left of=A] {2};
\draw [->] (A) to (B);
\draw [-] (B) to [out=200, in=160,looseness=10] (B);
\end{tikzpicture}\Bigg]$ 
\\ & $ 
-\frac{310}{3}
$ & $\Bigg[ \begin{tikzpicture}[->,>=bad to,baseline=-3pt,node distance=1.3cm,thick,main node/.style={circle,draw,font=\Large,scale=0.5}]
\node [main node] (A) {1};
\node[main node] (B) [left of=A] {2};
\draw [<-] (B) to [out=40, in=140] (A);
\draw [-] (B) to [out=-40, in=-140] (A);
\end{tikzpicture}\Bigg]$ & $
-\frac{136}{3}
$ & $\Bigg[ \begin{tikzpicture}[->,>=bad to,baseline=-3pt,node distance=1.3cm,thick,main node/.style={circle,draw,font=\Large,scale=0.5}]
\node [main node] (A) {1};
\node[main node] (B) [left of=A] {2};
\draw [->] (A) to (B);
\draw [-] (A) to [out=20, in=-20,looseness=10] (A);
\end{tikzpicture}\Bigg]$ & $
-\frac{37}{120}
$ & $\Bigg[ \begin{tikzpicture}[->,>=bad to,baseline=-3pt,node distance=1.3cm,thick,main node/.style={circle,draw,font=\Large,scale=0.5}]
\node[main node] (A) {1};
\draw [-] (A) to [out=100, in=140,looseness=10] (A);
\draw [-] (A) to [out=-100, in=-140,looseness=10] (A);
\draw [-] (A) to [out=20, in=-20,looseness=10] (A);
\end{tikzpicture}\Bigg]$ & $
+\frac{133}{144}
$ & $\Bigg[ \begin{tikzpicture}[->,>=bad to,baseline=-3pt,node distance=1.3cm,thick,main node/.style={circle,draw,font=\Large,scale=0.5}]
\node [scale=.3,draw,circle,fill=black] (A) {};
\node[main node] (B) [left of=A] {2};
\draw [-] (A) to (B);
\draw [-] (A) to [out=20, in=-20,looseness=30] (A);
\draw [-] (B) to [out=160, in=-160,looseness=10] (B);
\end{tikzpicture}\Bigg]$ & $
-\frac{9}{2}
$ & $\Bigg[ \begin{tikzpicture}[->,>=bad to,baseline=-3pt,node distance=1.3cm,thick,main node/.style={circle,draw,font=\Large,scale=0.5}]
\node [scale=.3,draw,circle,fill=black] (A) {};
\node[main node] (B) [left of=A] {2};
\draw [-] (A) to (B);
\draw [-] (B) to [out=40, in=140] (A);
\draw [-] (B) to [out=-40, in=-140] (A);
\end{tikzpicture}\Bigg]$ & $
+\frac{20}{9}
$ & $\Bigg[ \begin{tikzpicture}[->,>=bad to,baseline=-3pt,node distance=1.3cm,thick,main node/.style={circle,draw,font=\Large,scale=0.5}]
\node [scale=.3,draw,circle,fill=black] (A) {};
\node[main node] (B) [left of=A] {2};
\draw [-] (B) to [out=40, in=140] (A);
\draw [-] (B) to [out=-40, in=-140] (A);
\draw [-] (A) to [out=20, in=-20,looseness=30] (A);
\end{tikzpicture}\Bigg]$ 
\\ & $
+\frac{101}{36}
$ & $\Bigg[ \begin{tikzpicture}[->,>=bad to,baseline=-3pt,node distance=1.3cm,thick,main node/.style={circle,draw,font=\Large,scale=0.5}]
\node [scale=.3,draw,circle,fill=black] (A) {};
\node[main node] (B) [left of=A] {2};
\draw [-] (A) to (B);
\draw [-] (A) to [out=40, in=80,looseness=30] (A);
\draw [-] (A) to [out=-40, in=-80,looseness=30] (A);
\end{tikzpicture}\Bigg]$ & $
-\frac{85}{24}
$ & $\Bigg[ \begin{tikzpicture}[->,>=bad to,baseline=-3pt,node distance=1.3cm,thick,main node/.style={circle,draw,font=\Large,scale=0.5}]
\node [main node] (A) {1};
\node[main node] (B) [left of=A] {1};
\draw [-] (A) to (B);
\draw [-] (A) to [out=40, in=80,looseness=10] (A);
\draw [-] (A) to [out=-40, in=-80,looseness=10] (A);
\end{tikzpicture}\Bigg]$ & $
-\frac{221}{30}
$ & $\Bigg[ \begin{tikzpicture}[->,>=bad to,baseline=-3pt,node distance=1.3cm,thick,main node/.style={circle,draw,font=\Large,scale=0.5}]
\node [main node] (A) {1};
\node[main node] (B) [left of=A] {1};
\draw [-] (B) to [out=40, in=140] (A);
\draw [-] (B) to [out=-40, in=-140] (A);
\draw [-] (B) to [out=160, in=-160,looseness=10] (B);
\end{tikzpicture}\Bigg]$ & $
+\frac{26}{15}
$ & $\Bigg[ \begin{tikzpicture}[->,>=bad to,baseline=-3pt,node distance=1.3cm,thick,main node/.style={circle,draw,font=\Large,scale=0.5}]
\node [main node] (A) {1};
\node[main node] (B) [left of=A] {1};
\draw [-] (B) to (A);
\draw [-] (B) to [out=160, in=-160,looseness=10] (B);
\draw [-] (A) to [out=20, in=-20,looseness=10] (A);
\end{tikzpicture}\Bigg]$ & $
+\frac{243}{10}
$ & $\Bigg[ \begin{tikzpicture}[->,>=bad to,baseline=-3pt,node distance=1.3cm,thick,main node/.style={circle,draw,font=\Large,scale=0.5}]
\node [main node] (A) {1};
\node[main node] (B) [left of=A] {1};
\draw [-] (B) to [out=40, in=140] (A);
\draw [-] (B) to [out=-40, in=-140] (A);
\draw [-] (B) to (A);
\end{tikzpicture}\Bigg]$ & $
-\frac{57}{16}
$ & $\Bigg[ \begin{tikzpicture}[->,>=bad to,baseline=-3pt,node distance=1.3cm,thick,main node/.style={circle,draw,font=\Large,scale=0.5}]
\node [scale=.3,draw,circle,fill=black] (A) {};
\node[main node] (B) [left of=A] {2};
\node[main node] (C) [left of=B] {1};
\draw [-] (C) to (B);
\draw [-] (A) to (B);
\draw [-] (A) to [out=20, in=-20,looseness=30] (A);
\end{tikzpicture}\Bigg]$ 
\\ & $
-\frac{45}{16}
$ & $\Bigg[ \begin{tikzpicture}[->,>=bad to,baseline=-3pt,node distance=1.3cm,thick,main node/.style={circle,draw,font=\Large,scale=0.5}]
\node [scale=.3,draw,circle,fill=black] (A) {};
\node[main node] (B) [left of=A] {1};
\node[main node] (C) [left of=B] {2};
\draw [-] (C) to (B);
\draw [-] (A) to (B);
\draw [-] (A) to [out=20, in=-20,looseness=30] (A);
\end{tikzpicture}\Bigg]$ & $
-\frac{421}{12}
$ & $\Bigg[ \begin{tikzpicture}[->,>=bad to,baseline=-3pt,node distance=1.3cm,thick,main node/.style={circle,draw,font=\Large,scale=0.5}]
\node [scale=.3,draw,circle,fill=black] (A) {}; (A) {1};
\node[main node] (B) [right of=A] {2};
\node[main node] (C) [left of=A] {1};
\draw [-] (A) to [out=40, in=140] (B);
\draw [-] (A) to [out=-40, in=-140] (B);
\draw [-] (A) to (C);
\end{tikzpicture}\Bigg]$ & $
-\frac{70}{3}
$ & $\Bigg[ \begin{tikzpicture}[->,>=bad to,baseline=-3pt,node distance=1.3cm,thick,main node/.style={circle,draw,font=\Large,scale=0.5}]
\node [scale=.3,draw,circle,fill=black] (A) {}; (A) {1};
\node[main node] (B) [right of=A] {1};
\node[main node] (C) [left of=A] {2};
\draw [-] (A) to [out=40, in=140] (B);
\draw [-] (A) to [out=-40, in=-140] (B);
\draw [-] (A) to (C);
\end{tikzpicture}\Bigg]$ & $
+\frac{37}{3}
$ & $\Bigg[ \begin{tikzpicture}[->,>=bad to,baseline=-3pt,node distance=1.3cm,thick,main node/.style={circle,draw,font=\Large,scale=0.5}]
\node [scale=.3,draw,circle,fill=black] (A) {};
\node[main node] (B) [above left of=A] {1};
\node[main node] (C) [below left of=A] {2};
\draw [-] (A) to (B);
\draw [-] (A) to (C);
\draw [-] (A) to [out=20, in=-20,looseness=30] (A);
\end{tikzpicture}\Bigg]$ & $
-\frac{191}{5}
$ & $\Bigg[ \begin{tikzpicture}[->,>=bad to,baseline=-3pt,node distance=1.3cm,thick,main node/.style={circle,draw,font=\Large,scale=0.5}]
\node [main node] (A) {1};
\node[main node] (B) [above left of=A] {1};
\node[main node] (C) [below left of=A] {1};
\draw [-] (A) to (B);
\draw [-] (A) to (C);
\draw [-] (A) to [out=20, in=-20,looseness=10] (A);
\end{tikzpicture}\Bigg]$ & $
+17
$ & $\Bigg[ \begin{tikzpicture}[->,>=bad to,baseline=-3pt,node distance=1.3cm,thick,main node/.style={circle,draw,font=\Large,scale=0.5}]
\node [main node] (A) {1};
\node[main node] (B) [left of=A] {1};
\node[main node] (C) [left of=B] {1};
\draw [-] (A) to (B);
\draw [-] (B) to (C);
\draw [-] (A) to [out=20, in=-20,looseness=10] (A);
\end{tikzpicture}\Bigg]$ 
\\ & $
-\frac{251}{4}
$ & $\Bigg[ \begin{tikzpicture}[->,>=bad to,baseline=-3pt,node distance=1.3cm,thick,main node/.style={circle,draw,font=\Large,scale=0.5}]
\node [main node] (A) {1};
\node[main node] (B) [right of=A] {1};
\node[main node] (C) [left of=A] {1};
\draw [-] (A) to [out=40, in=140] (B);
\draw [-] (A) to [out=-40, in=-140] (B);
\draw [-] (A) to (C);
\end{tikzpicture}\Bigg]$ & $
-\frac{2019}{10}
$ & $\Bigg[ \begin{tikzpicture}[baseline=4pt,->,>=bad to,node distance=1.3cm,thick,main node/.style={circle,draw,font=\Large,scale=0.5}]
\node [main node] (A) {1};
\node[main node] (B) [above right of=A] {1};
\node[main node] (C) [above left of=A] {1};
\draw [-] (A) to  (B);
\draw [-] (C) to  (B);
\draw [-] (A) to (C);
\end{tikzpicture}\Bigg]$ &
\end{tabular}
}
 \end{theorem}

  \begin{remark}
  We only know that the equality of Theorem \ref{th:B4} holds in cohomology. Because the entire cohomology ring is tautological we know that the cohomology class $[\B_4]$ is tautological. However the authors are unaware of any argument that the corresponding Chow class is tautological. 
  
 \end{remark}

By Riemann-Hurwitz, a bielliptic involution on a genus $4$ curve has $6$ fixed points, so the natural space describing such curves is $\H_{4,\Z_2,(1^6)}$. As a first step, we need to compare the fundamental class $[\B_4]$ with the pushforward $ \phi_{(1^6)*} [\H_{4,\Z_2,(1^6)}]$ supported on the same locus.


Next we want to show that on the generic bielliptic curve $C$ of genus $4$, the bielliptic involution is unique. Assume otherwise, i.e. that $C$ has two distinct bielliptic involutions $\tau_1, \tau_2$. Then inside $\aut(C)$ they generate a (finite) subgroup. Since $\tau_1^2=\mathrm{id}_C = \tau_2^2$, this subgroup must be a dihedral group of the form $D_{2n}$, where the additional relation $(\tau_1 \tau_2)^n=\mathrm{id}_C$ holds.

We now consider the Riemann-Hurwitz formula for the quotient map $C \to C/D_{2n}$. Let $g'$ be the genus of $C/D_{2n}$, then since there exists a map $C/\langle \tau_1 \rangle \to C/D_{2n}$, given by taking a further quotient, we have $g' \leq 1$, since by assumption  $C/\langle \tau_1 \rangle$ has genus $1$. Let $b$ be the number of branch points of $C \to C/D_{2n}$ and let $h_1, \ldots, h_b \in D_{2n}$ be generators of the stabilizers of some point over these $b$ branch points. Then Riemann-Hurwitz tells us
\[6=2 \cdot 4 -2 = (2n)\cdot (2g'-2) + \sum_{i=1}^b \left(2n- \frac{2n}{\operatorname{ord} h_i} \right).\]
Now $\operatorname{ord} h_i \geq 2$, so $2n- \frac{2n}{\operatorname{ord} h_i} \geq 2n-n=n$ and we conclude
\[6 \geq 2n(2g'-2)+bn \implies \frac{6}{n} + 2(2-2g') \geq b.\]
Since $\tau_1, \tau_2$ were distinct by assumption, we have $n \geq 2$, which implies $b \leq 7$ in the case $g'=0$ and $b \leq 3$ in the case $g'=1$. Correspondingly, the moduli spaces $\M_{g',b}$ parametrizing the curve $C/D_{2n}$ with marked branch points have dimensions at most $4$ and $3$, respectively. Since this dimension is the same as the corresponding space $\H_{4,D_{2n},\xi'}$ parametrizing the curve $C$ with $D_{2n}$-action, the general bielliptic curve of genus $4$ cannot have such an action of a dihedral group $D_{2n}$, since the space of bielliptic curves is of dimension $6$. Thus the bielliptic involution on the general bielliptic curve is unique.
  
We conclude that the degree of the map $\phi_{(1^6)}: \H_{4,\Z_2,(1^6)} \to \B_4 \subset \M_4$ to its image is given by the number of ways to order the $6$ fixed points of the bielliptic involution, which is $|\S_6|=6!$. To summarize we have
   \[
   [\B_4] = \frac{1}{6!}  \phi_{(1^6)*} [\H_{4,\Z_2,(1^6)}].
   \]

We determine this class by pullback to three distinct boundary strata.  Let 
  \begin{center}
   \begin{tabular}{c@{\hskip 2cm}c@{\hskip 2cm}c}
      $A=
   \begin{tikzpicture}[->,>=bad to,baseline=-3pt,node distance=1.3cm,thick,main node/.style={circle,draw,font=\Large,scale=0.5}]
\node [main node] (A) {1};
\node[main node] (B) [left of=A] {2};
\draw [-] (B) to [out=40, in=140] (A);
\draw [-] (B) to [out=-40, in=-140] (A);
\end{tikzpicture}$
  & $B=
 \begin{tikzpicture}[->,>=bad to,baseline=-3pt,node distance=1.3cm,thick,main node/.style={circle,draw,font=\Large,scale=0.5}]
\node[main node] (A) {2};
\node[main node] (B) [right of=A] {2};
\draw [-] (A) to (B);
\end{tikzpicture}$
& $C=
\begin{tikzpicture}[->,>=bad to,baseline=-3pt,node distance=1.3cm,thick,main node/.style={circle,draw,font=\Large,scale=0.5}]
\node[main node] (A) {3};
\node[main node] (B) [right of=A] {1};
\draw [-] (A) to (B);
\end{tikzpicture}.$
   \end{tabular}
  \end{center}
The pullback maps 
\[
\xi_A^*\colon H^6(\M_4)\rightarrow H^6(\M_A), \quad  \xi_B^*\colon H^6(\M_4)\rightarrow H^6(\M_B), \quad \text{and} \quad \xi_C^*\colon H^6(\M_4)\rightarrow H^6(\M_C)
\]
 can be computed in terms of bases of decorated stratum classes  by Theorem \ref{th:prod}, using the generalized Faber-Zagier relations calculated by \cite{Pixtprogram}. Inside the vector space $H^6(\M_4)=RH^6(\M_4)$ of dimension $32$, we have
\[\dim \ker \xi_A^* = 4, \quad \dim \ker \xi_B^* = 14, \quad \dim \ker \xi_C^* = 3\]
and it holds that
\[(\ker \xi_A^*) \cap (\ker \xi_B^*) \cap (\ker \xi_C^*) = \{0\}.\]
Thus the pullbacks of $[\B_4]$ via $\xi_A, \xi_B, \xi_C$ uniquely determine $[\B_4]$. Furthermore, looking at the dimensions of the kernels of $\xi_A^*, \xi_B^*, \xi_C^*$ we see that the linear system of equations for $[\B_4]$ which they provide is significantly overdetermined. This provides a check for the computations below: an error in the formulas for $\xi_A^*[\B_4], \xi_B^*[\B_4]$ or $\xi_C^*[\B_4]$, which form the inhomogeneous term of the corresponding linear equation, would in general cause this equation to have no solution.

In the following we compute the pullbacks of $[\B_4]$ and point out for various classes that appear in the expression for these pullbacks how to express them as tautological classes.


\begin{lemma}
  The pullback of $[\B_4]$ along $\xi_A$ is given by:
\begin{align*}
 \xi_A^*([\B_4]) = &
-\Bigg[
\begin{tikzpicture}[->,>=bad to,baseline=-3pt,node distance=1.3cm,thick,main node/.style={circle,draw,font=\Large,scale=0.5}]
\node at (1.4,0) [main node] (A) {1};
\node at (0,0) [main node] (B) {2};
\node at (-.2,.5) (Bi) {$\B_{2,0,2}$};
\node at (1.6,.5) (Hi) {$\Hyp_{1,0,2}$};
\draw [-] (B) to [out=40, in=140] (A);
\draw [<-] (B) to [out=-40, in=-140] (A);
\draw [<->] (.7,-.2) to (.7,.2);
\end{tikzpicture}\Bigg]
- \Bigg[
\begin{tikzpicture}[->,>=bad to,baseline=-3pt,node distance=1.3cm,thick,main node/.style={circle,draw,font=\Large,scale=0.5}]
\node at (1.4,0) [main node] (A) {1};
\node at (0,0) [main node] (B) {2};
\node at (-.2,.5) (Bi) {$\B_{2,0,2}$};
\node at (1.6,.5) (Hi) {$\Hyp_{1,0,2}$};
\draw [-] (B) to [out=40, in=140] (A);
\draw [->] (B) to [out=-40, in=-140] (A);
\draw [<->] (.7,-.2) to (.7,.2);
\end{tikzpicture}\Bigg]
 -
\Bigg[
\begin{tikzpicture}[->,>=bad to,baseline=-3pt,node distance=1.3cm,thick,main node/.style={circle,draw,font=\Large,scale=0.5}]
\node at (1.4,0) [main node] (A) {1};
\node at (0,0) [main node] (B) {2};
\node at (-.2,.5) (Hi) {$\Hyp_{2,0,2}$};
\node at (1.6,.5) (Bi) {$\B_{1,0,2}$};
\draw [-] (B) to [out=40, in=140] (A);
\draw [<-] (B) to [out=-40, in=-140] (A);
\draw [<->] (.7,-.2) to (.7,.2);
\end{tikzpicture}  \Bigg]\\
&
-
 \Bigg[
\begin{tikzpicture}[->,>=bad to,baseline=-3pt,node distance=1.3cm,thick,main node/.style={circle,draw,font=\Large,scale=0.5}]
\node at (1.4,0) [main node] (A) {1};
\node at (0,0) [main node] (B) {2};
\node at (-.2,.5) (Hi) {$\Hyp_{2,0,2}$};
\node at (1.6,.5) (Bi) {$\B_{1,0,2}$};
\draw [-] (B) to [out=40, in=140] (A);
\draw [->] (B) to [out=-40, in=-140] (A);
\draw [<->] (.7,-.2) to (.7,.2);
\end{tikzpicture}   \Bigg]
 +
 \Bigg[
\begin{tikzpicture}[->,>=bad to,baseline=-3pt,node distance=1.3cm,thick,main node/.style={circle,draw,font=\Large,scale=0.5}]
\node at (1.4,0) [main node] (B) {1};
\node at (0,0) [main node] (A) {2};
\node at (-.2,.5) (Hi2) {$\Hyp_{2,2,0}$};
\node at (1.6,.5) (Hi1) {$\Hyp_{1,1,0}$};
\draw [-] (A) to [out=40, in=140] (B);
\draw [-] (A) to [out=-40, in=-140] (B);
\draw [->] (.65,-.4) to [out=-140, in =-40,looseness = 10] (.75,-.4);
\draw [->] (.65,.2) to [out=-140, in =-40,looseness = 10] (.75,.2);
\end{tikzpicture} \Bigg]
+ 4
\Bigg[
\begin{tikzpicture}[->,>=bad to,baseline=-3pt,node distance=1.3cm,thick,main node/.style={circle,draw,font=\Large,scale=0.5},shift={(0,-.5)}]
\node at (1.2,0) [main node] (A) {1};
\node at (0,0) [main node] (B) {1};
\node at (0,1.2) [main node] (C) {1};
\node at (-.7,.3) (Hi) {$\Hyp_{1,0,2}$};
\node at (1,1) (Hi) {$\Delta_{\M_{1,2}}$};
\draw [-] (B) to (A);
\draw [-] (B) to (C);
\draw [-] (A) to (C); 
\draw [<->] (.7,.2) to (.2,.7);
\end{tikzpicture} \Bigg]\\
=&  
 -\psi_1 [\B_{2,0,2}]\otimes [\Hyp_{1,0,2}]  -[\B_{2,0,2}]\otimes \psi_1[\Hyp_{1,0,2}] - \psi_1[\Hyp_{2,0,2}]\otimes [\B_{1,0,2}]\\ &
 - [\Hyp_{2,0,2}]\otimes \psi_1[\B_{1,0,2}] + [\Hyp_{2,2,0}]\otimes [\Hyp_{1,0,2}] + 2[\Hyp_{1,0,2}\times \Delta_{\M_{1,2}}]. 
  \end{align*}
 Where $\Delta_{\M_{1,2}}$ is as in Application \ref{app:tautkunneth}. 
 \end{lemma}
 
 \begin{proof}
 This is an application of Theorem \ref{th:main}. Let $\mathfrak{A}$ be the set consisting of all possible distributions of 6 legs on $A$. Graphs in $\mathfrak{A}$ having nonzero intersection with $\phi(\H_{4,\Z_2,(1^6)})$ are of the form:
\begin{center}
\begin{tabular}{c@{\hskip 1cm}c@{\hskip 1cm}c}
$B_1=\begin{tikzpicture}[->,>=bad to,baseline=-3pt,node distance=1.3cm,thick,main node/.style={circle,draw,font=\Large,scale=0.5}]
\node[main node] (A) {2};
\node at (1,0) [main node] (B) {1};
\node at (.5,.5) (e_1) {$e_1$};
\node at (.5,-.5) (e_2) {$e_2$};
\draw [-] (A) to [out=40, in=140] (B);
\draw [-] (A) to [out=-40, in=-140] (B);
\draw [-] (A) to (-.5,.1);
\draw [-] (A) to (-.5,-.1);
\draw [-] (A) to (-.4,.2);
\draw [-] (A) to (-.4,-.2);
\draw [-] (A) to (-.3,.3);
\draw [-] (A) to (-.3,-.3);
\end{tikzpicture}$
&
$B_2=\begin{tikzpicture}[->,>=bad to,baseline=-3pt,node distance=1.3cm,thick,main node/.style={circle,draw,font=\Large,scale=0.5}]
\node[main node] (A) {2};
\node at (1,0) [main node] (B) {1};
\node at (.5,.5) (e_1) {$e_1$};
\node at (.5,-.5) (e_2) {$e_2$};
\draw [-] (A) to [out=40, in=140] (B);
\draw [-] (A) to [out=-40, in=-140] (B);
\draw [-] (A) to (-.5,.1);
\draw [-] (A) to (-.5,-.1);
\draw [-] (A) to (-.4,.2);
\draw [-] (A) to (-.4,-.2);
\draw [-] (B) to (1.5,-.1);
\draw [-] (B) to (1.5,.1);
\end{tikzpicture}$
&
$B_3=\begin{tikzpicture}[->,>=bad to,baseline=-3pt,node distance=1.3cm,thick,main node/.style={circle,draw,font=\Large,scale=0.5}]
\node[main node] (A) {2};
\node at (1,0) [main node] (B) {1};
\draw [-] (A) to [out=40, in=140] (B);
\draw [-] (A) to [out=-40, in=-140] (B);
\draw [-] (A) to (-.5,.1);
\draw [-] (A) to (-.5,-.1);
\draw [-] (B) to (1.4,.2);
\draw [-] (B) to (1.4,-.2);
\draw [-] (B) to (1.5,-.1);
\draw [-] (B) to (1.5,.1);
\end{tikzpicture}$.
   \end{tabular}
  \end{center}
The following admissible $G$-graphs  admit a $B_1$-structure:
\begin{center}
\begin{tabular}{c@{\hskip 1cm}c}
  $(\Gamma_{1,1},\Z_2)=
\begin{tikzpicture}[->,>=bad to,baseline=-3pt,node distance=1.3cm,thick,main node/.style={circle,draw,font=\Large,scale=0.5}]
\node at (1.4,0) [main node] (B) {1};
\node at (0,0) [main node] (A) {2};
\draw [-] (A) to [out=40, in=140] (B);
\draw [-] (A) to [out=-40, in=-140] (B);
\draw [<->] (.7,-.2) to (.7,.2);
\draw [-] (A) to (-.5,.1);
\draw [-] (A) to (-.5,-.1);
\draw [-] (A) to (-.4,.2);
\draw [-] (A) to (-.4,-.2);
\draw [-] (A) to (-.3,.3);
\draw [-] (A) to (-.3,-.3);
\draw [<->] (.7,-.2) to (.7,.2);
\end{tikzpicture} $
&
  $(\Gamma_{1,2},\Z_2)=\begin{tikzpicture}[->,>=bad to,baseline=-3pt,node distance=1.3cm,thick,main node/.style={circle,draw,font=\Large,scale=0.5}]
\node [main node] (A2) {1};
\node at (1,-.5) [main node] (A) {1};
\node at (1,.5) [scale=.3,draw,circle,fill=black] (A3) {};
\node at (2,0) [main node] (B) {1};
\node at (1.5,.5) (e_1) {$\tilde{e}_1$};
\node at (1.6,-.5) (e_2) {$\tilde{e}_2$};
\draw [-] (B) to  (A);
\draw [-] (B) to  (A3);
\draw [-] (A2) to  (A);
\draw [-] (A2) to  (A3);
\draw [-] (A3) to (1.1,1);
\draw [-] (A3) to (.9,1);
\draw [-] (A) to (1.2,.-.9);
\draw [-] (A) to (.8,-.9);
\draw [-] (A) to (.9,-1);
\draw [-] (A) to (1.1,-1);
\draw [<->] (.7,0) to (1.3,0);
\end{tikzpicture}$.
 \end{tabular}
\end{center}
There is up to isomorphism  one $B_1$-structures $f$ on $(\Gamma_{1,1},\Z_2)$. The top Chern class of the excess bundle $E_{f}$ is given by a $\psi$ class on each side. In other words
\[
 c_1(E_{f}) \cap \Bigg[\begin{tikzpicture}[->,>=bad to,baseline=-3pt,node distance=1.3cm,thick,main node/.style={circle,draw,font=\Large,scale=0.5}]
\node at (1.4,0) [main node] (B) {1};
\node at (0,0) [main node] (A) {2};
\node at (-.2,.5) (Hi) {\tiny $\H_{2,\Z_2,(1^6,0)}$};
\node at (1.6,.5) (Bi) {\tiny $\H_{1,\Z_2,(0)}$};
\draw [-] (A) to [out=40, in=140] (B);
\draw [-] (A) to [out=-40, in=-140] (B);
\draw [<->] (.7,-.2) to (.7,.2);
\draw [-] (A) to (-.5,.1);
\draw [-] (A) to (-.5,-.1);
\draw [-] (A) to (-.4,.2);
\draw [-] (A) to (-.4,-.2);
\draw [-] (A) to (-.3,.3);
\draw [-] (A) to (-.3,-.3);
\draw [<->] (.7,-.2) to (.7,.2);
\end{tikzpicture}\Bigg]
=
- \Bigg[
\begin{tikzpicture}[->,>=bad to,baseline=-3pt,node distance=1.3cm,thick,main node/.style={circle,draw,font=\Large,scale=0.5}]
\node at (1.4,0) [main node] (A) {1};
\node at (0,0) [main node] (B) {2};
\node at (-.2,.5) (Hi) {\tiny $\H_{2,\Z_2,(1^6,0)}$};
\node at (1.6,.5) (Bi) {\tiny $\H_{1,\Z_2,(0)}$};
\draw [-] (B) to [out=40, in=140] (A);
\draw [<-] (B) to [out=-40, in=-140] (A);
\draw [-] (B) to (-.5,.1);
\draw [-] (B) to (-.5,-.1);
\draw [-] (B) to (-.4,.2);
\draw [-] (B) to (-.4,-.2);
\draw [-] (B) to (-.3,.3);
\draw [-] (B) to (-.3,-.3);
\draw [<->] (.7,-.2) to (.7,.2);
\end{tikzpicture}  \Bigg]
-
 \Bigg[
\begin{tikzpicture}[->,>=bad to,baseline=-3pt,node distance=1.3cm,thick,main node/.style={circle,draw,font=\Large,scale=0.5}]
\node at (1.4,0) [main node] (A) {1};
\node at (0,0) [main node] (B) {2};
\node at (-.2,.5) (Hi) {\tiny $\H_{2,\Z_2,(1^6,0)}$};
\node at (1.6,.5) (Bi) {\tiny $\H_{1,\Z_2,(0)}$};
\draw [-] (B) to [out=40, in=140] (A);
\draw [->] (B) to [out=-40, in=-140] (A);
\draw [-] (B) to (-.5,.1);
\draw [-] (B) to (-.5,-.1);
\draw [-] (B) to (-.4,.2);
\draw [-] (B) to (-.4,-.2);
\draw [-] (B) to (-.3,.3);
\draw [-] (B) to (-.3,-.3);
\draw [<->] (.7,-.2) to (.7,.2);
\end{tikzpicture}   \Bigg]
\]
  There are 2 isomorphism classes of $B_1$-structures on $(\Gamma_{1,2},\Z_2)$. Indeed a $B_1$-structure $f=(\alpha,\beta,\gamma)$ on $(\Gamma_{1,2},\Z_2)$ can be given either by sending $e_1$ to $\tilde{e}_1$ or to $\tilde{e}_2$ under $\beta$. It is easy to see they are not isomorphic.

  The admissible bielliptic pairs with a $B_2$-structure are given by
  \begin{center}
\begin{tabular}{c@{\hskip 1cm}c}
  $(\Gamma_{2,1},\Z_2)=
\begin{tikzpicture}[->,>=bad to,baseline=-3pt,node distance=1.3cm,thick,main node/.style={circle,draw,font=\Large,scale=0.5}]
\node at (1.4,0) [main node] (B) {1};
\node at (0,0) [main node] (A) {2};
\draw [-] (A) to [out=40, in=140] (B);
\draw [-] (A) to [out=-40, in=-140] (B);
\draw [<->] (.7,-.2) to (.7,.2);
\draw [-] (A) to (-.5,.1);
\draw [-] (A) to (-.5,-.1);
\draw [-] (A) to (-.4,.2);
\draw [-] (A) to (-.4,-.2);
\draw [-] (B) to (1.9,-.1);
\draw [-] (B) to (1.9,.1);
\draw [<->] (.7,-.2) to (.7,.2);
\end{tikzpicture} $
&
  $(\Gamma_{2,2},\Z_2)=\begin{tikzpicture}[->,>=bad to,baseline=-3pt,node distance=1.3cm,thick,main node/.style={circle,draw,font=\Large,scale=0.5},shift={(0,-.5)}]
\node at (0,1) [main node] (A2) {1};
\node [main node] (A) {1};
\node at (1,1) [scale=.3,draw,circle,fill=black] (A3) {};
\node at (1,0) [main node] (B) {1};
\node at (.5,1.2) (e_1) {$\tilde{e}_1$};
\node at (.5,-.2) (e_2) {$\tilde{e}_2$};
\draw [-] (B) to  (A);
\draw [-] (B) to  (A3);
\draw [-] (A2) to  (A);
\draw [-] (A2) to  (A3);
\draw [-] (A3) to (1.1,1.4);
\draw [-] (A3) to (.9,1.4);
\draw [-] (A) to (.2,.-.4);
\draw [-] (A) to (-.2,-.4);
\draw [-] (A) to (.1,-.5);
\draw [-] (A) to (-.1,-.5);
\draw [<->] (.3,.7) to (.7,.3);
\end{tikzpicture}$
 \end{tabular}
\end{center}
Again there is up to isomorphism only one $B_2$-structures $f$ on $(\Gamma_{2,1},\Z_2)$ and the excess bundle is given by 
\[
 E_{f}\cap
\Bigg[\begin{tikzpicture}[->,>=bad to,baseline=-3pt,node distance=1.3cm,thick,main node/.style={circle,draw,font=\Large,scale=0.5}]
\node at (1.4,0) [main node] (B) {1};
\node at (0,0) [main node] (A) {2};
\node at (-.2,.5) (Hi) {\tiny $\H_{2,\Z_2,(1^6,0)}$};
\node at (1.6,.5) (Bi) {\tiny $\H_{1,\Z_2,(0)}$};
\draw [-] (A) to [out=40, in=140] (B);
\draw [-] (A) to [out=-40, in=-140] (B);
\draw [<->] (.7,-.2) to (.7,.2);
\draw [-] (A) to (-.5,.1);
\draw [-] (A) to (-.5,-.1);
\draw [-] (A) to (-.4,.2);
\draw [-] (A) to (-.4,-.2);
\draw [-] (B) to (1.9,-.1);
\draw [-] (B) to (1.9,.1);
\draw [<->] (.7,-.2) to (.7,.2);
\end{tikzpicture} \Bigg]
=
-\Bigg[
\begin{tikzpicture}[->,>=bad to,baseline=-3pt,node distance=1.3cm,thick,main node/.style={circle,draw,font=\Large,scale=0.5}]
\node at (1.4,0) [main node] (A) {1};
\node at (0,0) [main node] (B) {2};
\node at (-.2,.5) (Hi) {\tiny $\H_{2,\Z_2,(1^6,0)}$};
\node at (1.6,.5) (Bi) {\tiny $\H_{1,\Z_2,(0)}$};
\draw [-] (B) to [out=40, in=140] (A);
\draw [<-] (B) to [out=-40, in=-140] (A);
\draw [<->] (.7,-.2) to (.7,.2);
\draw [-] (B) to (-.5,.1);
\draw [-] (B) to (-.5,-.1);
\draw [-] (B) to (-.4,.2);
\draw [-] (B) to (-.4,-.2);
\draw [-] (A) to (1.9,-.1);
\draw [-] (A) to (1.9,.1);
\end{tikzpicture}\Bigg]
- \Bigg[
\begin{tikzpicture}[->,>=bad to,baseline=-3pt,node distance=1.3cm,thick,main node/.style={circle,draw,font=\Large,scale=0.5}]
\node at (1.4,0) [main node] (A) {1};
\node at (0,0) [main node] (B) {2};
\node at (-.2,.5) (Hi) {\tiny $\H_{2,\Z_2,(1^6,0)}$};
\node at (1.6,.5) (Bi) {\tiny $\H_{1,\Z_2,(0)}$};
\draw [-] (B) to [out=40, in=140] (A);
\draw [->] (B) to [out=-40, in=-140] (A);
\draw [<->] (.7,-.2) to (.7,.2);
\draw [-] (B) to (-.5,.1);
\draw [-] (B) to (-.5,-.1);
\draw [-] (B) to (-.4,.2);
\draw [-] (B) to (-.4,-.2);
\draw [-] (A) to (1.9,-.1);
\draw [-] (A) to (1.9,.1);
\end{tikzpicture}\Bigg]
\]
There are 2 isomorphism classes of $B_2$-structures on $(\Gamma_{2,2},\Z_2)$. We can send the edge $e_1$ to either $\tilde{e}_1$ or $\tilde{e}_2$ and this completely determines the $B_2$-structure.

There is only one admissible pair $(\Gamma_{3,1},\Z_2)$  with a $B_3$-structure. It is given by
\[
(\Gamma_3,\Z_2)=
 \begin{tikzpicture}[->,>=bad to,baseline=-3pt,node distance=1.3cm,thick,main node/.style={circle,draw,font=\Large,scale=0.5}]
\node at (1.4,0) [main node] (B) {1};
\node at (0,0) [main node] (A) {2};
\draw [-] (A) to [out=40, in=140] (B);
\draw [-] (A) to [out=-40, in=-140] (B);
\draw [-] (A) to (-.5,.1);
\draw [-] (A) to (-.5,-.1);
\draw [-] (B) to (1.8,.2);
\draw [-] (B) to (1.8,-.2);
\draw [-] (B) to (1.9,-.1);
\draw [-] (B) to (1.9,.1);
\draw [->] (.65,-.4) to [out=-140, in =-40,looseness = 10] (.75,-.4);
\draw [->] (.65,.2) to [out=-140, in =-40,looseness = 10] (.75,.2);
\end{tikzpicture}
\]
There is only one isomorphism class of $B_3$-structures and the top Chern class of the excess bundle is trivial. 

We have determined the sum of the classes in the expression of Theorem \ref{th:main}. Pushing all of these classes forward through the morphisms $\pi_B\circ \phi_{f}$  we obtain the desired expression.
 \end{proof}

 \begin{remark}
  The classes $[\B_{1,0,2}]$ and $[\B_{2,0,2}]$ are relative easy to compute in terms of decorated stratum classes (see \cite[Remark 3.4.2, Proposition 3.2.10]{thesisjason}). We have $[\Hyp_{1,0,2}]=[\M_{1,0,2}]$. The class $[\Hyp_{2,0,2}]$ is computed in \cite[Lemma 6]{Belorousski1984}. The class $[\Hyp_{2,2,0}]$ is given as $[\overline{\mathcalorig{DR}}_2(2)]$ in \cite[Theorem 0.1]{Tarasca2014}. The class $[\Hyp_{1,0,2}]$ is given as $[\bar{A}_2]$ in \cite[Theorem 3.33]{Pagani2008}. The class of the diagonal $[\Delta_{\M_{1,2}}]$ can easily be determined by Application \ref{app:tautkunneth}. We have thus completely determined the class $\xi^*_A([\B_4])$ in terms of decorated stratum classes.
 \end{remark}	

The pullback of $[\B_4]$ under $\xi_B$ and $\xi_C$ now works in a similar way.
 
 \begin{lemma}
  We have 
  \[
  \xi_B^*([\B_4]) = [\B_{2,1,0} \times \Hyp_{2,1,0}] +  [\Hyp_{2,1,0} \times \B_{2,1,0}] \in H^6(\M_{2,1}\times \M_{2,1}).
  \]
 \end{lemma}

 \begin{proof}
  This is Example \ref{ex:B21H21}.
 \end{proof}

\begin{lemma}
  In $H^{6}(\M_{3,1} \times \M_{1,1})$ we have
\begin{align*}
\xi_C^*([\B_4])=&
\Bigg[
\begin{tikzpicture}[->,>=bad to,baseline=-3pt,node distance=1.3cm,thick,main node/.style={circle,draw,font=\Large,scale=0.5}]
\node at (1.4,0) [main node] (A) {1};
\node at (0,0) [main node] (B) {3};
\node at (-.2,.5) (Hi2) {$\B_{3,1,0}$};
\node at (1.6,.5) (Hi1) {$\Hyp_{1,1,0}$};
\draw [-] (B) to (A);
\end{tikzpicture} \Bigg]
+ 2
\Bigg[
\begin{tikzpicture}[->,>=bad to,baseline=-3pt,node distance=1.3cm,thick,main node/.style={circle,draw,font=\Large,scale=0.5},shift={(0,-.5)}]
\node at (1.2,0) [main node] (A) {1};
\node at (0,0) [main node] (B) {2};
\node at (0,1.2) [main node] (C) {1};
\node at (-.7,.3) (Hi) {$\Hyp_{2,0,2}$};
\node at (1,1) (Hi) {$\Delta_{\M_{1,1}}$};
\draw [-] (B) to (A);
\draw [-] (B) to (C);
\draw [<->] (.7,.2) to (.2,.7);
\end{tikzpicture} \Bigg]\\
=& [\B_{3,1,0}\times \Hyp_{1,1,0}] + 2 [\Hyp_{2,0,2}\times \Delta_{\M_{1,1}}]
\end{align*}
\end{lemma}


The class of  $\B_{3,1,0} \subset \M_{3,1}$ has not been computed before but it can be determined by further pullbacks to the boundary.

 \begin{proof}[Proof of Theorem \ref{th:B4}]
  We have determined the three linear maps $\xi^*_A$, $\xi^*_B$ and $\xi_C^*$ in terms of a basis of decorated stratum classes for $H^6(\M_4)$. 
 We have also determined the the pullbacks $\xi^*_A[\B_4]$, $\xi^*_B[\B_4]$ and $\xi^*_C[\B_4]$ in terms of decorated stratum classes. Since the kernels of $\xi_A^*, \xi_B^*, \xi_C^*$ have trivial intersection, this completely determines the class $[\B_4]$. 
 \end{proof}

\begin{remark}
 The class $[\B_4]$ has first been computed in \cite{thesisjason} using the combined information of $\xi^*_A[\B_4]$, $\xi^*_B[\B_4]$ and the pull-push $\delta_*\phi_{(1^{(10)})}[\B_4]$ through the hyperelliptic curves
 \[
 \begin{tikzcd}
 \H_{4,\Z_2,(1^{10})}\arrow[r,"\phi_{(1^{10})}"] \arrow[d,"\delta"]  & \M_4\\
 \M_{0,10}
 \end{tikzcd}.
 \]
  This last operation is not covered by our previous results but will be described in Section \ref{sec:inthyp}. It should be seen as proof of concept that the intersection of two loci of admissible covers can again be described in terms of tautological classes and further spaces of admissible covers. It turns out that the result for $\delta_*\phi_{(1^{10})}^* [\B_4]$ is consistent with the formula in Theorem \ref{th:B4}, providing a further check of the previous computations.
\end{remark}

\subsection{Further Examples and Applications} \label{Sect:furtherexa}
Using our implementation of the methods in the previous sections, we can compute many more examples of admissible cover cycles.




Apart from the hyperelliptic and bielliptic cycles listed in Figures \ref{fig:hyperelliptcycles} and \ref{fig:bielliptcycles} in the introduction, we can compute a few other loci of admissible covers.
\begin{itemize}
    \item $[\H_{3,\mathbb{Z}/2\mathbb{Z},()}] \in H^6(\M_3)$, the cycle of genus $3$ curves $C$ admitting an unramified double cover of a genus $2$ curve,
    \item $[\H_{3,\mathbb{Z}/3\mathbb{Z},(0)}] \in H^4(\M_{1,3})$, the cycle of elliptic curves $(E,p_0,p_1,p_2)$ such that there exists an unramified triple\footnote{Such a cover must necessarily be a group homomorphism. It follows that it is a cyclic cover and in fact for the elliptic curve $(E,p_0)$ the point $p_1$ is a $3$-torsion point and $p_2=2 \cdot p_1$.} cover $E \to E'$ to an elliptic curve $E'$ such that $p_0, p_1, p_2$ form a fibre, 
    \item $\pi_* [\H_{3,\mathbb{Z}/3\mathbb{Z},(1,2)}] \in H^8(\M_3)$, the cycle of genus $3$ curves $C$ admitting a cyclic triple cover of a genus $1$ curve, with full ramification at $2$ points.
\end{itemize}
In particular, our methods allow us to go beyond the case of double covers.

One of the main bottlenecks in the above computations is the enumeration of all tautological relations in various spaces $RH^{2k}(\M_{g,n})$, which is done using the program \cite{Pixtprogram}. Since cycles of admissible covers tend to have high codimension $k$, it takes a long time to enumerate all decorated stratum classes of this degree and to compute the coefficients of the relations between them. Using a more refined approach to express tautological classes in a basis of $RH^{2k}(\M_{g,n})$, e.g. by computing intersection numbers of cycles in opposite degree, could potentially allow us to extend the above list of examples.

\begin{remark}
 In the paper \cite{hurwitzhodge}, intersection numbers of Hodge and $\psi$-classes are computed on spaces $\M_{g',\gamma}(BG)$ of stable maps to the classifying stack $BG$ of a finite (abelian) group $G$. As described in Remark \ref{Rmk:relntoLit}, the space $\M_{g',\gamma}(BG)$ is related to the space $\H_{g,G,\gamma}$, where the target of the admissible cover has genus $g'$. Indeed, there is a map $\H_{g,G,\gamma} \to \M_{g',\gamma}(BG)$. Indeed, given a cover $G \curvearrowright C \to D$ over a scheme $S$, the natural map $C \to S$ is $G$-equivariant with respect to the trivial action of $G$ on $S$. Taking the quotient by $G$ on both sides gives us a map $[C/G] \to [S / G] = S \times BG \xrightarrow{\pi_{BG}} BG$, which gives an family in $\M_{g',\gamma}(BG)$ over $S$. 
 
 In the case $d=2, g'=1$ and for monodromy data $\gamma$ of the form $\gamma=(1,1, \ldots, 1) = (1^{2n})$, this map is an isomorphism. Given $n \geq 1$ both spaces parametrize bielliptic curves of genus $n+1$ ramified over $2n$ points. 
 
 The results of \cite{hurwitzhodge} have an explicit specialization for these spaces. To state it, let $\lambda_i \in H^{2i}(\M_{n+1})$ be the $i$th Chern class of the Hodge bundle and let 
 \[\phi : \H_{n+1, \mathbb{Z}/2\mathbb{Z}, (1^{2n})} \to \M_{n+1}\]
 be the map remembering the bielliptic curve, without markings. Consider the generating series of integrals
 \[F(u) = \sum_{n=1}^\infty \frac{u^{2n-1}}{(2n-1)!} \int_{\H_{n+1, \mathbb{Z}/2\mathbb{Z}, (1^{2n})}} \phi^* \left( \lambda_{n+1} \lambda_{n-1} \right).\]
 Then a computation using \cite[Theorem 1]{hurwitzhodge} shows that
 \begin{equation} \label{eqn:hurwitzhodge}
     F(u) = \frac{i}{24} \left(\frac{1-e^{iu}}{1+e^{iu}} \right) = \frac{1}{48}u + \frac{1}{576}u^3+\frac{1}{5760} u^5 + \ldots.
 \end{equation}
 Now, up to a factor, the pushforward $\phi_*[\H_{n+1, \mathbb{Z}/2\mathbb{Z}, (1^{2n})}]$ is nothing but the class $[\B_{n+1}]$ of the bielliptic locus and by the projection formula we have
 \[\int_{\H_{n+1, \mathbb{Z}/2\mathbb{Z}, (1^{2n})}} \phi^* \left( \lambda_{n+1} \lambda_{n-1} \right) = \int_{\M_{n+1}} \phi_*[\H_{n+1, \mathbb{Z}/2\mathbb{Z}, (1^{2n})}] \cdot \lambda_{n+1} \lambda_{n-1}\]
 For $n=1,2,3$ we can compute the class of the bielliptic locus in terms of decorated stratum classes and thus the corresponding intersection numbers with $\lambda_{n+1} \lambda_{n-1}$. The results agree with the coefficients in $F(u)$ predicted by the formula (\ref{eqn:hurwitzhodge}) above. This provides a further nontrivial check for our computations.
\end{remark}

\subsection{Intersecting with the Hyperelliptic Locus}\label{sec:inthyp}
 
 In \cite{Faber2012} the authors calculate the class $[\B_3]$ of the bielliptic locus in genus 3 by pulling it back to the hyperelliptic locus (and evaluating it on a small number of test surfaces). In other words they compute the pull-push of $[\B_3]$ along the diagram 
 \[
 \begin{tikzcd}
 \H_{3,\Z_2,(1^8)}\arrow[r,"\phi_{(1^8)}"]\arrow[d,"\delta"] & \M_3\\
 \M_{0,8}
 \end{tikzcd}.
 \]
  The space $\M_{0,8}$ is relatively easy to understand; the inverse image of $\B_3$ along $\phi_{(1^8)}$ can be determined by set theoretic arguments\footnote{In the original published version of the paper the authors make a mistake in this part of the argument (see \cite[Remark 3.4.2]{thesisjason} for details).}. This allows the authors to compute the class $I:=\delta_* \phi^*_{(1^8)}([\B_3])$ up to a constant.  The authors then determine the composition $\delta_* \circ \phi_{(1^8)}^*\colon A^2(\M_3)\rightarrow A^2(\H_{3,\Z_2,(1^8)})\rightarrow A^2(\M_{0,8})$, in terms of a basis for $A^2(\M_3)$ consisting of products of divisors and $\kappa_2$, by computing  the pull-push $\delta_*\phi_{(1^8)}^*(D)$ for all boundary divisors $D\in A^1(\M_3)$ and multiplying out on both sides, and by directly computing $\delta_*\phi_{(1^8)}^*(\kappa_2)$. 
 Since $\dim A^2(\M_3)=7$, $\dim A^2( \M_{0,8})=6$ and $\delta_*\phi^*$ is surjective this determines 5 out of the 7 coefficients of the class $[\B_3]$, where we miss 1 coefficient because we only determine the pullback of $[\B_3]$ to the hyperelliptic locus up to a scalar.
 
  Using Theorem \ref{th:mainpush} this approach becomes a bit more systematic. Indeed we have seen that the composition 
 \[\delta_* \phi^*_{\tilde{\xi}}:R^k(\M_{g,n}) \rightarrow R^k(\H_{g,G,\xi}) \rightarrow R^k(\M_{g',b})\]
 can be computed algorithmically. If we want to compute the class $[\phi_{\tilde{\xi}'}\H_{g,G',\xi'}]$ of another space of admissible covers then we still have to determine the intersection
 \begin{equation}\label{eq:pulltohur}
 \phi_{\tilde{\xi}}^*[\phi_{\tilde{\xi}'}\H_{g,G',\xi'}].
 \end{equation}
 Currently such pullbacks have not been worked out in full generality, but the expectation is that the fibered diagram
 \[
 \begin{tikzcd}
 F \arrow[r]\arrow[d]& \H_{g,G',\xi'}\arrow[d,"\phi_{\tilde{\xi}'}"]\\
 \H_{g,G,\xi} \arrow[r,"\phi_{\tilde{\xi}}"] & \M_{g,n}
 \end{tikzcd}
 \]
 and its excess bundle can be analysed in a way similar to that of Proposition \ref{prop:fibreprod} and \ref{prop:excess}. At the moment we just have set  theoretic arguments to analyse the pullback \eqref{eq:pulltohur} in particular cases. 
 
 \bigskip
 
 \textbf{Pulling Back the Bielliptic Locus in Genus 4:} As an example we will now sketch the computation of the pullback of $[\B_4]$  to the stack $\H_{4,\Z_2,(1^{10})}$ parametrizing hyperelliptic curves of genus  4. 
 \[
 \begin{tikzcd}
 \H_{4,\Z_2,(1^{10})}\arrow[r,"\phi_{(1^{10})}"]\arrow[d,"\delta"] & \M_4\\
 \M_{0,10}
 \end{tikzcd}.
 \]
 We start by giving the set theoretic intersection of the bielliptic and hyperelliptic locus.


 \begin{proposition}\label{prop:theloci}
  The inverse image of $\B_4$ under the map $\phi_{(1^{10})}\colon \H_{4,\Z_2,(1^{10})}\rightarrow \M_4$ is given by two strata:
  \begin{enumerate}
   \item the locus $\mathcalorig{A}\subset \H_{4,\Z_2,(1^{10})}$ of admissible hyperelliptic covers $C\rightarrow D$ where $C$ is a curve with two irreducible components $C_1$ and $C_2$ of genus 1 and 2 respectively and two nodes between them; and where $C_2$ admits a bielliptic involution switching the nodes,
   
   \item the locus $\mathcalorig{B}\subset \H_{4,\Z_2,(1^{10})}$ of admissible hyperelliptic covers $C\rightarrow D$ where $C$ is a curve with two irreducible components $C_1$ and $C_2$ of genus 1 and 2 and two nodes between them; and where $C_1$ admits a bielliptic involution switching the nodes.
  \end{enumerate}
    \begin{figure}[H]
  \centering
    \includegraphics[width=.8\textwidth]{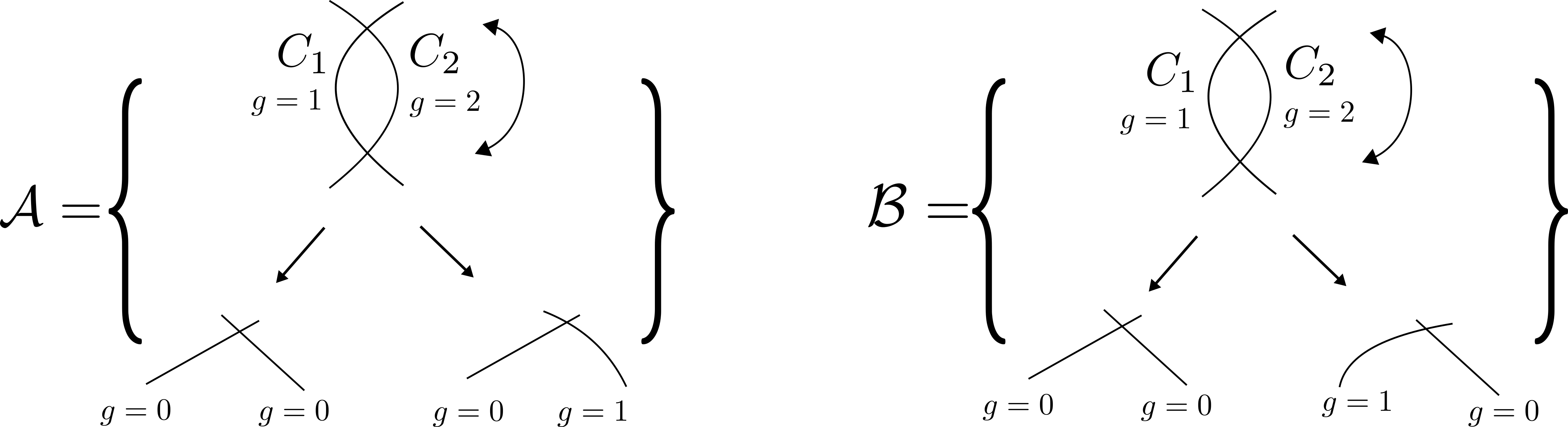}
    \label{fig:AB}
\end{figure}
 \end{proposition}

 \begin{proof} 
  Since $\mathcalorig{A}$ and $\mathcalorig{B}$ are subsets of the stack $\H_{4,\Z_2,(1^{10})}$, one sees easily that the hyperelliptic involution exchanges in both cases the nodes connecting $C_1$ and $C_2$. On the other hand, we also have a bielliptic involution in both cases: for $\mathcalorig{A}$ it is given by the required bielliptic involution on $C_2$ and by the hyperelliptic involution of $C_1$. On the other hand, for $\mathcalorig{B}$ the bielliptic involution is given by the required bielliptic involution on $C_1$ and by the hyperelliptic involution on $C_2$.
  
  That these are all loci follows from a combinatorial exhaustion in the spirit of \cite[Proposition 3]{Faber2012} (in particular from the Castelnuovo-Severi inequality it follows that there are no smooth curves of genus $\geq 4$ with both a hyperellptic and a bielliptic involution).
 \end{proof}

In the following proposition, we show that the loci $\mathcalorig{A}$ and $\mathcalorig{B}$ are again parametrized by spaces of admissible $G$-covers. Since the hyperelliptic and bielliptic involution commute, they naturally induce an action of the group $G=\Z_2\times \Z_2$, which appears in the formulas.

\begin{proposition}\label{prop:thelocihur}
Let $\phi_{\Z_2,((0,1)^2)}$ denote the composition
 \[
   \phi_{\Z_2,((0,1)^2)}\colon \H_{2,\Z_2\times \Z_2,((1,0)^3,(0,1),(1,1))}\to \H_{2,\Z_2,(1^6,0,0)}\to  \H_{2,\Z_2,(1^6,0)},
 \]
 where the first map forgets the second $\Z_2$ action and the second forgets the pair of markings that previously had stabilizer $(1,0)$. Let $\phi_{\Z_2}$ denote the map 
 \[
     \phi_{\Z_2}\colon \H_{1,\Z_2\times \Z_2,((1,0)^2,(1,1)^2)}\to \H_{2,\Z_2,(1^4,0^2)} \to \H_{2,\Z_2,(1^4,0)}.
 \]
 forgetting the second $\Z_2$ action and one of the pair of points that previously had stabilizer $(1,1)$.
Let $\xi\colon \H_{2,\Z_2,(1^6,0)}\times \H_{1,\Z_2,(1^4,0)}\to \H_{4,\Z_2,(1^{10})}$ be the map gluing the admissible hyperelliptic covers together in the points not fixed by the involution. We have 
\begin{align*}
    \mathcalorig{A} &= \xi \big( \im \phi_{\Z_2,((0,1)^2)}, \H_{1,\Z_2,(1^4,0)}\big)\\
        \mathcalorig{B} &= \xi \big( \H_{2,\Z_2,(1^6,0)}, \im\phi_{\Z_2}\big).
\end{align*}
\end{proposition}

\begin{proof}
We adopt the notation of Proposition \ref{prop:theloci}. It is obvious that both $\mathcalorig{A}$ and $\mathcalorig{B}$ are contained in the image of $\xi$. 

We start by considering $\mathcal{A}$. There are no conditions on the curve $C_1$. Let $f\colon C_2\to D_2$ be the hyperelliptic map on $C_2$ and let $\sigma$ be the corresponding hyperelliptic involution.
The curve $C_2$ also admits a bielliptic involution $\tau$, the involutions $\sigma$ and $\tau$ commute and the involution $\sigma\circ \tau$ is a bielliptic involution. The involution $\tau$ induces an involution on $D_2$. Modding out by this involution gives a map $D_2\rightarrow R$ where $R$ is a stable genus 0 curve. By the Riemann-Hurwitz formula, the ramification divisor of the composition $g\colon C_2\rightarrow D_2 \rightarrow R$ has degree 10. The only ramification points of $g$ are the fixed points of $\sigma$, $\tau$ and $\sigma\circ \tau$. Therefore the fixed points of $\sigma$ and of $\tau$ and of $\sigma \circ \tau$ are distinct. 

The situation is be summarized by the following diagram. Where the upper middle map is the hyperelliptic map. The lower middle map is the map $D_2\rightarrow R$ and the upper maps on the right and left are the two bielliptic maps.
\begin{center}
\begin{tikzpicture}
\draw (-1.5,0) to (1.5,0);
\draw (-5,0) to (-2,0);
\draw (2,0) to (5,0);
\draw (-2.5,1.5) to (2.5,1.5);
\draw (-2.5,-1.5) to (2.5,-1.5);
\foreach \x in {0,.25} {
\node[scale=.2,draw,star,star points=6,star point ratio=2.5,fill=black] at (-2+\x,1.5) {}; \node[scale=.3,draw,diamond,fill=black] at (-1+\x,1.5) {};
\node[scale=.2,draw,star,star points=6,star point ratio=2.5,fill=black] at (-4.5+\x,0) {};
\node[scale=.3,draw,diamond,fill=black] at (3+\x,0) {};
};
\foreach \x in {0,.25,.5,.75,1,1.25} {
\node[scale=.3,draw,circle,fill=black] at (.75+\x,1.5) {};
\node[scale=.3,draw,circle,fill=black] at (-.25+\x,0) {};
};
\foreach \x in {0,.25,.5} {
\node[scale=.3,draw,circle,fill=black] at (-3.25+\x,0) {};
\node[scale=.3,draw,circle,fill=black] at (3.75+\x,0) {};
\node[scale=.3,draw,circle,fill=black] at (1+\x,-1.5) {};
};
\node[scale=.2,draw,star,star points=6,star point ratio=2.5,fill=black] at (-1.25,0) {}; \node[scale=.3,draw,diamond,fill=black] at (-0.75,0) {};
\node[scale=.3,draw,diamond,fill=black] at (-3.75,0) {};
\node[scale=.2,draw,star,star points=6,star point ratio=2.5,fill=black] at (2.5,0) {};
\node[scale=.2,draw,star,star points=6,star point ratio=2.5,fill=black] at (-2,-1.5) {}; \node[scale=.3,draw,diamond,fill=black] at (-0.75,-1.5) {};
\draw [->] (-2,1.25) to (-3.5,.25);
\draw [->] (2,1.25) to (3.5,.25);
\draw [->] (0,1.25) to (0,.25);
\draw [<-] (-2,-1.25) to (-3.5,-.25);
\draw [<-] (2,-1.25) to (3.5,-.25);
\draw [<-] (0,-1.25) to (0,-.25);
\node at (.2,.75) {$f$};
\node [right] at (2.5,1.5) {$g=2$};
\node [left] at (-5,0) {$g=1$};
\node [right] at (5,0) {$g=1$};
\node [below] at (1,0) {$g=0$};
\node [right] at (2.5,-1.5) {$g=0$};
\end{tikzpicture}
\end{center}
Identifying $\sigma$ with $(1,0)$ and $\tau$ with $(0,1)$ it follows that cover $C_2$ lies in the image of $\phi_{\Z_2,((0,1)^2)}$.

The case of $\mathcalorig{B}$ is analogous.
\end{proof}

 Let us denote by 
 \[
\Hyp^s_{0,1,2m}:= \bigcup_{\sigma\in \S_{2m}} \sigma (\Hyp_{0,1,2m}) \subset \M_{0,1+2m}
 \]
 the image under the action of the symmetric group $\S_{2m}$ acting on the $2m$-markings permuted by the hyperellipitc involution.  We can draw a conclusion from Proposition \ref{prop:thelocihur} in terms of loci on the target space as follows

 \begin{proposition}\label{prop:identifyAB}
  Let $\tilde{\xi}\colon \M_{0,7}\times \M_{0,5} \rightarrow \M_{0,10}$ be the gluing morphism gluing the curves together in the last points. Let $\delta\colon \H_{4,\Z_2,(1^{10})} \rightarrow \M_{0,10}$ be the target map. We have 
  \begin{align*}
   \delta(\mathcalorig{A}) &= \tilde{\xi}(\Hyp^s_{0,1,7} \times \M_{0,5})\\
   \delta(\mathcalorig{B}) &= \tilde{\xi}(\M_{0,7} \times \Hyp^s_{0,1,4}).
  \end{align*}
 \end{proposition}
 
 \begin{proof}
This follows immediatly from applying the target map to the loci of Proposition \ref{prop:thelocihur}.
In particular note that $\delta$ induces a quotient map on the loci of Proposition \ref{prop:thelocihur} given by the $\Z_2$ action associated to the hyperelliptic structure.  
%
%
  \end{proof}

  Consider now the commutative diagram (which we know to be Cartesian on the level of sets)
  \begin{equation}\label{prgrph:normAB}
   \begin{tikzcd}
    \mathcalorig{A} \cup \mathcalorig{B} \arrow[r,"g"] \arrow[d,"i"] & \B_4\arrow[d] \\
    \H_{4,\Z_2,(1^{10})} \arrow[r,"\phi_{(1^{10})}"] & \M_4
   \end{tikzcd}.
  \end{equation}
We have $\codim_{\M_4} \B_4=3$. Since $\codim_{\M_{0,10}} \tilde{\xi}(\Hyp^s_{0,1,6}\times \M_{0,5}) = 3$ the excess bundle  $E= i^*N_{\phi_{(1^{10})}}/N_g$ is trivial on $\mathcalorig{A}$. 
 We have $\codim_{\M_{0,10}} \tilde{\xi}(\M_{0,5}\times \Hyp^s_{0,1,4}) = 2$. Let $\phi_{(1^6)}\colon \H_{2,\Z_2,(1^6,0)}\rightarrow \M_{2,2}$ be the source map and let $p\colon \mathcalorig{B}\rightarrow \H_{2,\Z_2,(1^6,0)}$ be the projection onto the second factor. The excess bundle restricted to $\mathcalorig{B}$ is given by 
\[
 p^*N_{\phi_{(1^6)}}.
\]
From the excess intersection formula we get the following.

 \begin{proposition}
  Let $a,b\in \Q_{>0}$ be the multiplicity of the loci $\mathcalorig{A}, \mathcalorig{B}$ in the (stack-theoretic) fibre product \eqref{prgrph:normAB}.
  Let $\tilde{\delta}\colon \H_{2,\Z_2,(1^6,0)}\rightarrow \M_{0,7}$ be the target map. With the above notation and that  of Proposition \ref{prop:identifyAB},
we have, 
  \begin{equation}\label{eq:B4pullback}
   \delta_* \phi_{(1^{10})}^* ([\B_4])= \tilde{\xi}_*(a[\Hyp^s_{0,1,6}]\otimes [\M_{0,5}] + b(\tilde{\delta}_*c_1(N_{\phi_{(1^6)}})\cap [\M_{0,7}])\otimes [\Hyp^s_{0,1,4}]).
  \end{equation}
 \end{proposition}

The classes $[\Hyp^s_{0,1,4}]$ and $[\Hyp^s_{0,1,6}]$ can be computed in terms of decorated stratum classes as outlined in the rest of this paper. On the other hand we have
\[\tilde{\delta}_*c_1(N_{\phi_{(1^6)}})= \tilde{\delta}_*\phi_{(1^6)}^*([\Hyp_{2,0,2}])\]
the class $[\Hyp_{2,0,2}]$ was first computed in terms of decorated stratum classes in \cite[Lemma 6]{Belorousski1984} and is again easy to compute with the methods of this paper. The pull-push map 
\[\tilde{\delta}_*\phi_{(1^6)}^*\colon A^1(\M_{2,2})\to A^1(\M_{0,7})\]
can be computed in terms of bases of decorated stratum classes using Theorem \ref{th:mainpush}. Together this gives an expression in terms of decorated stratum classes for $\tilde{\delta}_*c_1(N_{\phi_{(1^6)}})$.

Details of these computations can be found in \cite{thesisjason}. We end up with an expression in terms of decorated stratum classes for \eqref{eq:B4pullback}.

The pull-push 
\[
\delta_*\phi^*_{(1^{10})}\colon A^3(\M_4)\to A^3(\M_{0,10}) 
\]
can also be determined in terms of bases of decorated stratum classes. Now $\dim A^3(\M_4)=32$, $\dim  A^3(\M_{0,10})=21$ and $\delta_* \phi^*$ turns out to be surjective. Above, we determined the image of $[\B_4]$ under $\delta_* \phi^*$ up to 2 coefficients $a,b$. This gives $19=21-2$ out of the $32$ coefficients neccesary to compute $[\B_4]$. In \cite{thesisjason} this computation was neccessary for the computation of $[\B_4]$.

\bibliography{main.bbl}{}
\bibliographystyle{alpha}

J.~Schmitt, \textsc{Departement Mathematik, ETH Z\"urich, R\"amistrasse 101, 8092 Z\"urich, Switzerland}
\par\nopagebreak
  \textit{E-mail address}: \texttt{johannes.schmitt@math.ethz.ch}
\bigskip

J.~van~Zelm, \textsc{Deparment of Mathematics, Humboldt-Universit\"{a}t zu Berlin, Rudower Chaussee 25, Room 1.415, 12489 Berlin, Germany}
\par\nopagebreak
  \textit{E-mail address}: \texttt{jasonvanzelm@outlook.com}
\end{document}